\documentclass[12pt,reqno,a4paper]{amsart}
\usepackage[latin1]{inputenc}
\usepackage[T1]{fontenc}
\usepackage{amssymb}
\usepackage{bm}
\usepackage[alpine]{ifsym}
\usepackage{xpicture}
\usepackage{calculator}
\usepackage{graphicx}

\usepackage[a4paper,top=3.3cm,bottom=3cm,left=3cm,right=3cm,bindingoffset=5mm]{geometry}

\usepackage{color}

\usepackage[colorlinks=true]{hyperref}
\newtheorem{theorem}{Theorem}[section]
\newtheorem{lemma}[theorem]{Lemma}
\newtheorem{corollary}[theorem]{Corollary}
\newtheorem{proposition}[theorem]{Proposition}

\theoremstyle{definition}
\newtheorem{example}[theorem]{Example}
\newtheorem{remark}[theorem]{Remark}
\newtheorem{definition}[theorem]{Definition}

\newtheorem*{A}{Theorem A}
\newtheorem*{B}{Theorem B}

\numberwithin{equation}{section}

\newcommand{\betabarra}{\bar{\beta}}

\begin{document}
\title[Newton-Okounkov bodies]{Newton-Okounkov bodies of exceptional curve valuations}

\author[C. Galindo]{Carlos Galindo}

\address{Universitat Jaume I, Campus de Riu Sec, Departamento de Matem\'aticas \& Institut Universitari de Matem\`atiques i Aplicacions de Castell\'o, 12071
Caste\-ll\'on de la Plana, Spain.}\email{galindo@uji.es}  \email{moyano@uji.es}

\author[F. Monserrat]{Francisco Monserrat}
\address{Instituto Universitario de
Matem\'atica Pura y Aplicada, Universidad Polit\'ecnica de Valencia,
Camino de Vera s/n, 46022 Valencia (Spain).}
\email{framonde@mat.upv.es}

\author[J.J. Moyano-Fern\'andez]{Julio Jos\'e Moyano-Fern\'andez}


%
\author[M. Nickel]{Matthias Nickel}
\address{Goethe-Universit\"at Frankfurt, FB Informatik und Mathematik, 60054 Frankfurt am Main, Germany.}
\email{nickel@math.uni-frankfurt.de}
\thanks{The first three authors were partially supported by the Spanish Government Ministerio de Econom\'ia, Industria y Competitividad (MINECO), grants  MTM2015-65764-C3-2-P and MTM2016-81735-REDT, as well as by Universitat Jaume I, grant P1-1B2015-02.}

\begin{abstract}
We prove that the Newton-Okounkov body of the flag $E_{\bullet}:= \left\{ X=X_r \supset E_r \supset \{q\} \right\}$, defined by the surface $X$ and the exceptional divisor $E_r$ given by any divisorial valuation of the complex projective plane $\mathbb{P}^2$, with respect to the pull-back of the line-bundle $\mathcal{O}_{\mathbb{P}^2} (1)$ is either a triangle or a quadrilateral, characterizing when it is a triangle or a quadrilateral. We also describe the vertices of that figure.  Finally, we introduce a large family of flags for which we determine explicitly their Newton-Okounkov bodies which turn out to be triangular.
\end{abstract}

\maketitle

\section{Introduction}
Newton polygons are probably the first example of the usage of polyhedral objects to study algebraic varieties, curves in this case.  Newton-Okounkov bodies are more complicated objects that pursue the same idea; they  associate a convex body to a big divisor on a smooth irreducible normal projective variety $X$, with respect to a specific flag of subvarieties of $X$, via the corresponding valuation on the function field of $X$. They were introduced by Okounkov \cite{ok96, ok97, ok03} and independently developed by Lazarsfeld and Musta{\c t}{\u a} \cite{LM}, on the one hand, and Kaveh and Khovanskii \cite{KK}, on the other. One of the main advantages of these bodies is that their convex structure helps in the study of the asymptotic behavior of the linear systems given by the divisor and the valuation, the structure of the Mori cone of $X$, and positivity properties of divisors on $X$ \cite{BKMS, KL1, KL2, KL3, KLM}. The computation of Newton-Okounkov bodies is a very hard task and, sometimes, their behavior is unexpected \cite{KLM}. In this regard, little is known about them, even when $X$ is a surface. For these varieties, we know that they are polygons with rational slopes and can be computed from Zariski decompositions of divisors \cite{LM}.

Let $p \in \mathbb{P}^2 := \mathbb{P}^2_\mathbb{C}$ be a point in the complex projective plane and consider a surface $X$ obtained after a sequence of finitely many simple point blow-ups starting with the blow-up of $p$, where simple means that we only blow points in the exceptional divisor created last. We devote this paper to explicitly describe the  Newton-Okounkov body of a flag $E_{\bullet}:= \left\{ X=X_r \supset E_r \supset \{q\} \right\}$, with respect to the pull-back $H$ on $X$ of the line bundle $\mathcal{O}_{\mathbb{P}^2} (1)$, where $E_r$ is the last obtained exceptional divisor and $q$ a point on  it. These flags define and are defined by specific representatives $\nu = \nu_{E_{\bullet}}$ of the so-called exceptional curve valuations. This class of valuations corresponds to one of the five types of valuations of the fraction field of the local ring $\mathcal{O}_{\mathbb{P}^2,p}$ appearing in the classification given by Spivakovsky in \cite{spiv}; the name comes from \cite{FJ}. Notice that their rank and rational rank equal 2 and their transcendence degree is $0$. These valuations are usually considered up to equivalence \cite{z-s} and one can attach a Newton-Okounkov body to any representative whose value group is included in $\mathbb{R}^2$.

Consider the flag $E_{\bullet}$ and let $\nu$ be its exceptional curve valuation. The above mentioned Newton-Okounkov body will be denoted by $\Delta_\nu (H) = \Delta_\nu$. The valuation $\nu$ has two components: $\nu = (\upsilon_1, \upsilon_2)$, where $\upsilon_1 := \nu_r$ is the divisorial valuation, with $\mathbb{Z}$ as value group, defined by the divisor $E_r$. We denote by $\{\beta'_j (\nu_r)\}_{j=0}^{g+1}$ the so-called sequence of Puiseux exponents of the valuation $\nu_r$ (see Subsection \ref{Puiseux}) which determines, and it is determined by, the dual graph of the valuation $\nu_r$; then we say that $\nu_r$ has $g$ Puiseux pairs. The pair $(\nu_r,H)$ comes with an interesting value, denoted by $\hat{\mu} (\nu_r)$, which is defined as $\lim_{m \rightarrow \infty} m^{-1} a(mL)$, where $a(mL)$ is the last value of the vanishing sequence of $H^0(mL)$ along $\nu_r$ for a line $L$ of $\mathbb{P}^2$. This value, introduced in \cite{BKMS}, is an analogue of the Seshadri constant for the valuation $\nu_r$. We recall that the Seshadri constants were considered by Demailly in \cite{5Nach} in the study of the Fujita conjecture. In our case
\[
\hat{\mu}(\nu):=\lim_{d\rightarrow \infty} \frac{\max\{\nu(f)\mid f\in \mathbb{C}[u,v] \mbox{ and } \deg(f)\leq d\}}{d},
\]
where $(u,v)$ are coordinates in an affine chart containing $p$ as the origin.

Returning to Newton-Okounkov bodies, in \cite{cil} it is considered a sort of normalized Newton-Okounkov body for valuations $\nu = \nu_{E_{\bullet}}$ corresponding to valuations $\nu_r$ with only one Puiseux pair and whose point $q$ is not free. There, the authors prove that these Newton-Okounkov bodies are triangles or quadrilaterals whose vertices depend on the Puiseux exponents of $\nu_r$ and the value $\hat{\mu} (\nu_r)$. As we will show (see Remark \ref{315}), the Newton-Okounkov bodies considered in \cite{cil} are nothing but those attached to a different representative $\nu'$ of the valuation $\nu$. The motivation for considering $\nu'$ is to study the variation of infinitesimal Newton-Okounkov bodies as $E_r$ and $q$ vary.

The first main goal in this paper is to give a completely explicit expression of the Newton-Okounkov body $\Delta_\nu$ attached to any flag as above. Let us explain our result.  For a start, recall that the volume of a divisorial valuation $\nu_r$ is defined as
 \[
\mathrm{vol}(\nu_r) =\lim_{\alpha \rightarrow \infty} \frac{\mathrm{dim}_{\mathbb{C}} ( \mathcal{O}_{\mathbb{P}^2,p} \; / \;  \mathcal{P}_\alpha)}{\alpha^2/2},
\]
where $\mathcal{P}_\alpha = \{f \in \mathcal{O}_{\mathbb{P}^2,p} | \nu_r(f) \geq \alpha\} \cup \{0\}$. It is satisfied that $\hat{\mu} (\nu_r) \geq  \sqrt{1/\mathrm{vol}(\nu_r)}$ and $\nu_r$ is said to be minimal when the equality holds. An exceptional curve valuation $\nu$ is named {\it minimal} whenever its first component is minimal.

The following result, which is Theorem \ref{nobminimal} in the paper, describes $\Delta_\nu$ for exceptional curve valuations which are minimal.

\begin{A} {\it Consider a flag $E_{\bullet}:= \left\{ X=X_r \supset E_r \supset \{q\} \right\}$ and $\nu = \nu_{E_{\bullet}}$ its attached valuation. Then, the valuation $\nu$ is minimal if and only if the Newton-Okounkov body $\Delta_\nu$ is a triangle whose vertices are $(0,0),\;\left(\hat{\mu}(\nu_r), \frac{\hat{\mu}(\nu_r)\betabarra_0(\nu_{\eta}) }{\betabarra_0(\nu_r)} \right)$  $\mbox{and}\; \left(\hat{\mu}(\nu_r), \frac{\hat{\mu}(\nu_r)\betabarra_g(\nu_{\eta})}{\betabarra_g(\nu_r)} \right),$
whenever $q$  belongs to the intersection of the strict transform of two exceptional divisors, i.e.  $ q \in E_r\cap E_{\eta}$ with $r\neq \eta$. Otherwise, these vertices are
$(0,0),$ $\left(\hat{\mu}(\nu_r), 0 \right)$ $ \mbox{and}\; \left(\hat{\mu}(\nu_r), \frac{\hat{\mu}(\nu_r)}{\betabarra_{g+1}(\nu_r)} \right).$ }
\end{A}

The values $\{\betabarra_{j} (\nu_i) \}_{j=0}^{g+1}$ in the above result  are the so-called maximal contact values of the divisorial valuations $\nu_i$, defined by the exceptional divisors $E_i$ appearing in the sequence of point blowing-ups given by the divisorial valuation $\nu_r$, where either $i=r$ or $i=\eta$ (see Subsection \ref{Puiseux}). Notice that the maximal contact values generate the value semigroup of the corresponding valuations and can be easily computed from their dual graphs.

When $\nu$ (and so $\nu_r$) is not minimal, there exists a unique integral curve $C$ given by a polynomial $f\in \mathbb{C}[u,v]$ such that $\hat{\mu}(\nu_r) = \nu_r(f) / \deg(f)$, which is named supra-minimal for $\nu_r$ (see Lemma \ref{supra}). Setting $$ \frac{\nu_r(f)}{\deg(f)} = (\hat{\mu}(\nu_r),c),$$
we are able to prove the following result (Theorem \ref{53} in the paper) which describes the Newton-Okounkov bodies of flags $E_{\bullet}$ corresponding to non-minimal valuations $\nu$.

\begin{B} {\it Assume that $\nu = \nu_{E_{\bullet}}$ is not minimal, then the Newton-Okounkov body $\Delta_{\nu}$ is the convex hull of the set $\{(0,0), Q_1, Q_2, Q_3\}$, where $$Q_1=\frac{1}{\hat{\mu}(\nu_r)}(\betabarra_{g+1}(\nu_r), \nu_{r}(\varphi_{\eta})),  \;\;\; Q_2=\frac{1}{\hat{\mu}(\nu_r)}(\betabarra_{g+1}(\nu_r), \nu_{r}(\varphi_{\eta})+1)$$ and $Q_3=(\hat{\mu}(\nu_r),c)$ whenever $q \in E_r\cap E_{\eta}$ with $r\neq \eta$. Otherwise, these points are
$Q_1=\frac{1}{\hat{\mu}(\nu_r)}(\betabarra_{g+1}(\nu_r), 0)$, $Q_2=\frac{1}{\hat{\mu}(\nu_r)}(\betabarra_{g+1}(\nu_r),1)$ and $Q_3=(\hat{\mu}(\nu_r),c)$. As a consequence, $\Delta_{\nu}$ is either a triangle or a quadrilateral.}
\end{B}
Above, $\varphi_{\eta}$ is a general analytically irreducible element in $\mathcal{O}_{\mathbb{P}^2,p}$ whose strict transform on $X$ is transversal to the exceptional divisor $E_{\eta}$ (see Definition \ref{curveta}).

Recall that the number of vertices of the Newton-Okounkov body defined by a flag and a big divisor on a surface $X$ is bounded by $2 \varrho +2$, where $\varrho$ is the Picard number of $X$ \cite{KLM}. This is a consequence of the fact that a ray of the form $D - t C$, where $D$ is a big divisor and $C$ a curve on $X$, can only cross $\rho + 1$ Zariski chambers. It is conjectured in \cite{kl} (see also \cite{crm}) that the bound could be applied even if the flag is considered on a projective model dominating $X$ (and the Newton-Okounkov body associated to the pull-back of a big divisor on $X$). Our results can be regarded as new evidence supporting the conjecture.

Our second main goal is presented in Theorem \ref{triangulo}, where we characterize which Newton-Okounkov bodies are triangular (and, by exclusion, quadrangular) in the case in which $\nu$ is not minimal. We prove that the fact of being a triangle or a quadrilateral depends only on the branches of the supraminimal curve $C$ of $\nu_r$. Roughly speaking, the divisor $E_r$ gives a partition of the dual graph of $\nu$ into two connected components, one of them being finite and the other one infinite. {\it The Newton-Okounkov body $\Delta_{\nu}$ will be a triangle when all the branches of $C$ go through the same component, and a quadrilateral otherwise}. In particular, we prove that the Newton-Okounkov bodies of valuations $\nu$ with supraminimal curves having only one branch at $p$ are triangles.

The values $\hat{\mu} (\nu_r)$ and the supraminimal curves are not easy to compute. An interesting class of divisorial valuations are the so-called non-positive at infinity ones (see the beginning of Section \ref{NPI}). As one can see in \cite{GM}, they define surfaces with regular cone of curves and, for these surfaces, one can decide (by checking a condition) when their Cox rings are finitely generated. In Corollary \ref{nonpos} we prove that {\it the Newton-Okounkov bodies $\Delta_{\nu}$ of the valuations $\nu$ in  this large class are triangles and we give explicitly their vertices.} Although we deduce this result from Theorem B, we also provide in Remark \ref{la53} the Zariski decomposition of those divisors that, according to Theorem 6.4 of \cite{LM}, would allow us to compute $\Delta_\nu$.

The last section of this paper contains three examples, in which Newton-Okounkov bodies of flags of different nature are computed; we hope this will help the readers in a better understanding of our results.

\section{Divisorial and exceptional curve valuations}

Thereafter we will see that the valuations defined by  those flags we are interested in are exceptional curve valuations and, also, that their first components are divisorial valuations. So we devote this section to study these classes of valuations, providing results which will be useful.

Let $\mathbb{P}^2$ be the complex projective plane. For $p \in \mathbb{P}^2$, let $R$ be the local ring of $\mathbb{P}^2$ at $p$, and let $\mathfrak{m}$ be its maximal ideal; moreover, write $F$ for the field of fractions of $R$. A {\em valuation} of $F$ is an onto mapping $\nu : F^* (:= F \setminus \{0\}) \rightarrow G$, where $G$ is a totally ordered group such that
$\nu (f + h) \geq \min \{\nu(f), \nu (h) \}$ and $\nu (fh) =\nu (f) + \nu (h)$, for $f, h \in F^*$. The group $G$ is called the value group of $\nu$. The subring $R_{\nu}:=\{ f \in F^* \, | \, \nu(f) \geq 0\} \cup \{0\}$ of $F$ is called the {\em valuation ring of $\nu$}. The ring $R_{\nu}$ is local with maximal ideal $\mathfrak{m}_{\nu}:=\{ f \in F^* \, | \, \nu(f) > 0\} \cup \{0\}$. The valuations we consider are {\em centered at $R$}, that is, such that $R \subseteq R_{\nu}$ and $R \cap \mathfrak{m}_{\nu} = \mathfrak{m}$. Finally, the monoid $\nu(R \setminus \{0\})$ is called the {\em semigroup} of the valuation $\nu$ (relative to $R$). Most of the results in this section are also true when using an algebraically closed field instead of the complex numbers \cite{campillo, deganu}.

The valuations $\nu$ of $F$ centered at $R$ are in one-to-one correspondence with simple sequences of blowing-ups of points whose first center is $p$:
\begin{equation}\label{seq}
\pi : \;\;\; \cdots \longrightarrow X_n \longrightarrow X_{n-1} \longrightarrow \cdots \longrightarrow X_1 \longrightarrow X_0=\mathbb{P}^2.
\end{equation}
Here simple means that we only blow-up points $p_i$, $i >1$, belonging to the exceptional divisor created last. The \emph{cluster of centers} of $\pi$ will be denoted by  $\mathcal{C}_{\nu}=\{p=p_1, p_2,  \ldots\}$ and for  $i > j$ we say that a point $p_i$ is proximate to $p_j$, written $p_i \rightarrow p_j$, whenever $p_i$ belongs to the strict transform of the exceptional divisor $E_j$ obtained by blowing-up $p_j$. For all $i\geq 1$ we denote by $E_i$ the exceptional divisor on $X_i$ obtained by blowing-up $p_i$, and we say that $p_i$ and $E_i$ are \emph{satellite} if $p_i$ is proximate to two points of $\mathcal{C}_{\nu}$; otherwise we say that they are \emph{free}. For $i \geq 1$, we will write $R_i = \mathcal{O}_{X_i,p_{i+1}}$.

The above mentioned valuations (considered up to equivalence) were classified in five types by Spivakovsky \cite{spiv}. We are interested in one of them, namely the exceptional curve valuations---in the terminology of \cite{FJ}---, which correspond to Case 3 in \cite{spiv} and to those of type C of \cite{deganu}. Divisorial valuations are another remarkable type in Spivakovsky's classification. They correspond to finite sequences $\pi$ and will be merely instrumental in our development.

More specifically, $\nu$ is {\it divisorial} if it is defined by the order of vanishing along an exceptional divisor $E_n$ obtained from a finite simple sequence of blowing-ups $\pi$. In this case ${\mathcal C}_{\nu}=\{p_i\}_{i=1}^n$, $G \cong \mathbb{Z}$ and, when its value group $G$ is included in $\mathbb{R}$, there exists $0 \neq c \in \mathbb{R}$ such that $\nu(f)=c\cdot {\rm ord}_{E_n}(f)$ for all $f\in F^*$. Divisorial valuations associated with the same divisor $E_n$ but with different constants $c$ are said to be \emph{equivalent} \cite{z-s}. Throughout this paper, the representatives we choose for divisorial valuations will correspond to $c=1$.

A valuation $\nu$ is \emph{exceptional curve} if ${\mathcal C}_{\nu}$ is infinite and there exists a point $p_{r} \in \mathcal{C}_{\nu}$ such that $p_i \rightarrow p_{r}$ for all $i >r$. In this case the group $G$ is non-archimedian, so $G\cong \mathbb{Z}^2$ (with lexicographical ordering).

\subsection{Dual graphs and Enriques diagrams of divisorial and exceptional curve valuations}
\label{32}

For any non-empty finite set ${\mathcal W}$ of effective divisors on a non-singular surface, we define its \emph{dual graph} $\Gamma_{\mathcal W}$ as the graph whose set of vertices is in 1-1 correspondence with $\mathcal W$, and two vertices are joined by an edge if and only if the associated divisors intersect. The dual graph $\Gamma_{\nu}$ of  a divisorial valuation $\nu$, with associated configuration ${\mathcal C}_{\nu}=\{p_i\}_{i=1}^n$, is the dual graph of the set of divisors $\{E_1, E_2, \ldots, E_n\}$ together with a label $i$ attached to the vertex corresponding to $E_i$;  the symbols $E_i$ stand for the strict transforms of the divisors $E_i$ on $X_n$; this dual notation (for divisors and their strict transforms) will be used sometimes throughout  this paper. The graph $\Gamma_{\nu}$ is a rooted tree whose root is the vertex that is labeled with $\mathbf{1}$.

In Figure \ref{fig2} we have depicted the dual graph of a divisorial valuation $\nu$. Here we have shown only some labels. The vertices different from $\mathbf{1}$ which are adjacent to a unique vertex are called \emph{dead ends} (labeled as $\ell_1, \ell_2, \ldots, \ell_g$ in Figure \ref{fig2}), and those adjacent to three vertices are called \emph{star vertices}. We have labeled the star vertices with $st_1, st_2, \ldots, st_g$. If the divisor defining $\nu$ is free, then there is a finite sequence of vertices corresponding to free divisors which appears  after $st_g$ (the \emph{tail}, in Figure \ref{fig2}) and the last vertex is a dead end. Otherwise this tail does not appear and $st_g$ is only adjacent to two vertices.

\begin{figure}[h]
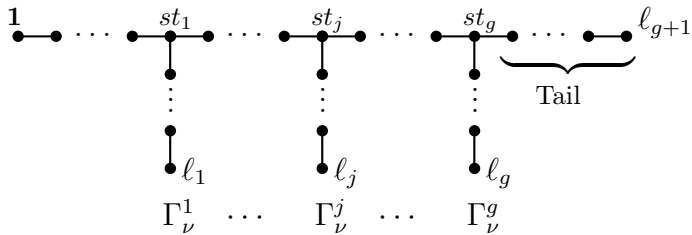


\begin{center}
\setlength{\unitlength}{0.5cm}%
\begin{Picture}(0,0)(20,8)
\thicklines


\xLINE(0,6)(1,6)
\Put(0,6){\circle*{0.3}}
\Put(1,6){\circle*{0.3}}
\put(1,6){$\;\;\ldots\;\;$}
\xLINE(3,6)(4,6)
\Put(3,6){\circle*{0.3}}
\Put(4,6){\circle*{0.3}}
\xLINE(4,6)(4,5)
\Put(4,5){\circle*{0.3}}
\Put(3.9,4){$\vdots$}
\xLINE(4,3.5)(4,2.5)
\Put(4,3.5){\circle*{0.3}}
\Put(4,2.5){\circle*{0.3}}
\Put(3.8,1){$\Gamma_{\nu}^1$}

\put(4.3,2.2){$\ell_1$}

\Put(3.7,6.3){\footnotesize $st_1$}
\Put(-0.3,6.3){\footnotesize $\mathbf{1}$}



\xLINE(4,6)(5,6)
\Put(4,6){\circle*{0.3}}
\Put(5,6){\circle*{0.3}}
\Put(5,6){$\;\;\ldots\;\;$}
\xLINE(7,6)(8,6)
\Put(7,6){\circle*{0.3}}
\Put(8,6){\circle*{0.3}}
\xLINE(8,6)(8,5)
\Put(8,5){\circle*{0.3}}
\Put(7.9,4){$\vdots$}
\xLINE(8,3.5)(8,2.5)
\Put(8,3.5){\circle*{0.3}}
\Put(8,2.5){\circle*{0.3}}
\Put(7.8,1){$\Gamma_{\nu}^j$}

\put(8.3,2.2){$\ell_j$}

\Put(7.7,6.3){\footnotesize $st_j$}

\Put(5,1){$\;\;\cdots\;\;$}


\xLINE(8,6)(9,6)
\Put(8,6){\circle*{0.3}}
\Put(9,6){\circle*{0.3}}
\put(9,6){$\;\;\ldots\;\;$}
\xLINE(11,6)(12,6)
\Put(11,6){\circle*{0.3}}
\Put(12,6){\circle*{0.3}}
\xLINE(12,6)(12,5)
\Put(12,5){\circle*{0.3}}
\Put(11.9,4){$\vdots$}
\xLINE(12,3.5)(12,2.5)
\Put(12,3.5){\circle*{0.3}}
\Put(12,2.5){\circle*{0.3}}
\Put(11.8,1){$\Gamma_{\nu}^g$}

\put(12.3,2.2){$\ell_g$}

\Put(11.7,6.3){\footnotesize $st_g$}

\Put(9,1){$\;\;\cdots\;\;$}


\xLINE(12,6)(13,6)
\Put(13,6){\circle*{0.3}}
\Put(13,6){$\;\;\ldots\;\;$}
\Put(15,6){\circle*{0.3}}
\xLINE(15,6)(16,6)
\Put(16,6){\circle*{0.3}}


\Put(12.7,5.5){$\underbrace{\;\;\;\;\;\;\;\;\;\;\;\;\;\;\;}$}
\Put(13.6,4.2){{\footnotesize Tail}}

\Put(16.3,6.2){$\ell_{g+1}$}

\end{Picture}
\end{center}
  \caption{Dual graph of a divisorial valuation}
  \label{fig2}
\end{figure}

Consider also the following ordering on the set of vertices: $\alpha \preccurlyeq \beta$ if the path in the dual graph joining $\mathbf{1}$ and $\beta$ goes through $\alpha$. For each $j=1, 2, \ldots,g$, we denote by $\Gamma_{\nu}^j$  the subgraph given by the vertices $\alpha$ such that $st_{j-1} \preccurlyeq \alpha \preccurlyeq \ell_j$  (where $st_0:=\mathbf{1}$) and the edges joining them.  The number $g$ of subgraphs $\Gamma_{\nu}^j$ will be called the \emph{number of Puiseux pairs} of $\nu$ in the sequel. Moreover, we will say that an exceptional divisor $E_i$ (or the point $p_i$) \emph{belongs to the $j$th Puiseux pair} of $\nu$ if its associated vertex is a vertex of $\Gamma_{\nu}^j$. The vertices of each subgraph $\Gamma_{\nu}^j$ corresponding to free divisors are some of the first consecutive ones (with respect to the ordering $\preccurlyeq$) and the one labeled by $\ell_j$; we call them (respectively, the exceptional divisors $E_i$ and points $p_i$ associated to them) the \emph{free part} of $\Gamma_{\nu}^j$ (respectively, the free part of the \emph{$j$th Puiseux pair} of $\nu$). Notice that, for $1\leq j<g$, the exceptional divisors $E_{st_j}$ belong to two consecutive Puiseux pairs and they are satellite.

The Enriques diagram of the configuration ${\mathcal C}_{\nu}$ (or of $\nu$) \cite[IV.I]{Enriques}, denoted by ${\mathcal E}_{\nu}$, is a rooted tree (with root $O=p_1$) that has a vertex for each point in ${\mathcal C}_{\nu}$ and an edge joining each pair of vertices that represent consecutive points $p_i$ and $p_{i+1}$. These edges are of two different kinds, straight or curved, according to the following rules:

\begin{itemize}
\item[$\diamond$] If $p_{i+1}$ is free then the edge joining $p_i$ and $p_{i+1}$ is smooth, curved and such that, if $p_i\neq O$, it has at $p_i$ the same tangent as the edge ending at $p_i$.
\item[$\diamond$] Assume that the vertices $p_i$ and $p_{i+1}$ (and the edge joining them) have been represented. Then all the infinitely near to $p_{i+1}$ points which are proximate to $p_i$, as well as the edges joining them, appear on a straight half-line which starts at $p_{i+1}$ and is orthogonal to the edge $p_i p_{i+1}$ at $p_{i+1}$. To avoid self-intersections of the diagram, such half-lines are alternately oriented to the right and to the left of the preceding edge.
\end{itemize}

Figure \ref{figura3} shows the Enriques diagram of a divisorial valuation. We have labeled the vertices associated with the points $p_{st_j}$ and $p_{\ell_j}$ for all $0\leq j\leq g$. It consists of the concatenation of $g$ subgraphs corresponding to the points in the $g$ Puiseux pairs of $\nu$ (see Figure \ref{fig4}) and, if the valuation is defined by a free divisor (as in Figure \ref{figura3}), a sequence of vertices corresponding to free points which appear after $p_{st_g}$ (the tail).

\begin{figure}
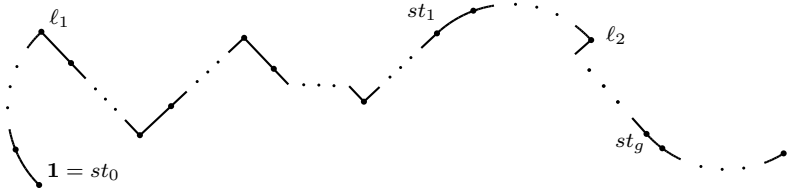

\begin{center}

\newcommand{\norma}[2]{%
\SQUARE{#1}{\m}
\SQUARE{#2}{\n}
\ADD{\m}{\n}{\k}
\SQUAREROOT{\k}{\norm}
}


\newcommand{\arco}[6]{%
\norma{#1}{#2}
\DIVIDE{#2}{\norm}{\tangentey}
\DIVIDE{#1}{\norm}{\tangentex}
\COPY{\tangentey}{\ortogonalx}
\COPY{\tangentey}{\centrox}
\SUBTRACT{0}{\tangentex}{\ortogonaly}
\MULTIPLY{\centrox}{#6}{\centrox}
\MULTIPLY{\ortogonaly}{#6}{\centroy}
\ADD{#3}{\centrox}{\centrox}
\ADD{#4}{\centroy}{\centroy}
\changereferencesystem(\centrox,\centroy)(\ortogonalx,\ortogonaly)(\tangentex,\tangentey)
\radiansangles
\SUBTRACT{\numberPI}{#5}{\angulo}
\ellipticArc{#6}{#6}{\angulo}{\numberPI}
\COS{\angulo}{\puntox}
\SIN{\angulo}{\puntoy}
\MULTIPLY{#6}{\puntox}{\puntox}
\MULTIPLY{#6}{\puntoy}{\puntoy}
\Put(\puntox,\puntoy){\circle*{0.08}}
\SUBTRACT{0}{\puntox}{\direcciony}
\COPY{\puntoy}{\direccionx}
\norma{\direccionx}{\direcciony}
\DIVIDE{\direccionx}{\norm}{\direccionx}
\DIVIDE{\direcciony}{\norm}{\direcciony}
\COPY{\direcciony}{\ortogonalx}
\COPY{\direccionx}{\ortogonaly}
\SUBTRACT{0}{\ortogonaly}{\ortogonaly}
\changereferencesystem(\puntox,\puntoy)(\ortogonalx,\ortogonaly)(\direccionx,\direcciony)
}

\newcommand{\arcosinpunto}[6]{%
\norma{#1}{#2}
\DIVIDE{#2}{\norm}{\tangentey}
\DIVIDE{#1}{\norm}{\tangentex}
\COPY{\tangentey}{\ortogonalx}
\COPY{\tangentey}{\centrox}
\SUBTRACT{0}{\tangentex}{\ortogonaly}
\MULTIPLY{\centrox}{#6}{\centrox}
\MULTIPLY{\ortogonaly}{#6}{\centroy}
\ADD{#3}{\centrox}{\centrox}
\ADD{#4}{\centroy}{\centroy}
\changereferencesystem(\centrox,\centroy)(\ortogonalx,\ortogonaly)(\tangentex,\tangentey)
\radiansangles
\SUBTRACT{\numberPI}{#5}{\angulo}
\ellipticArc{#6}{#6}{\angulo}{\numberPI}
\COS{\angulo}{\puntox}
\SIN{\angulo}{\puntoy}
\MULTIPLY{#6}{\puntox}{\puntox}
\MULTIPLY{#6}{\puntoy}{\puntoy}
\SUBTRACT{0}{\puntox}{\direcciony}
\COPY{\puntoy}{\direccionx}
\norma{\direccionx}{\direcciony}
\DIVIDE{\direccionx}{\norm}{\direccionx}
\DIVIDE{\direcciony}{\norm}{\direcciony}
\COPY{\direcciony}{\ortogonalx}
\COPY{\direccionx}{\ortogonaly}
\SUBTRACT{0}{\ortogonaly}{\ortogonaly}
\changereferencesystem(\puntox,\puntoy)(\ortogonalx,\ortogonaly)(\direccionx,\direcciony)
}

\newcommand{\arcoblanco}[6]{%
\norma{#1}{#2}
\DIVIDE{#2}{\norm}{\tangentey}
\DIVIDE{#1}{\norm}{\tangentex}
\COPY{\tangentey}{\ortogonalx}
\COPY{\tangentey}{\centrox}
\SUBTRACT{0}{\tangentex}{\ortogonaly}
\MULTIPLY{\centrox}{#6}{\centrox}
\MULTIPLY{\ortogonaly}{#6}{\centroy}
\ADD{#3}{\centrox}{\centrox}
\ADD{#4}{\centroy}{\centroy}
\changereferencesystem(\centrox,\centroy)(\ortogonalx,\ortogonaly)(\tangentex,\tangentey)
\radiansangles
\SUBTRACT{\numberPI}{#5}{\angulo}
\COS{\angulo}{\puntox}
\SIN{\angulo}{\puntoy}
\MULTIPLY{#6}{\puntox}{\puntox}
\MULTIPLY{#6}{\puntoy}{\puntoy}
\Put(\puntox,\puntoy){\circle*{0.04}}
\SUBTRACT{0}{\puntox}{\direcciony}
\COPY{\puntoy}{\direccionx}
\norma{\direccionx}{\direcciony}
\DIVIDE{\direccionx}{\norm}{\direccionx}
\DIVIDE{\direcciony}{\norm}{\direcciony}
\COPY{\direcciony}{\ortogonalx}
\COPY{\direccionx}{\ortogonaly}
\SUBTRACT{0}{\ortogonaly}{\ortogonaly}
\changereferencesystem(\puntox,\puntoy)(\ortogonalx,\ortogonaly)(\direccionx,\direcciony)
}

\newcommand{\arcoblancosinpunto}[6]{%
\norma{#1}{#2}
\DIVIDE{#2}{\norm}{\tangentey}
\DIVIDE{#1}{\norm}{\tangentex}
\COPY{\tangentey}{\ortogonalx}
\COPY{\tangentey}{\centrox}
\SUBTRACT{0}{\tangentex}{\ortogonaly}
\MULTIPLY{\centrox}{#6}{\centrox}
\MULTIPLY{\ortogonaly}{#6}{\centroy}
\ADD{#3}{\centrox}{\centrox}
\ADD{#4}{\centroy}{\centroy}
\changereferencesystem(\centrox,\centroy)(\ortogonalx,\ortogonaly)(\tangentex,\tangentey)
\radiansangles
\SUBTRACT{\numberPI}{#5}{\angulo}
\COS{\angulo}{\puntox}
\SIN{\angulo}{\puntoy}
\MULTIPLY{#6}{\puntox}{\puntox}
\MULTIPLY{#6}{\puntoy}{\puntoy}
\SUBTRACT{0}{\puntox}{\direcciony}
\COPY{\puntoy}{\direccionx}
\norma{\direccionx}{\direcciony}
\DIVIDE{\direccionx}{\norm}{\direccionx}
\DIVIDE{\direcciony}{\norm}{\direcciony}
\COPY{\direcciony}{\ortogonalx}
\COPY{\direccionx}{\ortogonaly}
\SUBTRACT{0}{\ortogonaly}{\ortogonaly}
\changereferencesystem(\puntox,\puntoy)(\ortogonalx,\ortogonaly)(\direccionx,\direcciony)
}

\newcommand{\segmento}[4]{%
\MULTIPLY{#3}{#4}{\long}
\MULTIPLY{#1}{\long}{\dirx}
\MULTIPLY{#2}{\long}{\diry}
\xLINE(0,0)(\dirx,\diry)
\Put(\dirx,\diry){\circle*{0.08}}
\changereferencesystem(\dirx,\diry)(#2,-#1)(#1,#2)
}

\newcommand{\segmentosinpunto}[4]{%
\MULTIPLY{#3}{#4}{\long}
\MULTIPLY{#1}{\long}{\dirx}
\MULTIPLY{#2}{\long}{\diry}
\xLINE(0,0)(\dirx,\diry)
\changereferencesystem(\dirx,\diry)(#2,-#1)(#1,#2)
}

\newcommand{\segmentoblanco}[4]{%
\MULTIPLY{#3}{#4}{\long}
\MULTIPLY{#1}{\long}{\dirx}
\MULTIPLY{#2}{\long}{\diry}
\Put(\dirx,\diry){\circle*{0.04}}
\changereferencesystem(\dirx,\diry)(#2,-#1)(#1,#2)
}

\newcommand{\segmentoblancosinpunto}[4]{%
\MULTIPLY{#3}{#4}{\long}
\MULTIPLY{#1}{\long}{\dirx}
\MULTIPLY{#2}{\long}{\diry}
\changereferencesystem(\dirx,\diry)(#2,-#1)(#1,#2)
}

\setlength{\unitlength}{1cm}%

\begin{Picture}(-1,-1)(11,3)
\thicklines

\referencesystem(0,0)(1,1)(-1,1)

\COPY{1}{\rad}
\COPY{0.4}{\ang}
\DIVIDE{\ang}{2}{\angbis}
\DIVIDE{\ang}{3}{\angbisbis}

\COPY{0}{\taux}
\COPY{1}{\tauy}
\COPY{0}{\pux}
\COPY{0}{\puy}
\Put(\pux,\puy){\circle*{0.08}}
\Put(0.1,\puy){\tiny{$\mathbf{1}=st_0$}}

\arco{\taux}{\tauy}{\pux}{\puy}{\ang}{\rad}

\arcosinpunto{0}{1}{0}{0}{\angbis}{\rad}
\arcoblanco{0}{1}{0}{0}{\angbis}{\rad}
\arcoblanco{0}{1}{0}{0}{\angbis}{\rad}
\arcoblanco{0}{1}{0}{0}{\angbis}{\rad}
\arcoblancosinpunto{0}{1}{0}{0}{\angbis}{\rad}
\arco{0}{1}{0}{0}{\angbis}{\rad}
\Put(0,0.1){\tiny{$\ell_1$}}

\segmento{1}{0}{\ang}{\rad}
\segmentosinpunto{0}{1}{\angbis}{\rad}
\segmentoblanco{0}{1}{\angbisbis}{\rad}
\segmentoblanco{0}{1}{\angbisbis}{\rad}
\segmentoblanco{0}{1}{\angbisbis}{\rad}
\segmentoblancosinpunto{0}{1}{\angbisbis}{\rad}
\segmento{0}{1}{\angbis}{\rad}

\segmento{-1}{0}{\ang}{\rad}
\segmentosinpunto{0}{1}{\angbis}{\rad}
\segmentoblanco{0}{1}{\angbisbis}{\rad}
\segmentoblanco{0}{1}{\angbisbis}{\rad}
\segmentoblanco{0}{1}{\angbisbis}{\rad}
\segmentoblancosinpunto{0}{1}{\angbisbis}{\rad}
\segmento{0}{1}{\angbis}{\rad}

\segmento{1}{0}{\ang}{\rad}
\segmentosinpunto{0}{1}{\angbis}{\rad}

\Put(-0.1,0.1){\circle*{0.04}}
\Put(-0.2,0.2){\circle*{0.04}}
\Put(-0.3,0.3){\circle*{0.04}}
\changereferencesystem(-0.4,0.4)(1,0)(0,1)

\segmento{0}{1}{\angbis}{\rad}
\segmentosinpunto{-1}{0}{\angbis}{\rad}

\segmentoblanco{0}{1}{\angbisbis}{\rad}
\segmentoblanco{0}{1}{\angbisbis}{\rad}
\segmentoblanco{0}{1}{\angbisbis}{\rad}
\segmentoblancosinpunto{0}{1}{\angbisbis}{\rad}
\segmento{0}{1}{\angbis}{\rad}

\Put(-0.3,-0.1){\tiny{$st_1$}}

\arco{0}{1}{0}{0}{\ang}{\rad}
\arcosinpunto{0}{1}{0}{0}{\angbis}{\rad}
\arcoblanco{0}{1}{0}{0}{\angbis}{\rad}
\arcoblanco{0}{1}{0}{0}{\angbis}{\rad}
\arcoblanco{0}{1}{0}{0}{\angbis}{\rad}
\arcoblancosinpunto{0}{1}{0}{0}{\angbis}{\rad}
\arco{0}{1}{0}{0}{\angbis}{\rad}

\Put(-0.1,0.1){\tiny{${\ell}_2$}}

\segmentosinpunto{1}{0}{\angbis}{\rad}
\Put(-0.2,0){\circle*{0.05}}
\Put(-0.4,0){\circle*{0.05}}
\Put(-0.6,0){\circle*{0.05}}

\changereferencesystem(-0.8,0)(1,0)(0,1)
\segmento{-1}{0}{\angbis}{\rad}


\changereferencesystem(0,0)(-1,0)(0,1)

\Put(-0.3,-0.1){\tiny{$st_g$}}
\arco{0}{1}{0}{0}{\angbis}{\rad}

\arcosinpunto{0}{1}{0}{0}{\angbis}{\rad}
\arcoblanco{0}{1}{0}{0}{\angbis}{\rad}
\arcoblanco{0}{1}{0}{0}{\angbis}{\rad}
\arcoblanco{0}{1}{0}{0}{\angbis}{\rad}
\arcoblancosinpunto{0}{1}{0}{0}{\angbis}{\rad}
\arco{0}{1}{0}{0}{\angbis}{\rad}

\end{Picture}
\end{center}

\caption{Enriques diagram of a divisorial valuation}
  \label{figura3}
\end{figure}


\begin{figure}
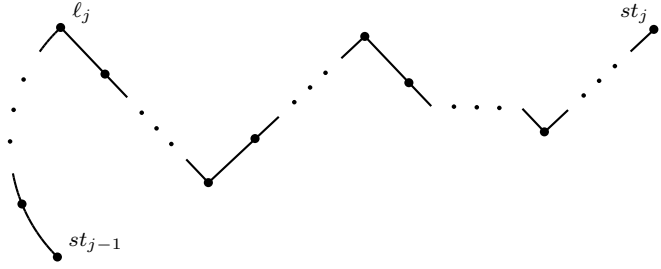

\begin{center}

\newcommand{\norma}[2]{%
\SQUARE{#1}{\m}
\SQUARE{#2}{\n}
\ADD{\m}{\n}{\k}
\SQUAREROOT{\k}{\norm}
}


\newcommand{\arco}[6]{%
\norma{#1}{#2}
\DIVIDE{#2}{\norm}{\tangentey}
\DIVIDE{#1}{\norm}{\tangentex}
\COPY{\tangentey}{\ortogonalx}
\COPY{\tangentey}{\centrox}
\SUBTRACT{0}{\tangentex}{\ortogonaly}
\MULTIPLY{\centrox}{#6}{\centrox}
\MULTIPLY{\ortogonaly}{#6}{\centroy}
\ADD{#3}{\centrox}{\centrox}
\ADD{#4}{\centroy}{\centroy}
\changereferencesystem(\centrox,\centroy)(\ortogonalx,\ortogonaly)(\tangentex,\tangentey)
\radiansangles
\SUBTRACT{\numberPI}{#5}{\angulo}
\ellipticArc{#6}{#6}{\angulo}{\numberPI}
\COS{\angulo}{\puntox}
\SIN{\angulo}{\puntoy}
\MULTIPLY{#6}{\puntox}{\puntox}
\MULTIPLY{#6}{\puntoy}{\puntoy}
\Put(\puntox,\puntoy){\circle*{0.08}}
\SUBTRACT{0}{\puntox}{\direcciony}
\COPY{\puntoy}{\direccionx}
\norma{\direccionx}{\direcciony}
\DIVIDE{\direccionx}{\norm}{\direccionx}
\DIVIDE{\direcciony}{\norm}{\direcciony}
\COPY{\direcciony}{\ortogonalx}
\COPY{\direccionx}{\ortogonaly}
\SUBTRACT{0}{\ortogonaly}{\ortogonaly}
\changereferencesystem(\puntox,\puntoy)(\ortogonalx,\ortogonaly)(\direccionx,\direcciony)
}

\newcommand{\arcosinpunto}[6]{%
\norma{#1}{#2}
\DIVIDE{#2}{\norm}{\tangentey}
\DIVIDE{#1}{\norm}{\tangentex}
\COPY{\tangentey}{\ortogonalx}
\COPY{\tangentey}{\centrox}
\SUBTRACT{0}{\tangentex}{\ortogonaly}
\MULTIPLY{\centrox}{#6}{\centrox}
\MULTIPLY{\ortogonaly}{#6}{\centroy}
\ADD{#3}{\centrox}{\centrox}
\ADD{#4}{\centroy}{\centroy}
\changereferencesystem(\centrox,\centroy)(\ortogonalx,\ortogonaly)(\tangentex,\tangentey)
\radiansangles
\SUBTRACT{\numberPI}{#5}{\angulo}
\ellipticArc{#6}{#6}{\angulo}{\numberPI}
\COS{\angulo}{\puntox}
\SIN{\angulo}{\puntoy}
\MULTIPLY{#6}{\puntox}{\puntox}
\MULTIPLY{#6}{\puntoy}{\puntoy}
\SUBTRACT{0}{\puntox}{\direcciony}
\COPY{\puntoy}{\direccionx}
\norma{\direccionx}{\direcciony}
\DIVIDE{\direccionx}{\norm}{\direccionx}
\DIVIDE{\direcciony}{\norm}{\direcciony}
\COPY{\direcciony}{\ortogonalx}
\COPY{\direccionx}{\ortogonaly}
\SUBTRACT{0}{\ortogonaly}{\ortogonaly}
\changereferencesystem(\puntox,\puntoy)(\ortogonalx,\ortogonaly)(\direccionx,\direcciony)
}

\newcommand{\arcoblanco}[6]{%
\norma{#1}{#2}
\DIVIDE{#2}{\norm}{\tangentey}
\DIVIDE{#1}{\norm}{\tangentex}
\COPY{\tangentey}{\ortogonalx}
\COPY{\tangentey}{\centrox}
\SUBTRACT{0}{\tangentex}{\ortogonaly}
\MULTIPLY{\centrox}{#6}{\centrox}
\MULTIPLY{\ortogonaly}{#6}{\centroy}
\ADD{#3}{\centrox}{\centrox}
\ADD{#4}{\centroy}{\centroy}
\changereferencesystem(\centrox,\centroy)(\ortogonalx,\ortogonaly)(\tangentex,\tangentey)
\radiansangles
\SUBTRACT{\numberPI}{#5}{\angulo}
\COS{\angulo}{\puntox}
\SIN{\angulo}{\puntoy}
\MULTIPLY{#6}{\puntox}{\puntox}
\MULTIPLY{#6}{\puntoy}{\puntoy}
\Put(\puntox,\puntoy){\circle*{0.04}}
\SUBTRACT{0}{\puntox}{\direcciony}
\COPY{\puntoy}{\direccionx}
\norma{\direccionx}{\direcciony}
\DIVIDE{\direccionx}{\norm}{\direccionx}
\DIVIDE{\direcciony}{\norm}{\direcciony}
\COPY{\direcciony}{\ortogonalx}
\COPY{\direccionx}{\ortogonaly}
\SUBTRACT{0}{\ortogonaly}{\ortogonaly}
\changereferencesystem(\puntox,\puntoy)(\ortogonalx,\ortogonaly)(\direccionx,\direcciony)
}

\newcommand{\arcoblancosinpunto}[6]{%
\norma{#1}{#2}
\DIVIDE{#2}{\norm}{\tangentey}
\DIVIDE{#1}{\norm}{\tangentex}
\COPY{\tangentey}{\ortogonalx}
\COPY{\tangentey}{\centrox}
\SUBTRACT{0}{\tangentex}{\ortogonaly}
\MULTIPLY{\centrox}{#6}{\centrox}
\MULTIPLY{\ortogonaly}{#6}{\centroy}
\ADD{#3}{\centrox}{\centrox}
\ADD{#4}{\centroy}{\centroy}
\changereferencesystem(\centrox,\centroy)(\ortogonalx,\ortogonaly)(\tangentex,\tangentey)
\radiansangles
\SUBTRACT{\numberPI}{#5}{\angulo}
\COS{\angulo}{\puntox}
\SIN{\angulo}{\puntoy}
\MULTIPLY{#6}{\puntox}{\puntox}
\MULTIPLY{#6}{\puntoy}{\puntoy}
\SUBTRACT{0}{\puntox}{\direcciony}
\COPY{\puntoy}{\direccionx}
\norma{\direccionx}{\direcciony}
\DIVIDE{\direccionx}{\norm}{\direccionx}
\DIVIDE{\direcciony}{\norm}{\direcciony}
\COPY{\direcciony}{\ortogonalx}
\COPY{\direccionx}{\ortogonaly}
\SUBTRACT{0}{\ortogonaly}{\ortogonaly}
\changereferencesystem(\puntox,\puntoy)(\ortogonalx,\ortogonaly)(\direccionx,\direcciony)
}

\newcommand{\segmento}[4]{%
\MULTIPLY{#3}{#4}{\long}
\MULTIPLY{#1}{\long}{\dirx}
\MULTIPLY{#2}{\long}{\diry}
\xLINE(0,0)(\dirx,\diry)
\Put(\dirx,\diry){\circle*{0.08}}
\changereferencesystem(\dirx,\diry)(#2,-#1)(#1,#2)
}

\newcommand{\segmentosinpunto}[4]{%
\MULTIPLY{#3}{#4}{\long}
\MULTIPLY{#1}{\long}{\dirx}
\MULTIPLY{#2}{\long}{\diry}
\xLINE(0,0)(\dirx,\diry)
\changereferencesystem(\dirx,\diry)(#2,-#1)(#1,#2)
}

\newcommand{\segmentoblanco}[4]{%
\MULTIPLY{#3}{#4}{\long}
\MULTIPLY{#1}{\long}{\dirx}
\MULTIPLY{#2}{\long}{\diry}
\Put(\dirx,\diry){\circle*{0.04}}
\changereferencesystem(\dirx,\diry)(#2,-#1)(#1,#2)
}

\newcommand{\segmentoblancosinpunto}[4]{%
\MULTIPLY{#3}{#4}{\long}
\MULTIPLY{#1}{\long}{\dirx}
\MULTIPLY{#2}{\long}{\diry}
\changereferencesystem(\dirx,\diry)(#2,-#1)(#1,#2)
}

\setlength{\unitlength}{1.5cm}%

\begin{Picture}(-1,-1)(6,3)
\thicklines

\referencesystem(0,0)(1,1)(-1,1)

\COPY{1}{\rad}
\COPY{0.4}{\ang}
\DIVIDE{\ang}{2}{\angbis}
\DIVIDE{\ang}{3}{\angbisbis}

\COPY{0}{\taux}
\COPY{1}{\tauy}
\COPY{0}{\pux}
\COPY{0}{\puy}
\Put(\pux,\puy){\circle*{0.08}}
\Put(0.1,\puy){\tiny{$st_{j-1}$}}

\arco{\taux}{\tauy}{\pux}{\puy}{\ang}{\rad}

\arcosinpunto{0}{1}{0}{0}{\angbis}{\rad}
\arcoblanco{0}{1}{0}{0}{\angbis}{\rad}
\arcoblanco{0}{1}{0}{0}{\angbis}{\rad}
\arcoblanco{0}{1}{0}{0}{\angbis}{\rad}
\arcoblancosinpunto{0}{1}{0}{0}{\angbis}{\rad}
\arco{0}{1}{0}{0}{\angbis}{\rad}
\Put(0,0.1){\tiny{$\ell_j$}}

\segmento{1}{0}{\ang}{\rad}
\segmentosinpunto{0}{1}{\angbis}{\rad}
\segmentoblanco{0}{1}{\angbisbis}{\rad}
\segmentoblanco{0}{1}{\angbisbis}{\rad}
\segmentoblanco{0}{1}{\angbisbis}{\rad}
\segmentoblancosinpunto{0}{1}{\angbisbis}{\rad}
\segmento{0}{1}{\angbis}{\rad}

\segmento{-1}{0}{\ang}{\rad}
\segmentosinpunto{0}{1}{\angbis}{\rad}
\segmentoblanco{0}{1}{\angbisbis}{\rad}
\segmentoblanco{0}{1}{\angbisbis}{\rad}
\segmentoblanco{0}{1}{\angbisbis}{\rad}
\segmentoblancosinpunto{0}{1}{\angbisbis}{\rad}
\segmento{0}{1}{\angbis}{\rad}

\segmento{1}{0}{\ang}{\rad}
\segmentosinpunto{0}{1}{\angbis}{\rad}

\Put(-0.1,0.1){\circle*{0.04}}
\Put(-0.2,0.2){\circle*{0.04}}
\Put(-0.3,0.3){\circle*{0.04}}
\changereferencesystem(-0.4,0.4)(1,0)(0,1)

\segmento{0}{1}{\angbis}{\rad}
\segmentosinpunto{-1}{0}{\angbis}{\rad}

\segmentoblanco{0}{1}{\angbisbis}{\rad}
\segmentoblanco{0}{1}{\angbisbis}{\rad}
\segmentoblanco{0}{1}{\angbisbis}{\rad}
\segmentoblancosinpunto{0}{1}{\angbisbis}{\rad}
\segmento{0}{1}{\angbis}{\rad}

\Put(-0.2,-0.1){\tiny{$st_j$}}

\end{Picture}
\end{center}

\caption{Enriques diagram of the $j$th Puiseux pair of a divisorial valuation}
  \label{fig4}
\end{figure}

For an exceptional curve valuation $\nu$, we consider the sequence of dual graphs (respectively, Enriques diagrams) $\Gamma_{\nu}^{\bullet}=(\Gamma_{\nu_i})_{i\geq 1}$ (respectively, ${\mathcal E}_{\nu}^{\bullet}=({\mathcal E}_{\nu_i})_{i\geq 1}$), where $\nu_i$ is the divisorial valuation defined by the exceptional divisor $E_i$.

Let $r$ be the index such that $p_i\rightarrow p_r$ for all $i> r$. Then, for all $i\geq r+1$, $\Gamma_{\nu_{i+1}}$ is obtained from $\Gamma_{\nu_i}$ by the following procedure (see Sections 5 and 9 of \cite{spiv}): (1) delete the edge joining the vertices of the divisors $E_i$ and $E_r$, (2) add a new vertex (the one associated with $E_{i+1}$), and (3) add two edges joining the new vertex with those corresponding with $E_i$ and $E_r$. Also, for all $i\geq r+1$, the graph ${\mathcal E}_{\nu_{i+1}}$ is obtained from ${\mathcal E}_{\nu_i}$ adding a new vertex (the one corresponding with $p_{i+1}$) and an edge $p_ip_{i+1}$ orthogonal to the edge $p_rp_{r+1}$.

We define the \emph{dual graph} (respectively, the \emph{Enriques diagram}) of $\nu$, denoted also by $\Gamma_{\nu}$ (respectively, ${\mathcal E}_{\nu}$), as the (infinite) graph whose set of vertices is the union of the sets of vertices of $\Gamma_{\nu_i}$ (respectively, ${\mathcal E}_{\nu_i}$) for all $i\geq 1$, with the same labels, and such that two vertices are connected by an edge in $\Gamma_{\nu}$ (respectively, ${\mathcal E}_{\nu}$) if and only if they are connected by an edge in $\Gamma_{\nu_i}$ (respectively, ${\mathcal E}_{\nu_i}$) for all $i$ sufficiently large.

Notice that ${\mathcal E}_{\nu}$ is connected. However $\Gamma_{\nu}$ has two connected components and exactly one of them has infinitely many vertices.
In Figure \ref{fig5} we have depicted the dual graph of an exceptional curve valuation $\nu$. If $g$ stands for the maximum of the number of Puiseux pairs of the divisorial valuations $\nu_i$, $i\geq 1$, then for all $j$ such that $1\leq j\leq g-1$ we denote by $\Gamma_{\nu}^j$ the subgraph of $\Gamma_{\nu}$ such that $\Gamma_{\nu}^j=\Gamma_{\nu_i}^j$ for all $i$ large enough. Also, we denote by $\Gamma_{\nu}^g$ the subgraph of $\Gamma_{\nu}$ whose set of vertices is the union of the sets of vertices of $\Gamma_{\nu_i}^g$ for $i$ large enough and whose edges are those of $\Gamma_{\nu}$ joining the mentioned vertices.
Notice that all subgraphs $\Gamma_{\nu}^j$ are finite except the last one, $\Gamma_{\nu}^g$. As in the case of divisorial valuations, we will say that an exceptional divisor $E_i$ (or the point $p_i$) \emph{belongs to the $j$th Puiseux pair} of $\nu$ if its associated vertex is a vertex of $\Gamma_{\nu}^j$.

\begin{figure}[h]
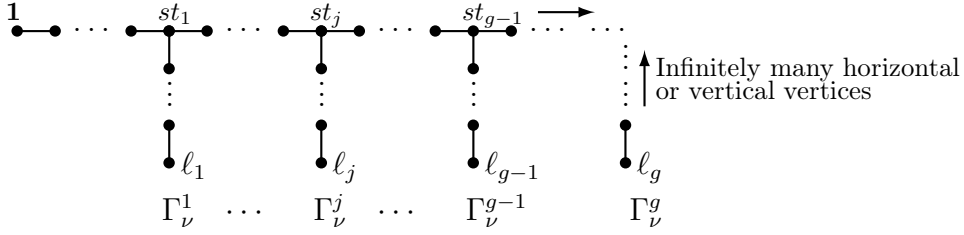


\begin{center}
\setlength{\unitlength}{0.5cm}%
\begin{Picture}(0,0)(24,8)
\thicklines


\xLINE(0,6)(1,6)
\Put(0,6){\circle*{0.3}}
\Put(1,6){\circle*{0.3}}
\put(1,6){$\;\;\ldots\;\;$}
\xLINE(3,6)(4,6)
\Put(3,6){\circle*{0.3}}
\Put(4,6){\circle*{0.3}}
\xLINE(4,6)(4,5)
\Put(4,5){\circle*{0.3}}
\Put(3.9,4){$\vdots$}
\xLINE(4,3.5)(4,2.5)
\Put(4,3.5){\circle*{0.3}}
\Put(4,2.5){\circle*{0.3}}
\Put(3.8,1){$\Gamma^1_\nu$}

\put(4.3,2.2){$\ell_1$}

\Put(3.7,6.3){\footnotesize $st_1$}
\Put(-0.3,6.3){\footnotesize $\mathbf{1}$}



\xLINE(4,6)(5,6)
\Put(4,6){\circle*{0.3}}
\Put(5,6){\circle*{0.3}}
\Put(5,6){$\;\;\ldots\;\;$}
\xLINE(7,6)(8,6)
\Put(7,6){\circle*{0.3}}
\Put(8,6){\circle*{0.3}}
\xLINE(8,6)(8,5)
\Put(8,5){\circle*{0.3}}
\Put(7.9,4){$\vdots$}
\xLINE(8,3.5)(8,2.5)
\Put(8,3.5){\circle*{0.3}}
\Put(8,2.5){\circle*{0.3}}
\Put(7.8,1){$\Gamma^j_\nu$}

\put(8.3,2.2){$\ell_j$}

\Put(7.7,6.3){\footnotesize $st_j$}

\Put(5,1){$\;\;\cdots\;\;$}


\xLINE(8,6)(9,6)
\Put(8,6){\circle*{0.3}}
\Put(9,6){\circle*{0.3}}
\put(9,6){$\;\;\ldots\;\;$}
\xLINE(11,6)(12,6)
\Put(11,6){\circle*{0.3}}
\Put(12,6){\circle*{0.3}}
\xLINE(12,6)(12,5)
\Put(12,5){\circle*{0.3}}
\Put(11.9,4){$\vdots$}
\xLINE(12,3.5)(12,2.5)
\Put(12,3.5){\circle*{0.3}}
\Put(12,2.5){\circle*{0.3}}
\Put(11.8,1){$\Gamma^{g-1}_\nu$}

\put(12.3,2.2){$\ell_{g-1}$}

\Put(11.7,6.3){\footnotesize $st_{g-1}$}

\Put(9,1){$\;\;\cdots\;\;$}


\xLINE(12,6)(13,6)
\Put(13,6){\circle*{0.3}}
\Put(13,6){$\;\;\ldots\;\;\ldots$}
\Put(15.9,5){$\vdots$}
\Put(15.9,4){$\vdots$}
\Put(16,3.5){\circle*{0.3}}
\xLINE(16,3.5)(16,2.5)
\Put(16,2.5){\circle*{0.3}}

\xVECTOR(16.5,4)(16.5,5.5)

\xVECTOR(13.7,6.5)(15.2,6.5)

\Put(16.8,4.8){\footnotesize Infinitely many horizontal }
\Put(16.8,4.2){\footnotesize or vertical vertices}

\put(16.3,2.2){$\ell_{g}$}
\Put(16.1,1){$\Gamma^{g}_\nu$}

\end{Picture}
\end{center}
  \caption{Dual graph of an exceptional curve valuation}
  \label{fig5}
\end{figure}

\subsection{Sequence of values}

Let $\nu$ be a divisorial or exceptional curve valuation of $F$ centered at $R$, let ${\mathcal C}_{\nu}=\{p_i\}_{i\geq 1}$ be its cluster of centers and, for each $i\geq 1$, let $\mathfrak{m}_i$ be the maximal ideal of the local ring $R_i={\mathcal O}_{X_i,p_i}$; write $\nu(\mathfrak{m}_i):=\min\{\nu(x)\mid x\in \mathfrak{m}_i\setminus \{0\}\}$. The sequence $\{\nu(\mathfrak{m}_i)\}_{i\geq 1}$ is called the \emph{sequence of values} of $\nu$, and its elements satisfy the \emph{proximity equalities} \cite[Theorem 8.1.7]{C}:
\begin{equation}\label{proximity}
\nu(\mathfrak{m}_i)=\sum_{p_j\rightarrow p_i} \nu(\mathfrak{m}_j),\;\;\; i\geq 1,
\end{equation}
whenever the set $\{p_j\in {\mathcal C}_{\nu}\mid p_j \rightarrow p_i\}$ is not empty.

In the case of exceptional curve valuations, when $p_i\rightarrow p_r$ for all $i> r$, the equality $\nu(\mathfrak{m}_r)=\sum_{p_j\rightarrow p_r} \nu(\mathfrak{m}_j)$ (which involves an infinite sum) must be understood as $\nu(\mathfrak{m}_r)>n\cdot \nu(\mathfrak{m}_{r+1})$ for any positive integer $n$. As a consequence, considering the value group included in $\mathbb{Z}^2$ (with lexicographical ordering), it holds that $\nu(\mathfrak{m}_r)=(a,b)$ and $\nu(\mathfrak{m}_i)=(0,c)$ for all $i>r$, where $a,b,c$ are integers, $a$ and $c$ being strictly positive. In fact, any choice of values $a,b,c$ with these conditions determines an exceptional curve valuation $\omega$ equivalent to $\nu$ such that ${\mathcal C}_{\omega}={\mathcal C}_{\nu}$.

\begin{definition}
\label{curveta}
{\rm
An $i$-\emph{curvette} is an analytically irreducible element of $R$, which gives rise to a germ of curve whose strict transform on $X_i$ is non-singular and meets transversally the divisor $E_i$ at a general point.
}
\end{definition}

Throughout the paper we will use the symbol $\varphi_i$ to denote an $i$-curvette.\\

Notice that, in the divisorial case, the proximity equalities (\ref{proximity}) imply that $\nu(\mathfrak{m}_i)={\rm mult}_{p_i}(\varphi_n)$ for all $i\geq 1$, where $E_n$ is the divisor defining $\nu$ and, for all $f\in R\setminus \{0\}$, ${\rm mult}_{p_i}(f)$ denotes the multiplicity of the strict transform in $R_i$ of the germ of curve with equation $f=0$.

\subsection{Noether formula}

Let $\nu$ be a divisorial or exceptional curve valuation. For every $f\in R\setminus \{0\}$, the value $\nu(f)$ can be computed using the so-called Noether formula \cite[Theorem 8.1.6]{C}:
\begin{equation}\label{Noether}
\nu(f)=\sum_{i\geq 1} {\rm mult}_{p_i}(f) \cdot \nu(\mathfrak{m}_i).
\end{equation}
A straightforward consequence of the above formula is that, when $\nu$ is divisorial, $\nu(f)$ coincides with the intersection multiplicity at $p$ of the germs of curve with equations $f=0$ and $\varphi_n=0$, where $E_n$ is the divisor defining $\nu$.

\begin{remark}\label{equivalentinf}
{\rm Assume that $\nu$ is exceptional curve and let ${\mathcal C}_{\nu}=\{p_i\}_{i\geq 1}$ and $r$ be as before. Notice that, due to (\ref{proximity}) and the Noether formula, the valuation $\nu$ is determined by the points $p_1, p_2, \ldots,p_{r+1}$ as well as by the values $\nu(\mathfrak{m}_r)$ and $\nu(\mathfrak{m}_{r+1})$. Moreover, for any linear automorphism $\psi:\mathbb{R}^2\rightarrow \mathbb{R}^2$, the composition $\psi\circ\nu$ defines a plane valuation that is equivalent to $\nu$ (in the sense of \cite[page 33]{z-s}) and whose value group is $\psi(\mathbb{Z}^2)$ (with the order relation that is induced, through $\psi$, by the lexicographical ordering in $\mathbb{Z}^2$). Furthermore, each plane valuation equivalent to $\nu$ with value group contained in $\mathbb{R}^2$ arises in this way \cite[page 49]{z-s}.
}
\end{remark}




\subsection{Maximal contact values and Puiseux exponents}\label{Puiseux}
Next we recall two families of invariants: the maximal contact values (for divisorial and exceptional curve valuations) and the Puiseux exponents (for divisorial valuations) (see \cite{deganu}).

Let $\nu$ be a divisorial or exceptional curve valuation and let $g$ be its number of Puiseux pairs. Define $\betabarra_0({\nu}):=\nu(\varphi_1)=\nu(\mathfrak{m}_1)$ and, for each $j\in \{1, 2, \ldots, g\}$, $\betabarra_j({\nu}):=\nu(\varphi_{\ell_j})$. If $\nu$ is divisorial, we define also $\betabarra_{g+1}({\nu}):=\nu(\varphi_n)$, where $E_n$ is the divisor defining $\nu$. The elements $\betabarra_0(\nu), \betabarra_1(\nu),\ldots$ are called \emph{maximal contact values} of $\nu$.


\begin{remark}
{\rm When $\nu$ is divisorial, the Noether formula (\ref{Noether}) proves that
$$\betabarra_{g+1}(\nu)=\sum_{i=1}^n \nu(\mathfrak{m}_i)^2.$$
Moreover, $$\dim_{\mathbb{C}}\; {\mathcal O}_{\mathbb{P}^2,p}/{\mathcal P}_{\alpha(k)}=\sum_{i=1}^n k \nu(\mathfrak{m}_i)(k \nu(\mathfrak{m}_i)+1)/2,$$ where $\alpha(k):=k\sum_{i=1}^n  \nu(\mathfrak{m}_i)^2$ for all positive integers $k$ (see \cite[4.7]{C}). This proves the equality
$$\betabarra_{g+1}(\nu)={\rm vol}(\nu)^{-1}.$$
}
\end{remark}

The \emph{Puiseux exponents} of $\nu$ are the rational numbers defined by:
\begin{equation}\label{Delta}
\beta'_{j}(\nu):=1+\frac{\betabarra_j(\nu)-n_{j-1}(\nu)\betabarra_{j-1}(\nu)}{e_{j-1}(\nu)},\;\;\;\; j=1, 2, \ldots, g+1,
\end{equation}
where $n_0(\nu)=1$, $e_0(\nu):=\betabarra_0(\nu)$ and, for all $j\geq 1$, set $$e_j(\nu):=\gcd(\betabarra_0(\nu), \betabarra_1(\nu), \ldots, \betabarra_j(\nu))$$ and  $n_j(\nu):=e_{j-1}(\nu)/e_j(\nu)$.

Consider an index $1 \leq j \leq g$; then the continued fraction expansion of the Puiseux exponent $\beta'_j(\nu)$ determines (and is determined by) the Enriques diagram of the $j$th Puiseux pair of $\nu$. Indeed, writing
$$[a_0,a_1,\ldots,a_s]:=a_0 + \frac{1}{a_1 +
\frac{1}{\ddots a_s}}$$ for non-negative integers $a_0, a_1, \ldots, a_s$ such that $a_i\neq 0$ if $i>0$, it holds that the integers $m_l$ in the expansion
$\beta'_j(\nu)=[m_0,m_1,\ldots,m_t]$ are the number of consecutive points $p_k$ in the $j$th Puiseux pair with the same value $\nu(\mathfrak{m}_k)$.
Moreover $\beta'_{g+1}(\nu)$ is the number of vertices in the \emph{tail} of $\nu$ plus 1. This allows us to get the Enriques diagram of $\nu$ from its Puiseux exponents. 
Conversely, an analogous reasoning yields that the continued fraction expansion of the values $\beta'_j(\nu)$ of a divisorial valuation $\nu$ can be computed from its Enriques diagram ${\mathcal E}_{\nu}$. In fact, for all $k>1$, it holds that $\nu(\mathfrak{m}_k)\neq \nu(\mathfrak{m}_{k-1})$ if and only if the edge $p_kp_{k+1}$ (if any) is orthogonal to the edge $p_{k-1}p_k$.

To sum up, the Enriques diagram ${\mathcal E}_{\nu}$ is a union of maximal smooth paths such that there is a one-to-one correspondence between the set of (different) values in $\{\nu(\mathfrak{m_i})\}$ and this set of smooth paths (see the forthcoming Example \ref{enric}). Moreover,
for each $j=1, 2, \ldots, g$ (respectively, $g+1$), $\beta'_j(\nu)$ has a continued fraction expansion $[m_0,m_1,\ldots,m_t]$, where $t+1$ is the number of maximal smooth paths in the subgraph of ${\mathcal E}_{\nu}$ whose vertices correspond to divisors $E_i$ belonging to the $j$th Puiseux pair of ${\nu}$ (respectively, 1). Let us denote the number $t$ by $\sigma_j(\nu)$.

\begin{example}\label{enric}
Let $\nu$ be a divisorial valuation whose Enriques diagram is that depicted in Figure \ref{fig7}; there we have labeled the vertices with the values $\{\nu(\mathfrak{m}_i)\}_{i=1}^{10}$. Then $\betabarra_0(\nu)=24$, $\betabarra_1(\nu)=57$ and $\betabarra_2(\nu)=458$, $\betabarra_3(\nu)=1374$, and
$$\beta'_1(\nu)=\frac{57}{24}=[2,2,1,2],\;\;\; \beta'_2(\nu)=\frac{5}{3}=[1,1,2],\;\;\; \beta'_3(\nu)=1.$$
\end{example}

\begin{figure}
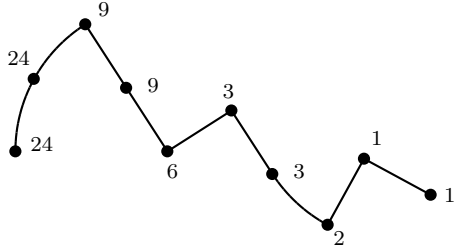

\begin{center}

\newcommand{\norma}[2]{%
\SQUARE{#1}{\m}
\SQUARE{#2}{\n}
\ADD{\m}{\n}{\k}
\SQUAREROOT{\k}{\norm}
}


\newcommand{\arco}[6]{%
\norma{#1}{#2}
\DIVIDE{#2}{\norm}{\tangentey}
\DIVIDE{#1}{\norm}{\tangentex}
\COPY{\tangentey}{\ortogonalx}
\COPY{\tangentey}{\centrox}
\SUBTRACT{0}{\tangentex}{\ortogonaly}
\MULTIPLY{\centrox}{#6}{\centrox}
\MULTIPLY{\ortogonaly}{#6}{\centroy}
\ADD{#3}{\centrox}{\centrox}
\ADD{#4}{\centroy}{\centroy}
\changereferencesystem(\centrox,\centroy)(\ortogonalx,\ortogonaly)(\tangentex,\tangentey)
\radiansangles
\SUBTRACT{\numberPI}{#5}{\angulo}
\ellipticArc{#6}{#6}{\angulo}{\numberPI}
\COS{\angulo}{\puntox}
\SIN{\angulo}{\puntoy}
\MULTIPLY{#6}{\puntox}{\puntox}
\MULTIPLY{#6}{\puntoy}{\puntoy}
\Put(\puntox,\puntoy){\circle*{0.08}}
\SUBTRACT{0}{\puntox}{\direcciony}
\COPY{\puntoy}{\direccionx}
\norma{\direccionx}{\direcciony}
\DIVIDE{\direccionx}{\norm}{\direccionx}
\DIVIDE{\direcciony}{\norm}{\direcciony}
\COPY{\direcciony}{\ortogonalx}
\COPY{\direccionx}{\ortogonaly}
\SUBTRACT{0}{\ortogonaly}{\ortogonaly}
\changereferencesystem(\puntox,\puntoy)(\ortogonalx,\ortogonaly)(\direccionx,\direcciony)
}

\newcommand{\arcosinpunto}[6]{%
\norma{#1}{#2}
\DIVIDE{#2}{\norm}{\tangentey}
\DIVIDE{#1}{\norm}{\tangentex}
\COPY{\tangentey}{\ortogonalx}
\COPY{\tangentey}{\centrox}
\SUBTRACT{0}{\tangentex}{\ortogonaly}
\MULTIPLY{\centrox}{#6}{\centrox}
\MULTIPLY{\ortogonaly}{#6}{\centroy}
\ADD{#3}{\centrox}{\centrox}
\ADD{#4}{\centroy}{\centroy}
\changereferencesystem(\centrox,\centroy)(\ortogonalx,\ortogonaly)(\tangentex,\tangentey)
\radiansangles
\SUBTRACT{\numberPI}{#5}{\angulo}
\ellipticArc{#6}{#6}{\angulo}{\numberPI}
\COS{\angulo}{\puntox}
\SIN{\angulo}{\puntoy}
\MULTIPLY{#6}{\puntox}{\puntox}
\MULTIPLY{#6}{\puntoy}{\puntoy}
\SUBTRACT{0}{\puntox}{\direcciony}
\COPY{\puntoy}{\direccionx}
\norma{\direccionx}{\direcciony}
\DIVIDE{\direccionx}{\norm}{\direccionx}
\DIVIDE{\direcciony}{\norm}{\direcciony}
\COPY{\direcciony}{\ortogonalx}
\COPY{\direccionx}{\ortogonaly}
\SUBTRACT{0}{\ortogonaly}{\ortogonaly}
\changereferencesystem(\puntox,\puntoy)(\ortogonalx,\ortogonaly)(\direccionx,\direcciony)
}

\newcommand{\arcoblanco}[6]{%
\norma{#1}{#2}
\DIVIDE{#2}{\norm}{\tangentey}
\DIVIDE{#1}{\norm}{\tangentex}
\COPY{\tangentey}{\ortogonalx}
\COPY{\tangentey}{\centrox}
\SUBTRACT{0}{\tangentex}{\ortogonaly}
\MULTIPLY{\centrox}{#6}{\centrox}
\MULTIPLY{\ortogonaly}{#6}{\centroy}
\ADD{#3}{\centrox}{\centrox}
\ADD{#4}{\centroy}{\centroy}
\changereferencesystem(\centrox,\centroy)(\ortogonalx,\ortogonaly)(\tangentex,\tangentey)
\radiansangles
\SUBTRACT{\numberPI}{#5}{\angulo}
\COS{\angulo}{\puntox}
\SIN{\angulo}{\puntoy}
\MULTIPLY{#6}{\puntox}{\puntox}
\MULTIPLY{#6}{\puntoy}{\puntoy}
\Put(\puntox,\puntoy){\circle*{0.04}}
\SUBTRACT{0}{\puntox}{\direcciony}
\COPY{\puntoy}{\direccionx}
\norma{\direccionx}{\direcciony}
\DIVIDE{\direccionx}{\norm}{\direccionx}
\DIVIDE{\direcciony}{\norm}{\direcciony}
\COPY{\direcciony}{\ortogonalx}
\COPY{\direccionx}{\ortogonaly}
\SUBTRACT{0}{\ortogonaly}{\ortogonaly}
\changereferencesystem(\puntox,\puntoy)(\ortogonalx,\ortogonaly)(\direccionx,\direcciony)
}

\newcommand{\arcoblancosinpunto}[6]{%
\norma{#1}{#2}
\DIVIDE{#2}{\norm}{\tangentey}
\DIVIDE{#1}{\norm}{\tangentex}
\COPY{\tangentey}{\ortogonalx}
\COPY{\tangentey}{\centrox}
\SUBTRACT{0}{\tangentex}{\ortogonaly}
\MULTIPLY{\centrox}{#6}{\centrox}
\MULTIPLY{\ortogonaly}{#6}{\centroy}
\ADD{#3}{\centrox}{\centrox}
\ADD{#4}{\centroy}{\centroy}
\changereferencesystem(\centrox,\centroy)(\ortogonalx,\ortogonaly)(\tangentex,\tangentey)
\radiansangles
\SUBTRACT{\numberPI}{#5}{\angulo}
\COS{\angulo}{\puntox}
\SIN{\angulo}{\puntoy}
\MULTIPLY{#6}{\puntox}{\puntox}
\MULTIPLY{#6}{\puntoy}{\puntoy}
\SUBTRACT{0}{\puntox}{\direcciony}
\COPY{\puntoy}{\direccionx}
\norma{\direccionx}{\direcciony}
\DIVIDE{\direccionx}{\norm}{\direccionx}
\DIVIDE{\direcciony}{\norm}{\direcciony}
\COPY{\direcciony}{\ortogonalx}
\COPY{\direccionx}{\ortogonaly}
\SUBTRACT{0}{\ortogonaly}{\ortogonaly}
\changereferencesystem(\puntox,\puntoy)(\ortogonalx,\ortogonaly)(\direccionx,\direcciony)
}

\newcommand{\segmento}[4]{%
\MULTIPLY{#3}{#4}{\long}
\MULTIPLY{#1}{\long}{\dirx}
\MULTIPLY{#2}{\long}{\diry}
\xLINE(0,0)(\dirx,\diry)
\Put(\dirx,\diry){\circle*{0.08}}
\changereferencesystem(\dirx,\diry)(#2,-#1)(#1,#2)
}

\newcommand{\segmentosinpunto}[4]{%
\MULTIPLY{#3}{#4}{\long}
\MULTIPLY{#1}{\long}{\dirx}
\MULTIPLY{#2}{\long}{\diry}
\xLINE(0,0)(\dirx,\diry)
\changereferencesystem(\dirx,\diry)(#2,-#1)(#1,#2)
}

\newcommand{\segmentoblanco}[4]{%
\MULTIPLY{#3}{#4}{\long}
\MULTIPLY{#1}{\long}{\dirx}
\MULTIPLY{#2}{\long}{\diry}
\Put(\dirx,\diry){\circle*{0.04}}
\changereferencesystem(\dirx,\diry)(#2,-#1)(#1,#2)
}

\newcommand{\segmentoblancosinpunto}[4]{%
\MULTIPLY{#3}{#4}{\long}
\MULTIPLY{#1}{\long}{\dirx}
\MULTIPLY{#2}{\long}{\diry}
\changereferencesystem(\dirx,\diry)(#2,-#1)(#1,#2)
}

\setlength{\unitlength}{2cm}%

\begin{Picture}(0,-0.5)(3,2)
\thicklines

\referencesystem(0,0)(1,0)(0,1)

\COPY{1}{\rad}
\COPY{0.5}{\ang}
\DIVIDE{\ang}{2}{\angbis}
\DIVIDE{\ang}{3}{\angbisbis}

\COPY{0}{\taux}
\COPY{1}{\tauy}
\COPY{0}{\pux}
\COPY{0}{\puy}
\Put(\pux,\puy){\circle*{0.08}}
\Put(0.1,\puy){\tiny{$24$}}

\arco{\taux}{\tauy}{\pux}{\puy}{\ang}{\rad}\Put(-0.2,0){\tiny{$24$}}
\arco{0}{1}{0}{0}{\ang}{\rad}\Put(0,0.1){\tiny{$9$}}
\segmento{1}{0}{\ang}{\rad}\Put(-0.1,0.1){\tiny{$9$}}
\segmento{0}{1}{\ang}{\rad}\Put(0.1,0.15){\tiny{$6$}}
\segmento{-1}{0}{\ang}{\rad}\Put(-0.1,0){\tiny{$3$}}
\segmento{1}{0}{\ang}{\rad}\Put(-0.1,0.1){\tiny{$3$}}

\changereferencesystem(0,0)(-1,0)(0,1)
\arco{0}{1}{0}{0}{\ang}{\rad}\Put(-0.1,0.1){\tiny{$2$}}
\segmento{1}{0}{\ang}{\rad}\Put(0,0.1){\tiny{$1$}}
\segmento{-1}{0}{\ang}{\rad}\Put(0,0.1){\tiny{$1$}}

\end{Picture}
\end{center}

\caption{Enriques diagram in Example \ref{enric}}
  \label{fig7}
\end{figure}

\subsection{Intersection numbers between curvettes}
\label{computation}

In this section we provide formulae to compute the intersection multiplicity at $p$ of two curvettes in terms of the maximal contact values of the divisorial valuations of their defining exceptional divisors $E_i$. This result will be essential in the proof of Theorem \ref{triangulo}.

Fix a divisorial valuation $\nu$, consider its cluster of centers ${\mathcal C}_{\nu}=\{p_i\}_{i=1}^n$ and the dual graph $\Gamma_{\nu}$. For each index $k\in \{1, 2, \ldots, n\}$ we denote by $\rho(k)$ the maximum natural number $m$ such that the strict transform of $\varphi_k$ passes through all points of ${\mathcal C}_{\nu}$ in the $m$th Puiseux pair of $\nu$. Now we state the above mentioned result on intersection multiplicity:

\begin{proposition}\label{felix}
Let $i,j$ be two indices in $\{1, 2, \ldots,n\}$ such that $i<j$, and let $\nu_i$ and $\nu_j$ be the divisorial valuations defined by the divisors $E_i$ and $E_j$. Then:
\begin{itemize}
\item[(a)] If $E_i$ is free and $i< \ell_{\rho(i)+1}$ we have that
$$
(\varphi_i,\varphi_j)_p= e_{\rho(i)-1}(\nu_i) \bar{\beta}_{\rho(i)}(\nu_j) + d e_{\rho(i)} (\nu_i) e_{\rho(i)}(\nu_j)$$
$$= e_{\rho(i)-1}(\nu_j) \bar{\beta}_{\rho(i)}(\nu_i) + d e_{\rho(i)}(\nu_i) e_{\rho(i)}(\nu_j),
$$
where $d$ is the length of the path $[st_{\rho(i)}+1,i]$  (in the dual graph $\Gamma_{\nu}$) and $e_{-1}(\nu_i)=e_{-1}(\nu_j):=0$.
\item[(b)] Otherwise
$$(\varphi_i,\varphi_j)_p=\min \{e_{\rho(i)}(\nu_i) \bar{\beta}_{\rho(i) +1}(\nu_j), e_{\rho(i)} (\nu_j)\bar{\beta}_{\rho(i) +1}(\nu_i) \}$$
and, moreover, $\frac{\bar{\beta}_{\rho(i) +1}(\nu_i)}{e_{\rho(i)}(\nu_i)}\leq \frac{\bar{\beta}_{\rho(i) +1}(\nu_j)}{e_{\rho(i)}(\nu_j)}$ if and only if $j\not\preccurlyeq i$ (with equality when $\rho(i)<\rho(j)$).
\end{itemize}
\end{proposition}

\begin{proof}
This result, except for the last part of (b), can be deduced from \cite[page 362]{del} adapting the notation to our purposes.

For the remaining part notice that, by (\ref{Delta}) the inequality $\frac{\bar{\beta}_{\rho(i) +1}(\nu_i)}{e_{\rho(i)}(\nu_i)}\leq \frac{\bar{\beta}_{\rho(i) +1}(\nu_j)}{e_{\rho(i)}(\nu_j)}$ holds if and only if $\beta'_{\rho(i) +1}(\nu_i)\leq \beta'_{\rho(i) +1}(\nu_j)$; and this is the case if and only if $j\not \preccurlyeq i$ (with equality when $\rho(i)<\rho(j)$). Indeed, this follows after considering the construction of the dual graph $\Gamma_{\nu}$ from the Enriques diagram of $\nu$ (see \cite[Proposition 4.2.2]{C} and the paragraphs before Example \ref{enric}), and using the continued fraction expansions of the Puiseux exponents (see Section \ref{Puiseux}) and well-known properties of continued fractions.
\end{proof}

\section{The Newton-Okounkov body of an exceptional curve valuation}

\subsection{Newton-Okounkov bodies. General setting}

Let $X$ be a smooth irreducible projective variety of dimension $n$ over the complex numbers $\mathbb{C}$ and let $K(X)$ be its function field. Consider a flag of subvarieties, i.e., a sequence of subvarieties of $X$
$$
Y_{\bullet}:=\left\{ X=Y_0 \supset Y_1 \supset Y_2 \supset \cdots \supset Y_n = \{q\} \right\}
$$
such that each $Y_i$ is a smooth irreducible subvariety of codimension $i$ in $X$. The point $ q \in X$ is called the \emph{center} of the flag.

A discrete valuation $ \nu_{Y_{\bullet}}$ of rank $n$ may be associated to the flag $Y_{\bullet}$  as follows. First, let $g_i=0$ be the equation of $Y_i$ in $Y_{i-1}$ in a Zariski open set containg $q$, which is possible since $Y_i$ has codimension $i$. Then, for $f \in K(X)$, define
$$
\upsilon_1 (f):=\mathrm{ord}_{Y_1}(f), \;\;\; f_1=\left.\frac{f}{g_1^{v_1(f)}}\right|_{Y_1}
$$
and, for $2 \leq i \leq n$,
$$
\upsilon_i (f):=\mathrm{ord}_{Y_i}(f_{i-1}), ~ \mbox{ where } f_{i}:=(f_{i-1}/g_{i}^{\upsilon_{i}(f)})|_{Y_i}.
$$

The map $\nu_{Y_{\bullet}}:K(X) \setminus \{0\} \to \mathbb{Z}^n_{\mathrm{lex}}$ is defined by the sequence of maps $\upsilon_i$, $1 \leq i \leq n$ as $\nu_{Y_{\bullet}}:=(\upsilon_1, \upsilon_2, \ldots, \upsilon_n)$. It is a discrete valuation of rank $n$, and any valuation of maximal rank comes from a flag \cite[Theorem 2.9]{cil}.

\begin{definition}
Given a flag  $Y_{\bullet}$ and a big divisor $D$ on $X$, the \emph{Newton-Okounkov body} of $D$ with respect to $Y_{\bullet}$ (or $ \nu_{Y_{\bullet}}$) is defined to be the following subset of $\mathbb{R}^n$:
$$
\Delta_{Y_{\bullet}}(D) = \Delta_{\nu_{Y_{\bullet}}}(D) :=\overline{\bigcup_{m \geq 1}\left \{ \frac{\nu_{Y_{\bullet}} (f)}{m} \;|\;  f \in H^0(X, m D)  \setminus \{0\} \right \}},
$$
where $\overline{\{~\cdot~\}}$ stands for the closed convex hull of the set $\{~\cdot~\}$.
\end{definition}

Newton-Okounkov bodies are convex and compact sets with nonempty interior. In \cite{KLM} it is proved that $\Delta_{Y_{\bullet}}(D)$ is a polygon if $X$ is a surface, and that in higher dimensions it can be non-polyhedral. Moreover,
\begin{equation}\label{dimension}
\mathrm{vol}_X(D) = n! \; \mathrm{vol}_{\mathbb{R}^n} \left(\Delta_{Y_{\bullet}}(D)\right),
\end{equation}
where $\mathrm{vol}_{\mathbb{R}^n} $ means Euclidean volume and
\[
\mathrm{vol}_X(D):=\lim_{m \to \infty} \frac{h^0\left(X, \mathcal{O}_X(m D)\right)}{m^n/n!}.
\]

\subsection{Exceptional curve valuations as flag valuations}\label{elflag}

From now on in this paper, we will consider any surface $X=X_r$ defined by a finite sequence of blow-ups of points as in (\ref{seq}) and its corresponding divisorial valuation $\nu_r$. Our goal is to study the Newton-Okounkov body of the flag
\[
E_{\bullet}:= \left\{ X=X_r \supset E_r \supset \{q:=p_{r+1}\} \right\}.
\]
This flag determines a cluster of centers of blowing-ups ${\mathcal C}_{\nu}=\{p_i\}_{i=1}^{\infty}$, where $\{p_i\}_{i=1}^{r+1}$ are given directly by the flag and the remaining points satisfy $p_i \rightarrow p_r$ for all $i >r$. It is well known that ${\mathcal C}_{\nu}$ defines an exceptional curve valuation $\nu$ (up to equivalence of valuations). Also, when $q$ is not free, we set $\eta$ for the index such that $\eta\neq r$ and $p_{r+1}\in E_{\eta}$.

As already mentioned, the flag $E_{\bullet}$ defines a \emph{flag valuation} $\nu_{E_{\bullet}}$ such that,
for $f \in R = \mathcal{O}_{\mathbb{P}^2,p}$, it holds that $\nu_{E_{\bullet}}(f)=(\upsilon_1(f),\upsilon_2(f))$ with $\upsilon_1(f):=\nu_{r}(f)$ (where $\nu_r$ is the divisorial valuation defined by $E_r$) and $\upsilon_2(f):=\mathrm{ord}_q \left(\pi_r^{\ast}(f)/z_r^{\upsilon_1(f)}\right)$, where $\pi_r:X_r\to X_0$ is the composition of the first $r$ point blowing-ups with centers in ${\mathcal C}_{\nu}$, $z_r=0$ a local equation for $E_r$ and $\pi^{\ast}(f)/z_r^{\upsilon_1(f)}$ is seen as a function on $E_r$. Notice that
$$
\upsilon_2(f)=\left(\pi^{\ast}(f)/z_r^{\upsilon_1(f)},E_r \right)_q,
$$
where, as above, $(\cdot~,\cdot)_q$ stands for the intersection multiplicity at $q$.

\begin{proposition}
Under the above notation it holds that, for all $f\in R\setminus \{0\}$, $$\upsilon_2(f)=\nu_{\eta}(f)+\sum_{p_i\rightarrow p_r}{\rm mult}_{p_i}({f}),$$ where $\nu_{\eta}$ denotes the divisorial valuation defined by $E_{\eta}$.
\end{proposition}

\begin{proof}
Let $f\in R\setminus\{0\}$. It is clear that $\upsilon_2(f)=\nu_{\eta}(f)+ {\rm mult}_{q}(f)$. Then the result follows from the proximity equalities for germs of plane curves \cite[Theorem 3.5.3]{C}.
\end{proof}

As a consequence of the above proposition and the Noether formula (\ref{Noether}), one can deduce the following:

\begin{corollary}\label{ppp}
Let $\nu$ be an exceptional curve valuation and let $E_\bullet$ be the flag defined by ${\mathcal C}_{\nu}$. Denote also by $\nu'$ the unique valuation equivalent to $\nu$ whose value group is $\mathbb{Z}_{lex}^2$ and such that $\nu'(\mathfrak{m}_r)=(1,0)$ and $\nu'(\mathfrak{m}_{r+1})=(0,1)$. Then $\nu'=\nu_{E_{\bullet}}$.
\end{corollary}


In spite of their importance, very few explicit examples of Newton-Okounkov bodies can be found in the literature.
We are interested in an explicit computation of the Newton-Okounkov bodies $\Delta_{\nu_{E_{\bullet}}}(H) = \Delta_{\nu_{E_{\bullet}}}$ of flags $E_{\bullet}$ as before, with respect to the  divisor $H$ given by the pull-back of the line-bundle $\mathcal{O}_{\mathbb{P}^2} (1)$. Therefore the rest of the section is devoted to state and prove results which provide an explicit description of these bodies.  Corollary \ref{ppp} allows us to associate them to arbitrary exceptional curve valuations.  In fact, we can give a slightly more general definition:

\begin{definition}
{\rm
The Newton-Okounkov body of an exceptional curve valuation $\nu$ that takes values in $\mathbb{R}^2$ is defined as
$$
\Delta_{\nu}:=\overline{\bigcup_{m \geq 1}\left \{ \frac{\nu (f)}{m} \;|\;  f \in H^0(X, {\mathcal O}_{\mathbb{P}^2}(m))  \setminus \{0\} \right \}}.
$$
}
\end{definition}

If $\nu(\mathfrak{m}_r)=(1,0)$ and $\nu(\mathfrak{m}_{r+1})=(0,1)$ we get the equality $\Delta_{\nu}=\Delta_{\nu_{E_{\bullet}}}$ by Corollary \ref{ppp}, i.e., $\Delta_{\nu}$ is the Newton-Okounkov body of the flag $E_{\bullet}$ with respect to the line bundle $H$.  Moreover notice that, by Remark \ref{equivalentinf}, if $\nu'$ is a valuation equivalent to $\nu$ with value group in $\mathbb{R}^2$, then there exists a linear automorphism $\psi:\mathbb{R}^2\rightarrow \mathbb{R}^2$ such that $\Delta_{\nu'}=\psi(\Delta_{\nu})$.

Therefore, \emph{from now on, unless otherwise be stated, we will assume that every considered exceptional curve valuation $\nu$ satisfies these conditions: its value group is $\mathbb{Z}_{lex}^2$, $\nu(\mathfrak{m}_r)=(1,0)$ and $\nu(\mathfrak{m}_{r+1})=(0,1)$}.

Newton-Okounkov bodies of exceptional curve valuations $\nu$ with only one Puiseux pair and whose last divisor $E_r$ is satellite were studied in \cite{cil}.

\subsection{The Newton-Okounkov body of a minimal exceptional curve valuation}
\label{lacuatro}

Let $(X:Y:Z)$ be projective coordinates in $\mathbb{P}^2$ and assume that $p=(1:0:0)$. Consider affine coordinates $u=Y/X$ and $v=Z/X$ around $p$ and let $\nu$ be an exceptional curve valuation centered at $R={\mathcal O}_{\mathbb{P}^2,p}$. \emph{We will keep this framework for all valuations that we will consider from now on}.

Let ${\mathcal C}_{\nu}=\{p_1=p,p_2,\ldots\}$ be the cluster of centers of $\nu$ and, keeping the notation of the preceding sections, let $p_r$ be the point such that $p_i\rightarrow p_r$ for all $i>r$.

An analogue of the Seshadri constant for valuations and line bundles on normal projective varieties was introduced in \cite{BKMS}. For the divisorial valuation $\nu_r$ and the line bundle $H$, this analogue can be defined as
\[
\hat{\mu}(\nu_r) : = \lim_{d \rightarrow \infty} \frac{\mu_d (\nu_r)}{d},
\]
where $\mu_d (\nu_r) = \max \{\nu_r(f) \;| \; f \in \mathbb{C}[u,v], \; \mathrm{deg}(f) \leq d\}$. 

It is known that $\hat{\mu}(\nu_r) \geq \sqrt{1/\mathrm{vol} (\nu_r)}$ (see \cite{BKMS}). The divisorial valuation $\nu_r$ is called to be \emph{minimal} if $\hat{\mu}(\nu_r) = \sqrt{1/\mathrm{vol} (\nu_r)}$ (see \cite{d-h-k-r-s} and \cite{g-m-m} for further information about minimal valuations).

\begin{definition}\label{minimalvaluation}
{\rm 
An exceptional curve valuation $\nu$ is called to be \emph{minimal} whenever its first component $\nu_r$ is minimal.
}
\end{definition}

Let us denote by $S_{\nu}$ the \emph{semigroup of values} of $\nu$, that is,
$$S_{\nu}:=\{\nu(f)\mid f\in R\setminus \{0\}\}\subseteq \mathbb{Z}^2,$$
endowed with the lexicographical ordering. Let $\mathfrak{C}(\nu)$ be the convex cone of $\mathbb{R}^2$ spanned by $S_{\nu}$ and $\mathfrak{H}(\nu)$ be the half-plane $\{(x,y)\in \mathbb{R}^2\mid x\leq \hat{\mu}(\nu_r)\}$. The following result yields a description of $\mathfrak{C}(\nu)\cap \mathfrak{H}(\nu)$. Let $g$ be the number of Puiseux pairs of the divisorial valuation $\nu_r$; notice that $\nu$ has $g$ (respectively, $g+1$) Puiseux pairs whenever $q=p_{r+1}$ is satellite (respectively, free).

\begin{proposition}\label{tri}
The set $\mathfrak{C}(\nu)\cap \mathfrak{H}(\nu)$ is a triangle whose vertices are
\begin{itemize}
\item[$\diamond$] The points in the plane $\mathbb{R}^2$
$$(0,0),\;\left(\hat{\mu}(\nu_r), \frac{\hat{\mu}(\nu_r)\betabarra_0(\nu_{\eta}) }{\betabarra_0(\nu_r)} \right)\; \mbox{and}\; \left(\hat{\mu}(\nu_r), \frac{\hat{\mu}(\nu_r)\betabarra_g(\nu_{\eta})}{\betabarra_g(\nu_r)} \right),$$
whenever $q= p_{r+1}$ is a satellite point and belongs to the intersection $E_r\cap E_{\eta}$ (with $\eta\neq r$).
\item[$\diamond$] The points
$$(0,0),\;\left(\hat{\mu}(\nu_r), 0 \right)\; \mbox{and}\; \left(\hat{\mu}(\nu_r), \frac{\hat{\mu}(\nu_r)}{\betabarra_{g+1}(\nu_r)} \right),$$
otherwise (i.e., when $q$ is free).
\end{itemize}

\end{proposition}
\begin{proof}

To begin with, let us observe that the semigroup $S_{\nu}$ is spanned by the maximal contact values $\betabarra_0({\nu}), \betabarra_1({\nu}), \ldots,\betabarra_g({\nu})$ (respectively, $\betabarra_0({\nu}), \betabarra_1({\nu}), \ldots,\betabarra_{g+1}({\nu})$) if $q=p_{r+1}$ is satellite (respectively, free) \cite{spiv}.


Now, taking into account that $\nu(\mathfrak{m}_r)=(1,0)$ and $\nu(\mathfrak{m}_{r+1})=(0,1)$, the proximity equalities (\ref{proximity}) and the Noether formula (\ref{Noether}), one  deduces that, if $q$ is satellite, then
$\betabarra_j({\nu})=(\betabarra_j(\nu_r),\betabarra_j(\nu_{\eta}))$  for all $j\in \{0, 1, \ldots,g\}$;  moreover, if $B$ stands for the line joining the origin with $\betabarra_0({\nu})$, it holds that $\betabarra_j({\nu})\in B$ for all $j<g$ and $\betabarra_g({\nu})\not\in B$ (see \cite{deganu}).

 Finally, if $q$ is free, one can deduce similarly that $\betabarra_j({\nu})=(\betabarra_j(\nu_r),0)$  for all $j\in \{0, 1, \ldots,g\}$ and $\betabarra_{g+1}(\nu)=(\betabarra_{g+1}(\nu_r),1)$. This concludes the proof by considering the properties of the maximal contact values of $\nu_r$ and the slopes of the lines bounding $\mathfrak{C}(\nu)$.
\end{proof}

We finish this section describing the Newton-Okounkov bodies of minimal  exceptional curve valuations $\nu$ and proving that this description characterizes this class of valuations. We will need a previous lemma, which is straightforward.

\begin{lemma}\label{cont}
The Newton-Okounkov body $\Delta_{\nu}$ is contained in the set $\mathfrak{C}(\nu)\cap \mathfrak{H}(\nu)$. Moreover, each side of the boundary of this set contains a vertex of $\Delta_{\nu}$.
\end{lemma}

\begin{theorem}
\label{nobminimal}
The Newton-Okounkov body $\Delta_\nu$ coincides with the triangle $\mathfrak{C}(\nu)\cap \mathfrak{H}(\nu)$ if and only if $\nu$ is minimal.
\end{theorem}
\begin{proof}
Proposition \ref{tri} proves  that the area of the triangle $\mathfrak{C}(\nu)\cap \mathfrak{H}(\nu)$ is $$\frac{\hat{\mu}(\nu_r)^2}{2}\left| \frac{\betabarra_g(\nu_{\eta})}{\betabarra_g(\nu_r)}- \frac{\betabarra_0(\nu_{\eta})}{\betabarra_0(\nu_r)}\right|$$ (respectively, $\hat{\mu}(\nu_r)^2\;/\;2\betabarra_{g+1}(\nu_r)$) if $q=p_{r+1}$ is satellite and $q \in E_r\cap E_{\eta}$ (respectively, $q$ is free).

The Newton-Okounkov body $\Delta_\nu$ is contained in $\mathfrak{C}(\nu)\cap \mathfrak{H}(\nu)$ by Lemma \ref{cont}. The next lemma shows that $\hat{\mu}(\nu_r)^2\;/\;2\betabarra_{g+1}(\nu_r)$ is also the area of $\mathfrak{C}(\nu)\cap \mathfrak{H}(\nu)$ in the satellite case and this concludes the proof because this area is larger than or equal to $1/2$ (the area of $\Delta_{\nu}$ by (\ref{dimension})), being equal to $1/2$ exactly when $\nu_r$ is minimal.
\end{proof}

\begin{lemma}
\label{45}
Keeping the above notation, if the point $p_{r+1}$ is satellite, then
$$\betabarra_{g+1}(\nu_r)=\left| \frac{\betabarra_g(\nu_{\eta})}{\betabarra_g(\nu_r)}- \frac{\betabarra_0(\nu_{\eta})}{\betabarra_0(\nu_r)}\right|^{-1}.$$
\end{lemma}
\begin{proof}
With the notation of Subsection \ref{elflag},  consider the divisorial valuations $\nu_r$ and $\nu_{\eta}$.  We will prove that
\begin{equation}
\label{laclave}
|e_{g-1}(\nu_{\eta}) \betabarra_g(\nu_r) - e_{g-1}(\nu_r) \betabarra_g(\nu_{\eta})| =1,
\end{equation}
which allows us to conclude the result. Indeed, when $$e_{g-1}(\nu_{\eta}) \betabarra_g(\nu_r) - e_{g-1}(\nu_r) \betabarra_g(\nu_\eta)= 1,$$
it holds that
\[
\betabarra_0(\nu_\eta) \betabarra_g(\nu_r) - \betabarra_0(\nu_r) \betabarra_g(\nu_\eta) = \frac{e_{g-1}(\nu_{\eta}) \betabarra_0(\nu_r) }{e_{g-1}(\nu_r) } \betabarra_g(\nu_r) - \betabarra_0(\nu_r) \betabarra_g(\nu_\eta) =
\]
\[
\frac{1}{e_{g-1}(\nu_r)} \left(e_{g-1}(\nu_{\eta}) \betabarra_g(\nu_r) \betabarra_0(\nu_r) - \betabarra_0(\nu_r) e_{g-1}(\nu_r) \betabarra_g(\nu_\eta)\right) = \frac{\betabarra_0(\nu_r)}{e_{g-1}(\nu_r)}.
\]
This proves that
\[
\left(\frac{\betabarra_0(\nu_{\eta})}{\betabarra_0(\nu_r)} - \frac{\betabarra_g(\nu_{\eta})}{\betabarra_g(\nu_r)}\right)^{-1} = \frac{\betabarra_0(\nu_r) \betabarra_g(\nu_r)}{\betabarra_0(\nu_\eta) \betabarra_g(\nu_r) - \betabarra_0(\nu_r) \betabarra_g(\nu_\eta)}
\]
\[ = \frac{\betabarra_0(\nu_r) \betabarra_g(\nu_r) e_{g-1}(\nu_r)}{\betabarra_0(\nu_r)} = \betabarra_{g+1}(\nu_r).
\]
An analogous reasoning shows the case when $$e_{g-1}(\nu_{\eta}) \betabarra_g(\nu_r) - e_{g-1}(\nu_r) \betabarra_g(\nu_\eta)= -1.$$

We finish proving the equality (\ref{laclave}).
By (\ref{Delta}) we get
\[
\frac{\betabarra_g(\nu_i)}{e_{g-1}(\nu_i)} = (\beta'_g(\nu_i) -1) + \frac{n_{g-1}(\nu_i)}{e_{g-1}(\nu_i)} \betabarra_{g-1}(\nu_i),
\]
where $i$ is either $r$ or $\eta$. From the definition of the values $\betabarra_j(\nu_i)$, we get that
\[
\frac{n_{g-1}(\nu_r)}{e_{g-1}(\nu_r)} \betabarra_{g-1}(\nu_r) = \frac{n_{g-1}(\nu_{\eta})}{e_{g-1}(\nu_{\eta})} \betabarra_{g-1}(\nu_{\eta}).
\]
As a consequence,
\begin{equation}\label{Chebichev1}
\frac{\betabarra_g(\nu_r)}{e_{g-1}(\nu_r)} - \frac{\betabarra_g(\nu_{\eta})}{e_{g-1}(\nu_{\eta})} = \beta'_g(\nu_r)-\beta'_g(\nu_{\eta}).
\end{equation}
With the help of the Enriques diagram of $\nu_r$ and the explanations given in the paragraphs before Example \ref{enric}, it is easily checked that, with the notations in those paragraphs,  $\sigma_g(\nu_r)=\sigma_g(\nu_{\eta})+1$.
This concludes the proof because of the equality
\begin{equation}\label{Chebichev2}
\beta'_g(\nu_r)-\beta'_g(\nu_{\eta}) = \frac{(-1)^{\sigma_g(\nu_{\eta})}}{e_{g-1}(\nu_r) e_{g-1}(\nu_{\eta})},
\end{equation}
which can be deduced from \cite[Theorem 7.5]{niven} and the equality above. \end{proof}

\subsection{The Newton-Okounkov body of a non-minimal exceptional curve valuation}

We have just found the vertices of the Newton-Okounkov body of a minimal exceptional curve valuation. This subsection is devoted to the description of the non-minimal case.
Let us start by proving the existence of supraminimal curves for divisorial valuations. This fact was proved in \cite{d-h-k-r-s} for the case when $\nu_r$ has only one Puiseux pair and it is defined by a satellite divisor.

\begin{lemma}\label{supra}
Let $\nu$ be a divisorial valuation of $F$ centered at $R$ and assume the existence of an irreducible polynomial $f\in \mathbb{C}[u,v]$ such that $\nu(f)>\sqrt{\betabarra_{g+1}(\nu)}\deg(f)$.  Then
\[
\frac{\nu(f)}{\deg (f)} = \hat{\mu}(\nu).
\]
Moreover, if $\nu$ is not a minimal valuation then there exists such an irreducible polynomial $f$ and it is the unique irreducible polynomial (up to product by a non-zero constant) satisfying the above condition.
\end{lemma}

\begin{proof}
To show the first part of the statement, suppose the existence of the mentioned polynomial $f$. Then it is enough to prove the following assertion: if $h\in \mathbb{C}[u,v]$ satisfies that $\nu(h)>\sqrt{\betabarra_{g+1}(\nu)}\deg(h)$, then $f$ is a component of $h$. Indeed,  if the assertion is true, then $\nu(f) = \deg(f) \hat{\mu}(\nu)$, since otherwise there would exist $h \in \mathbb{C}[u,v]$ such that
\[
\sqrt{\betabarra_{g+1}(\nu)} < \frac{\nu(f)}{\deg(f)} < \frac{\nu(h)}{\deg(h)}
\]
and therefore $\nu(h) > \sqrt{\betabarra_{g+1}(\nu)} \deg(h)$. By the assertion, $h= f^a h_1$, where $a>0$ and $h_1 \in \mathbb{C}[u,v]$ such that $f$ does not divide $h_1$. Then
\[
\frac{\nu(h)}{\deg(h)} = \frac{a \nu(f) + \nu(h_1)}{a \deg(f) + \deg{h_1}} > \frac{\nu(f)}{\deg(f)},
\]
which, again by the assertion would imply that $f$ divides $h_1$, a contradiction.

Now we are going to prove the assertion. Let $k\in \mathbb{N}$ be any common multiple of $\nu(h)$ and $\betabarra_{g+1}(\nu)$. Unloading procedure (see \cite{C}) shows that the ideal ${\mathcal P}_{k}:=\{\varphi\in R\mid \nu(\varphi)\geq k\}$ is equal to $$\pi_*\left({\mathcal O}_{X}(-\frac{k}{\betabarra_{g+1}(\nu)}\sum_{i=1}^r m_i E_i^*) \right),$$ where $m_i=\nu(\mathfrak{m}_i)$ for all $i$, and $E_i^*$ stands for the total transform on $X$ of the divisor $E_i$. Moreover, since the divisor $-\sum_{i=1}^r m_i E_i^*$ is nef, there exists a non-empty Zariski subset $U$ of ${\mathcal P}_k$ such that ${\rm mult}_{p_i}(\varphi)= \frac{k}{\betabarra_{g+1}(\nu)} m_i$, for all $i=1, 2, \ldots,r$ and for all $\varphi\in U$ (see page 60 and Corollary 4.2.4 of \cite{C}). By \cite[Lemma 7.2.1]{C} we can also assume that $(\varphi,f)_p=w$ for all $\varphi\in U$, where $w:=\min \{ (\psi,f)_p\mid \psi\in {\mathcal P}_{k}\}$. Therefore
$$(h^{k/\nu(h)},f)_p \geq w\geq \sum_{i=1}^r \frac{k}{\betabarra_{g+1}(\nu)} m_i\cdot {\rm mult}_{p_i}(f) =\frac{k}{\betabarra_{g+1}(\nu)}\nu(f)>\frac{k}{\sqrt{\betabarra_{g+1}(\nu)}}\deg(f).$$
Hence
$$(h,f)_p>\frac{\nu(h)}{\sqrt{\betabarra_{g+1}}}\deg(f)>\deg(f)\deg(h)$$
and we conclude that $f$ is a component of $h$.

Finally, assume that $\nu$ is not minimal. Then, there exists a polynomial $h\in \mathbb{C}[u,v]$ such that $\nu(h)>\sqrt{\betabarra_{g+1}}\deg(h)$ and, therefore, the fact that the valuation of a product is the sum of valuations of the factors shows that, at least an irreducible component $f$ of $h$, satisfies  $\nu(f)>\sqrt{\betabarra_{g+1}}\deg(f)$. This proves the existence of the irreducible polynomial satisfying the condition given in the statement. Its uniqueness is a straightforward consequence of the assertion at the beginning of the proof.

\end{proof}

\begin{definition}
{\rm Let $\nu$ be a divisorial valuation that is not minimal. Then a  curve $C$ defined by an irreducible polynomial $f\in \mathbb{C}[u,v]$ such that $\nu(f) / \deg (f) = \hat{\mu}(\nu)$ is called \emph{supraminimal}. Notice that we have just proved that this curve is unique.
}
\end{definition}

Next, preserving the notation of Subsection \ref{elflag}, we will describe the Newton-Okounkov bodies of the flags $E_\bullet$ (i.e., of exceptional curve valuations $\nu$).
\begin{theorem}
\label{53}
Let $\nu$ be a non-minimal exceptional curve  valuation, let $g$ be the number of Puiseux pairs of the divisorial valuation of $\nu_r$ attached to $\nu$, and let $C$ be the supraminimal curve of $\nu_r$. Moreover, assume that $C$ is defined by $f=0$ and write $\nu(f)/\deg (f) = \left(\hat{\mu} (\nu_r), c \right)$. Then, the Newton-Okounkov body $\Delta_{\nu}$ is the convex hull of the set $\{(0,0), Q_1, Q_2, Q_3\}$, where the coordinates of the points $Q_i$ are as follows:
\begin{itemize}
\item[(a)] If $q= p_{r+1}\in E_r\cap E_{\eta}$ is a satellite point, then
\[ Q_1=\frac{1}{\hat{\mu}(\nu_r)}(\betabarra_{g+1}(\nu_r), \nu_{r}(\varphi_{\eta})), \; \; Q_2=\frac{1}{\hat{\mu}(\nu_r)}(\betabarra_{g+1}(\nu_r), \nu_{r}(\varphi_{\eta})+1)\] and
$Q_3=(\hat{\mu}(\nu_r),c)$.
\item[(b)] Otherwise (i.e., if $p_{r+1}$ is free), we have
    \[
    Q_1=\frac{1}{\hat{\mu}(\nu_r)}(\betabarra_{g+1}(\nu_r), 0), \;\; Q_2=\frac{1}{\hat{\mu}(\nu_r)}(\betabarra_{g+1}(\nu_r),1) \; {\mbox and } \; Q_3=(\hat{\mu}(\nu_r),c).
    \]
\end{itemize}
Therefore $\Delta_{\nu}$ is either a triangle or a quadrilateral.
\end{theorem}

\begin{proof}
For a start, we notice that the existence of the supraminimal curve $C$ is a consequence of Lemma \ref{supra} and therefore $Q_3\in \Delta_{\nu}$.

By Theorem 6.4 of \cite{LM} and preserving notation, it holds that
$$\Delta_{\nu}=\{(t,y)\in \mathbb{R}^2\mid 0\leq t\leq \mu\;\mbox{ and }\; \alpha(t)\leq y\leq \beta(t)\},$$
where $\mu:=\sup \{s>0\mid H-sE_r\; \mbox{is big}\}$ and, for all $t\in [0,\mu]$, it is $\alpha(t):={\rm ord}_{p_{r+1}}(N_t\mid_{{E}_r})$, $\beta(t):=\alpha(t)+P_t\cdot {E}_r$; here $P_t+N_t$ stands for the Zariski decomposition of the $\mathbb{Q}$-divisor $H-tE_r$. Notice that $\mu=\hat{\mu}(\nu_r)$ because the big cone and the cone of curves of $X$ have the same closure.

Set $$t_0:=\frac{\betabarra_{g+1}(\nu_r)}{\hat{\mu}(\nu_r)}.$$
The Zariski decomposition of the $\mathbb{Q}$-divisor $H-t_0E_r$ is $P_{t_0}+N_{t_0}$, where $$P_{t_0}=H- \frac{1}{\hat{\mu}(\nu_r)}\sum_{i=1}^r \nu_r(\varphi_i)E_i$$ and $$N_{t_0}=\frac{1}{\hat{\mu}(\nu_r)}\sum_{i=1}^{r-1} \nu_r(\varphi_i)E_i.$$
Notice that $P_{t_0}=H-\frac{1}{\hat{\mu}(\nu_r)}\sum_{i=1}^r \nu_r(\mathfrak{m}_i) {E}^*_i$, hence it is nef. Indeed, on the one hand, if $D$ is the strict transform on $X_r$ of a curve of $\mathbb{P}^2$ defined by a polynomial $h\in \mathbb{C}[u,v]$ then $$P_{t_0}\cdot D=\deg(h)-\frac{1}{\hat{\mu}(\nu_r)} \sum_{i=1}^r \nu_r(\mathfrak{m}_i)\cdot {\rm mult}_{p_i}(h)=\deg(h)-\frac{1}{\hat{\mu}(\nu_r)} \cdot \nu_r(h)\geq 0,$$ where the second equality is a consequence of the Noether formula (\ref{Noether}); on the other hand $P_{t_0}\cdot E_i= 0$ for all $i=1, 2, \ldots,r-1$ by the proximity equalities (\ref{proximity}).

Now, when the point $p_{r+1}$ is satellite (respectively, free), it holds that
\[
\alpha(t_0)=\frac{1}{\hat{\mu}(\nu_r)}\nu_{r}(\varphi_{\eta})=\frac{1}{\hat{\mu}(\nu_r)}\nu_{\eta}(\varphi_r) \;\; \mbox{ (respectively, $\alpha(t_0)=0$)}
\]
 and $$
 \beta(t_0)=\alpha(t_0)+P_{t_0}\cdot E_r=\alpha(t_0)+\frac{1}{\hat{\mu}(\nu_r)}=
 \frac{1}{\hat{\mu}(\nu_r)}(\nu_{\eta}(\varphi_r)+1)$$ (respectively,
$\beta(t_0)=P_{t_0}\cdot E_r=\frac{1}{\hat{\mu}(\nu_r)}$).

Therefore, we have proved that $Q_1,Q_2\in \Delta_{\nu}$. This concludes the proof since the area of the convex hull of $\{(0,0),Q_1,Q_2,Q_3\}$ is $1/2$.
\end{proof}

To conclude this section, for every exceptional curve valuation $\nu$ we determine an equivalent valuation $\nu'$ whose Newton-Okounkov body has a side on the $X$-axis. First, we will need the following result.

\begin{lemma}
\label{el412}
Let $\nu$ be an exceptional curve valuation. Let $g^*$ (respectively, $g$) be the number of Puiseux pairs of $\nu$ (respectively, $\nu_r$). Then there exists a unique valuation $\nu'$ equivalent to $\nu$ with values in $\mathbb{R}^2$ such that $\betabarra_0(\nu')=(1,0)$ and $\betabarra_{g^*}(\nu')=(\betabarra_{g^*}(\nu_r)/\betabarra_0(\nu_r),1)$. Moreover it holds that
\begin{itemize}
\item[(a)] if $p_{r+1}$ is satellite, then
$\nu'(\mathfrak{m}_r)=(1/\betabarra_0(\nu_r), (-1)^{a-1} e_{g-1}(\nu_{\eta}))$ and $\nu'(\mathfrak{m}_{r+1})=(0, (-1)^{a} e_{g-1}(\nu_r))$,
where $a=1$ (respectively, $a=2$) if $\eta \preccurlyeq r$ (respectively, $\eta \not\preccurlyeq r$).
\item[(b)] if $p_{r+1}$ is free, then
$\nu'(\mathfrak{m}_r)=(1/\betabarra_0(\nu_r), 0)$ and $\nu'(\mathfrak{m}_{r+1})=(0, 1)$.
\end{itemize}
\end{lemma}

\begin{proof}
First of all, notice that $g^*=g$ (respectively, $g^*=g+1$) when $p_{r+1}$ is satellite (respectively, free).

If $p_{r+1}$ is satellite then $\betabarra_j({\nu})=(\betabarra_j(\nu_r),\betabarra_j(\nu_{\eta}))$ for $j\in \{0,g\}$ (see the proof of Proposition \ref{tri}) and  straightforward computations show that the unique linear automorphism $\psi: \mathbb{R}^2\rightarrow \mathbb{R}^2$ such that $\psi(\betabarra_0(\nu))=(1,0)$ and $\psi(\betabarra_{g^*}(\nu))=(\betabarra_{g^*}(\nu_r)/\betabarra_0(\nu_r),1)$  is defined by $x\mapsto xA$, where $A$ is the matrix
$$\begin{pmatrix} \frac{1}{\betabarra_0(\nu_r)} & \frac{\betabarra_0(\nu_{\eta})}{\delta} \\
0 & \frac{-\betabarra_0(\nu_{r})}{\delta}
 \end{pmatrix},
$$
where $\delta$ is $\betabarra_g(\nu_{r})\betabarra_0(\nu_{\eta})-\betabarra_g(\nu_{\eta})
\betabarra_0(\nu_{r})$. Therefore, the valuation $\nu'=\psi\circ \nu$ is the unique valuation equivalent to $\nu$ satisfying the required conditions (see Remark \ref{equivalentinf}).

Moreover $$\nu'(\mathfrak{m}_r)=\left(\frac{1}{\betabarra_0(\nu_r)},
\frac{\betabarra_0(\nu_{\eta})}{\delta}\right) \; \mbox{and} \; \nu'(\mathfrak{m}_{r+1})=\left(0,\frac{-\betabarra_0(\nu_{r})}{\delta}\right).$$ Notice that
$$\delta=e_{g-1}(\nu_r)\betabarra_0(\nu_{\eta})\left(\frac{\betabarra_g(\nu_{r})}{e_{g-1}(\nu_r)}-\frac{\betabarra_g(\nu_{\eta})}{e_{g-1}(\nu_{\eta})}\right)=e_{g-1}(\nu_{\eta})\betabarra_0(\nu_{r})\left(\frac{\betabarra_g(\nu_{r})}{e_{g-1}(\nu_r)}-\frac{\betabarra_g(\nu_{\eta})}{e_{g-1}(\nu_{\eta})}\right)$$
since $\betabarra_0(\nu_r)/e_{g-1}(\nu_r)=\betabarra_0(\nu_{\eta})/e_{g-1}(\nu_{\eta})$. Then, by (\ref{Chebichev1}) and (\ref{Chebichev2}),
$$\delta=\frac{\betabarra_0(\nu_{\eta})}{ e_{g-1}(\nu_{\eta})}(-1)^{\sigma_g(\nu_{\eta})}=\frac{\betabarra_0(\nu_{r})}{e_{g-1}(\nu_r) }(-1)^{\sigma_g(\nu_{\eta})}.$$
Hence, Part (a) follows from Proposition \ref{felix}.

If $p_{r+1}$ is free, taking into account that $\betabarra_0({\nu})=(\betabarra_0(\nu_r),0)$ and $\betabarra_{g+1}(\nu)=(\betabarra_{g+1}(\nu_r),1)$ (see the proof of Proposition \ref{tri}) and applying a similar reasoning as before, one can deduce the existence of the unique valuation $\nu'$ as in Part (b).
\end{proof}

The following result describes the Newton-Okounkov bodies of valuations $\nu'$ as in the previous lemma, showing that they have a side on the $X$-axis.

\begin{theorem}
\label{normalizacion}
Let $E_{\bullet}$ a flag  as above and let $\nu = \nu_{E_\bullet}$ be its attached exceptional curve  valuation. Let $\Delta_{\nu'} $ be the Newton-Okounkov body  of the  valuation $\nu'$ equivalent to $\nu$ defined in Lemma \ref{el412}.
\begin{itemize}
\item[$\diamond$] If $\nu$ is minimal, then $\Delta_{\nu'} $ is a triangle with vertices $(0,0)$,
\[
\left( \frac{\hat{\mu} (\nu_r)}{\betabarra_0(\nu_r)}, 0 \right) \; \mbox{ and } \; \hat{\mu} (\nu_r) \left( \frac{1}{\betabarra_0(\nu_r)}, \frac{1}{\betabarra_{g^*}(\nu_r)} \right),
 \]
where $g^*$ stands for the number of Puiseux pairs of $\nu$.
\item[$\diamond$] Otherwise ($\nu$ is not minimal),  $\Delta_{\nu'}$ is the convex hull of the following points $(0,0)$,
    \[
    \frac{1}{\hat{\mu} (\nu_r)}\left(\frac{\betabarra_{g+1}(\nu_r)}{\betabarra_0(\nu_r)}, e_{g^*-1}(\nu_r)\right), \;\;
    \frac{1}{\hat{\mu} (\nu_r)}\left(\frac{\betabarra_{g+1}(\nu_r)}{\betabarra_0(\nu_r)}, 0\right),
     \]
     and $\left(\hat{\mu} (\nu_r) / \betabarra_0(\nu_r), \alpha\right)$, where $0 \leq \alpha \leq \frac{\hat{\mu} (\nu_r)  e_{g^*-1}(\nu_r)}{\betabarra_{g+1}(\nu_r)}$ and  $g^* $ is the number of Puiseux pairs of $\nu$.
\end{itemize}
\end{theorem}
\begin{proof}
The case of $\nu$ minimal follows after computing the images of the defining pairs of the Newton-Okounkov bodies $\Delta_\nu$, described in Proposition \ref{tri} and Theorem \ref{53}, by the linear automorphism appearing in the proof of Lemma \ref{el412}. Let us explain what happens in the non-minimal case when $q=p_{r+1}$ is not free and $\eta \preccurlyeq r$. The remaining  cases can be shown analogously. According to Theorem \ref{53}, it is
$Q_1=\frac{1}{\hat{\mu}(\nu_r)}\left(\betabarra_{g+1}(\nu_r), \nu_{r}(\varphi_{\eta})\right)$; by Proposition \ref{felix} this can be expressed as
\[
\frac{\betabarra_{g+1}(\nu_r) }{\hat{\mu}(\nu_r) \betabarra_{g}(\nu_r)} \left( \betabarra_{g}(\nu_r), \betabarra_{g} (\nu_\eta)  \right),
\]
which multiplied by the matrix
$$
A=\begin{pmatrix} \frac{1}{\betabarra_0(\nu_r)} & e_{g-1}(\nu_{\eta}) \\
0 & - e_{g-1}(\nu_{r})
 \end{pmatrix},
$$
gives
\[
\frac{1}{\hat{\mu}(\nu_r)} \left(\frac{\betabarra_{g+1}(\nu_r)}{\betabarra_{0}(\nu_r)}, e_{g-1}(\nu_r) \right)
\]
by (\ref{laclave}) and the fact that $e_{g-1}(\nu_r) \betabarra_{g}(\nu_r) = \betabarra_{g+1}(\nu_r)$. This provides our first non-vanishing vertex.

With respect to $Q_2=\frac{1}{\hat{\mu}(\nu_r)}(\betabarra_{g+1}(\nu_r), \nu_{r}(\varphi_{\eta})+1)$, again by Proposition \ref{felix} and by Lemma \ref{45}, it holds that
\[
Q_2= \frac{\betabarra_{g+1}(\nu_r)}{\hat{\mu}(\nu_r)} \left(1, \frac{\betabarra_{0}(\nu_\eta)}{\betabarra_{0}(\nu_r)} \right),
\]
which, after multiplication by $A$, gives the point $ \frac{1}{\hat{\mu}(\nu_r)} \left(\frac{\betabarra_{g+1}(\nu_r)}{\betabarra_{0}(\nu_r)}, 0 \right)$.

Finally, the equality $(\hat{\mu}(\nu_r),c) A = (\hat{\mu}(\nu_r)/\betabarra_{0}^r, \hat{\mu}(\nu_r) e_{g-1} (\nu_\eta) - c e_{g-1}(\nu_r))$ holds and gives the last vertex, which completes the proof.
\end{proof}

\begin{remark}
\label{315}
{\rm The Newton-Okounkov body $\Delta_{\nu'}$ for $\nu'$ as above when $\nu$ has only one Puiseux pair and $p_{r+1}$ is satellite was described in \cite[Corollary 5.8]{cil} by a very different procedure.}
\end{remark}

\section{Characterization of triangular Newton-Okounkov bodies of exceptional curve valuations}

Keep the above notation and recall that the Newton-Okounkov body $\Delta_{\nu}$ of an exceptional curve valuation $\nu$ is either a quadrilateral or a triangle. We have seen that $\Delta_{\nu}$ is  a triangle when $\nu$ is minimal.
In this section, we assume that $\nu$ \emph{is not minimal} and we characterize the cases when $\Delta_{\nu}$ is also a triangle.

For a start, let $\Gamma_\nu$ be the (infinite) dual graph of $\nu$ (see Subsection \ref{32}) and let $C$ be the supraminimal curve of $\nu_r$. Define $\Gamma_{\nu}^*$ as the graph obtained from $\Gamma_{\nu}$ by adding a new vertex (associated to $C$) which is  joined (by an edge) with those vertices of $\Gamma_{\nu}$ given by divisors $E_i$ such that, for all $k$ large enough, the strict transforms of $E_i$ and $C$ in the surface $X_k$ have non-empty intersection.

\begin{lemma}\label{curvetas}
Let $\nu$ be a non-minimal exceptional curve valuation, let $h$ be an analytically irreducible element of $R$ and let $i_0$ be the maximum of the set of indices $i$ such that the strict transform of the germ defined by $h=0$ passes through $p_i$. Then, either there exists a positive integer $a$ such that $\nu(h)=a\;\nu(\varphi_{i_0})$ or there exist positive integers $a,b$ and $j_0$ such that
\begin{itemize}
\item[(a)] $\nu(h)=a\;\nu(\varphi_{i_0})+b\;\nu(\varphi_{j_0})$ and
\item[(b)] the vertices of $\Gamma_{\nu}$ associated with the divisors $E_{i_0}$ and $E_{j_0}$ belong to the same connected component of $\Gamma_{\nu}$.
\end{itemize}
\end{lemma}

\begin{proof}
Let ${\mathcal C}_h=\{q_i\}_{i=1}^{\infty}$ be the \emph{cluster of centers of $h$}, that is, the set of infinitely near points such that $q_1=p$ and, for each $i\geq 2$, $q_i$ is the intersection point between the exceptional divisor $E_{i-1}$ and the strict transform of the germ of curve defined by $h$. Then ${\mathcal C}_{\nu}\cap {\mathcal C}_h=\{p_1, p_2, \ldots,p_{i_0}\}$. The proximity equalities for germs of curves \cite[Theorem 3.5.3]{C} yield that ${\rm mult}_{q_j}(h)=\sum_{q_k} {\rm mult}_{q_k}(h)=0$, where the summation runs over all points $q_k$ which are proximate to $q_j$.

For each $j\in \{1, 2, \ldots,i_0\}$ we define $e_j:={\rm mult}_{p_j}(h)-\sum_{p_i} {\rm mult}_{p_i}(h)$, where $p_i$ runs over the set of points $p_i\in {\mathcal C}_{\nu}\cap {\mathcal C}_h$ which are proximate to $p_j$. Notice that $e_j\geq 0$ for all $j=1, 2, \ldots,i_0$.

Let ${\mathcal P}$ be the complete ideal of $R$ given by the stalk at $p$ of the sheaf $$\pi_{*}\left( {\mathcal O}_{X_{i_0}}(-\sum_{j=1}^{i_0} {\rm mult}_{p_j}(h) E_j^*)\right),$$
where $\pi: X_{i_0} \rightarrow X_0$ is the composition of the blowing-ups at ${\mathcal C}_{\nu}\cap {\mathcal C}_h$. By Zariski's theory of complete ideals of a 2-dimensional regular local ring, the ideal ${\mathcal P}$ decomposes as a product of simple complete ideals in the form
$${\mathcal P}=\prod_{j=1}^{i_0} {\mathcal P}_j^{e_j},$$
where ${\mathcal P}_j$ is the simple complete ideal of $R$ defined by the divisor $E_j$.
Notice that $\nu(h)$ coincides with the valuation of a general element of ${\mathcal P}$. Since $\varphi_j$ is a general element of ${\mathcal P}_j$ for all $j$, it holds that
\begin{equation}\label{prrrt}
\nu(h)=\sum_{j\in \Omega} e_j \; \nu(\varphi_j),
\end{equation}
where $\Omega:=\{j\in \{1,2, \ldots,i_0\}\mid e_j> 0\}$. Notice that $i_0\in \Omega$. We will see that the cardinality $|\Omega|$ of $\Omega$ is at most 2 and that, when $|\Omega|=2$, the associated exceptional divisors correspond to vertices of $\Gamma_{\nu}$ lying on the same connected component.

Write $\mathfrak{q}:=q_{i_0+1}$. If $\mathfrak{q}$ is free it is clear that $\Omega=\{i_0\}$; therefore we assume that $\mathfrak{q}$ is satellite and consider the divisorial valuation $\nu_{i_0+1}$. It is sufficient to prove that either $\Omega=\{i_0\}$, or $\Omega=\{i_0,j_0\}$ with $i_0,j_0 \preccurlyeq i_0+1$ or
$i_0,j_0 \not\preccurlyeq i_0+1$ in the dual graph $\Gamma_{\nu_{i_0+1}}$.

Let $\tilde{\mathcal E}_{\nu_{i_0+1}}$ be the Enriques diagram obtained from ${\mathcal E}_{\nu_{i_0+1}}$ by adding a new vertex (associated with the point $\mathfrak{q}$) which is joined (by an edge) with $p_{i_0}$ (straight or curved according with the definition of Enriques diagram).

Suppose that $p_{i_0}$ is free, then the unique possibility for $\tilde{\mathcal E}_{\nu_{i_0+1}}$ is the one depicted in Figure \ref{figura10b} (notice that we are assuming that $\mathfrak{q}$ is satellite).

Otherwise, $p_{i_0}$ is satellite and then the unique possibilities for $\tilde{\mathcal E}_{\nu_{i_0+1}}$ are those depicted in Figure \ref{figura10a}, and those obtained replacing the edge with origin at $p_s$ by a curved edge and/or the edge $p_{i_0}p_{i_0+1}$ by a curved edge that is tangent to $p_{i_0-1}p_{i_0}$.

The proximity equalities for $h$ in each case prove that $|\Omega|\leq 2$. Moreover, when $\Omega=\{i_0,j_0\}$ ($j_0\neq i_0$), the position of the vertices associated with the divisors $E_{i_0}$ and $E_{j_0}$ in the dual graph $\Gamma_{\nu_{i_0+1}}$ shows that either $i_0,j_0 \preccurlyeq i_0+1$ or
$i_0,j_0 \not\preccurlyeq i_0+1$. This proves the lemma.

\begin{figure}
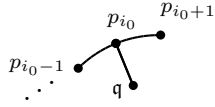

\begin{center}

\newcommand{\norma}[2]{%
\SQUARE{#1}{\m}
\SQUARE{#2}{\n}
\ADD{\m}{\n}{\k}
\SQUAREROOT{\k}{\norm}
}


\newcommand{\arco}[6]{%
\norma{#1}{#2}
\DIVIDE{#2}{\norm}{\tangentey}
\DIVIDE{#1}{\norm}{\tangentex}
\COPY{\tangentey}{\ortogonalx}
\COPY{\tangentey}{\centrox}
\SUBTRACT{0}{\tangentex}{\ortogonaly}
\MULTIPLY{\centrox}{#6}{\centrox}
\MULTIPLY{\ortogonaly}{#6}{\centroy}
\ADD{#3}{\centrox}{\centrox}
\ADD{#4}{\centroy}{\centroy}
\changereferencesystem(\centrox,\centroy)(\ortogonalx,\ortogonaly)(\tangentex,\tangentey)
\radiansangles
\SUBTRACT{\numberPI}{#5}{\angulo}
\ellipticArc{#6}{#6}{\angulo}{\numberPI}
\COS{\angulo}{\puntox}
\SIN{\angulo}{\puntoy}
\MULTIPLY{#6}{\puntox}{\puntox}
\MULTIPLY{#6}{\puntoy}{\puntoy}
\Put(\puntox,\puntoy){\circle*{0.08}}
\SUBTRACT{0}{\puntox}{\direcciony}
\COPY{\puntoy}{\direccionx}
\norma{\direccionx}{\direcciony}
\DIVIDE{\direccionx}{\norm}{\direccionx}
\DIVIDE{\direcciony}{\norm}{\direcciony}
\COPY{\direcciony}{\ortogonalx}
\COPY{\direccionx}{\ortogonaly}
\SUBTRACT{0}{\ortogonaly}{\ortogonaly}
\changereferencesystem(\puntox,\puntoy)(\ortogonalx,\ortogonaly)(\direccionx,\direcciony)
}

\newcommand{\arcosinpunto}[6]{%
\norma{#1}{#2}
\DIVIDE{#2}{\norm}{\tangentey}
\DIVIDE{#1}{\norm}{\tangentex}
\COPY{\tangentey}{\ortogonalx}
\COPY{\tangentey}{\centrox}
\SUBTRACT{0}{\tangentex}{\ortogonaly}
\MULTIPLY{\centrox}{#6}{\centrox}
\MULTIPLY{\ortogonaly}{#6}{\centroy}
\ADD{#3}{\centrox}{\centrox}
\ADD{#4}{\centroy}{\centroy}
\changereferencesystem(\centrox,\centroy)(\ortogonalx,\ortogonaly)(\tangentex,\tangentey)
\radiansangles
\SUBTRACT{\numberPI}{#5}{\angulo}
\ellipticArc{#6}{#6}{\angulo}{\numberPI}
\COS{\angulo}{\puntox}
\SIN{\angulo}{\puntoy}
\MULTIPLY{#6}{\puntox}{\puntox}
\MULTIPLY{#6}{\puntoy}{\puntoy}
\SUBTRACT{0}{\puntox}{\direcciony}
\COPY{\puntoy}{\direccionx}
\norma{\direccionx}{\direcciony}
\DIVIDE{\direccionx}{\norm}{\direccionx}
\DIVIDE{\direcciony}{\norm}{\direcciony}
\COPY{\direcciony}{\ortogonalx}
\COPY{\direccionx}{\ortogonaly}
\SUBTRACT{0}{\ortogonaly}{\ortogonaly}
\changereferencesystem(\puntox,\puntoy)(\ortogonalx,\ortogonaly)(\direccionx,\direcciony)
}

\newcommand{\arcoblanco}[6]{%
\norma{#1}{#2}
\DIVIDE{#2}{\norm}{\tangentey}
\DIVIDE{#1}{\norm}{\tangentex}
\COPY{\tangentey}{\ortogonalx}
\COPY{\tangentey}{\centrox}
\SUBTRACT{0}{\tangentex}{\ortogonaly}
\MULTIPLY{\centrox}{#6}{\centrox}
\MULTIPLY{\ortogonaly}{#6}{\centroy}
\ADD{#3}{\centrox}{\centrox}
\ADD{#4}{\centroy}{\centroy}
\changereferencesystem(\centrox,\centroy)(\ortogonalx,\ortogonaly)(\tangentex,\tangentey)
\radiansangles
\SUBTRACT{\numberPI}{#5}{\angulo}
\COS{\angulo}{\puntox}
\SIN{\angulo}{\puntoy}
\MULTIPLY{#6}{\puntox}{\puntox}
\MULTIPLY{#6}{\puntoy}{\puntoy}
\Put(\puntox,\puntoy){\circle*{0.04}}
\SUBTRACT{0}{\puntox}{\direcciony}
\COPY{\puntoy}{\direccionx}
\norma{\direccionx}{\direcciony}
\DIVIDE{\direccionx}{\norm}{\direccionx}
\DIVIDE{\direcciony}{\norm}{\direcciony}
\COPY{\direcciony}{\ortogonalx}
\COPY{\direccionx}{\ortogonaly}
\SUBTRACT{0}{\ortogonaly}{\ortogonaly}
\changereferencesystem(\puntox,\puntoy)(\ortogonalx,\ortogonaly)(\direccionx,\direcciony)
}

\newcommand{\arcoblancosinpunto}[6]{%
\norma{#1}{#2}
\DIVIDE{#2}{\norm}{\tangentey}
\DIVIDE{#1}{\norm}{\tangentex}
\COPY{\tangentey}{\ortogonalx}
\COPY{\tangentey}{\centrox}
\SUBTRACT{0}{\tangentex}{\ortogonaly}
\MULTIPLY{\centrox}{#6}{\centrox}
\MULTIPLY{\ortogonaly}{#6}{\centroy}
\ADD{#3}{\centrox}{\centrox}
\ADD{#4}{\centroy}{\centroy}
\changereferencesystem(\centrox,\centroy)(\ortogonalx,\ortogonaly)(\tangentex,\tangentey)
\radiansangles
\SUBTRACT{\numberPI}{#5}{\angulo}
\COS{\angulo}{\puntox}
\SIN{\angulo}{\puntoy}
\MULTIPLY{#6}{\puntox}{\puntox}
\MULTIPLY{#6}{\puntoy}{\puntoy}
\SUBTRACT{0}{\puntox}{\direcciony}
\COPY{\puntoy}{\direccionx}
\norma{\direccionx}{\direcciony}
\DIVIDE{\direccionx}{\norm}{\direccionx}
\DIVIDE{\direcciony}{\norm}{\direcciony}
\COPY{\direcciony}{\ortogonalx}
\COPY{\direccionx}{\ortogonaly}
\SUBTRACT{0}{\ortogonaly}{\ortogonaly}
\changereferencesystem(\puntox,\puntoy)(\ortogonalx,\ortogonaly)(\direccionx,\direcciony)
}

\newcommand{\segmento}[4]{%
\MULTIPLY{#3}{#4}{\long}
\MULTIPLY{#1}{\long}{\dirx}
\MULTIPLY{#2}{\long}{\diry}
\xLINE(0,0)(\dirx,\diry)
\Put(\dirx,\diry){\circle*{0.08}}
\changereferencesystem(\dirx,\diry)(#2,-#1)(#1,#2)
}

\newcommand{\segmentosinpunto}[4]{%
\MULTIPLY{#3}{#4}{\long}
\MULTIPLY{#1}{\long}{\dirx}
\MULTIPLY{#2}{\long}{\diry}
\xLINE(0,0)(\dirx,\diry)
\changereferencesystem(\dirx,\diry)(#2,-#1)(#1,#2)
}

\newcommand{\segmentoblanco}[4]{%
\MULTIPLY{#3}{#4}{\long}
\MULTIPLY{#1}{\long}{\dirx}
\MULTIPLY{#2}{\long}{\diry}
\Put(\dirx,\diry){\circle*{0.04}}
\changereferencesystem(\dirx,\diry)(#2,-#1)(#1,#2)
}

\newcommand{\segmentoblancosinpunto}[4]{%
\MULTIPLY{#3}{#4}{\long}
\MULTIPLY{#1}{\long}{\dirx}
\MULTIPLY{#2}{\long}{\diry}
\changereferencesystem(\dirx,\diry)(#2,-#1)(#1,#2)
}

\setlength{\unitlength}{1.5cm}%

\begin{Picture}(-0.5,-0.5)(2.2,2)
\thicklines

\referencesystem(0,0)(1,0)(0,1)

\COPY{1}{\rad}
\COPY{0.4}{\ang}
\DIVIDE{\ang}{2}{\angbis}
\DIVIDE{\ang}{3}{\angbisbis}
\COPY{0.7}{\angtris}

\COPY{0}{\taux}
\COPY{1}{\tauy}
\COPY{0}{\pux}
\COPY{7}{\puy}

\Put(-0.3,-0.3){\reflectbox{$\ddots$}}
\Put(0.2,0){\circle*{0.08}}
\Put(-0.4,0){\tiny{$p_{i_0-1}$}}

\arco{1}{1}{0.2}{0}{\ang}{\rad}
\Put(-0.2,0){\tiny{$p_{i_0}$}}
\segmento{1}{0}{\ang}{\rad}
\Put(0.2,0){\tiny{$\mathfrak{q}$}}
\segmento{0}{-1}{\ang}{\rad}
\arco{1}{0}{0}{0}{\ang}{\rad}
\Put(-0.2,0){\tiny{$p_{i_0+1}$}}



\end{Picture}
\end{center}

\caption{Graph $\tilde{\mathcal E}_{\nu_{i_0+1}}$ ($p_{i_0}$  free and $\mathfrak{q}$ satellite)}
  \label{figura10b}
\end{figure}

\begin{figure}
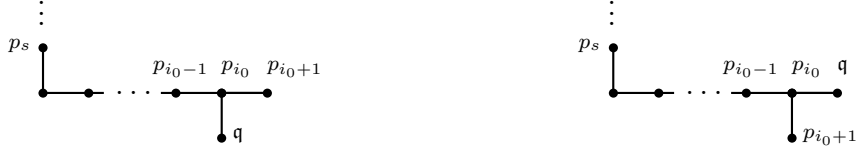

\begin{center}

\newcommand{\norma}[2]{%
\SQUARE{#1}{\m}
\SQUARE{#2}{\n}
\ADD{\m}{\n}{\k}
\SQUAREROOT{\k}{\norm}
}


\newcommand{\arco}[6]{%
\norma{#1}{#2}
\DIVIDE{#2}{\norm}{\tangentey}
\DIVIDE{#1}{\norm}{\tangentex}
\COPY{\tangentey}{\ortogonalx}
\COPY{\tangentey}{\centrox}
\SUBTRACT{0}{\tangentex}{\ortogonaly}
\MULTIPLY{\centrox}{#6}{\centrox}
\MULTIPLY{\ortogonaly}{#6}{\centroy}
\ADD{#3}{\centrox}{\centrox}
\ADD{#4}{\centroy}{\centroy}
\changereferencesystem(\centrox,\centroy)(\ortogonalx,\ortogonaly)(\tangentex,\tangentey)
\radiansangles
\SUBTRACT{\numberPI}{#5}{\angulo}
\ellipticArc{#6}{#6}{\angulo}{\numberPI}
\COS{\angulo}{\puntox}
\SIN{\angulo}{\puntoy}
\MULTIPLY{#6}{\puntox}{\puntox}
\MULTIPLY{#6}{\puntoy}{\puntoy}
\Put(\puntox,\puntoy){\circle*{0.08}}
\SUBTRACT{0}{\puntox}{\direcciony}
\COPY{\puntoy}{\direccionx}
\norma{\direccionx}{\direcciony}
\DIVIDE{\direccionx}{\norm}{\direccionx}
\DIVIDE{\direcciony}{\norm}{\direcciony}
\COPY{\direcciony}{\ortogonalx}
\COPY{\direccionx}{\ortogonaly}
\SUBTRACT{0}{\ortogonaly}{\ortogonaly}
\changereferencesystem(\puntox,\puntoy)(\ortogonalx,\ortogonaly)(\direccionx,\direcciony)
}

\newcommand{\arcosinpunto}[6]{%
\norma{#1}{#2}
\DIVIDE{#2}{\norm}{\tangentey}
\DIVIDE{#1}{\norm}{\tangentex}
\COPY{\tangentey}{\ortogonalx}
\COPY{\tangentey}{\centrox}
\SUBTRACT{0}{\tangentex}{\ortogonaly}
\MULTIPLY{\centrox}{#6}{\centrox}
\MULTIPLY{\ortogonaly}{#6}{\centroy}
\ADD{#3}{\centrox}{\centrox}
\ADD{#4}{\centroy}{\centroy}
\changereferencesystem(\centrox,\centroy)(\ortogonalx,\ortogonaly)(\tangentex,\tangentey)
\radiansangles
\SUBTRACT{\numberPI}{#5}{\angulo}
\ellipticArc{#6}{#6}{\angulo}{\numberPI}
\COS{\angulo}{\puntox}
\SIN{\angulo}{\puntoy}
\MULTIPLY{#6}{\puntox}{\puntox}
\MULTIPLY{#6}{\puntoy}{\puntoy}
\SUBTRACT{0}{\puntox}{\direcciony}
\COPY{\puntoy}{\direccionx}
\norma{\direccionx}{\direcciony}
\DIVIDE{\direccionx}{\norm}{\direccionx}
\DIVIDE{\direcciony}{\norm}{\direcciony}
\COPY{\direcciony}{\ortogonalx}
\COPY{\direccionx}{\ortogonaly}
\SUBTRACT{0}{\ortogonaly}{\ortogonaly}
\changereferencesystem(\puntox,\puntoy)(\ortogonalx,\ortogonaly)(\direccionx,\direcciony)
}

\newcommand{\arcoblanco}[6]{%
\norma{#1}{#2}
\DIVIDE{#2}{\norm}{\tangentey}
\DIVIDE{#1}{\norm}{\tangentex}
\COPY{\tangentey}{\ortogonalx}
\COPY{\tangentey}{\centrox}
\SUBTRACT{0}{\tangentex}{\ortogonaly}
\MULTIPLY{\centrox}{#6}{\centrox}
\MULTIPLY{\ortogonaly}{#6}{\centroy}
\ADD{#3}{\centrox}{\centrox}
\ADD{#4}{\centroy}{\centroy}
\changereferencesystem(\centrox,\centroy)(\ortogonalx,\ortogonaly)(\tangentex,\tangentey)
\radiansangles
\SUBTRACT{\numberPI}{#5}{\angulo}
\COS{\angulo}{\puntox}
\SIN{\angulo}{\puntoy}
\MULTIPLY{#6}{\puntox}{\puntox}
\MULTIPLY{#6}{\puntoy}{\puntoy}
\Put(\puntox,\puntoy){\circle*{0.04}}
\SUBTRACT{0}{\puntox}{\direcciony}
\COPY{\puntoy}{\direccionx}
\norma{\direccionx}{\direcciony}
\DIVIDE{\direccionx}{\norm}{\direccionx}
\DIVIDE{\direcciony}{\norm}{\direcciony}
\COPY{\direcciony}{\ortogonalx}
\COPY{\direccionx}{\ortogonaly}
\SUBTRACT{0}{\ortogonaly}{\ortogonaly}
\changereferencesystem(\puntox,\puntoy)(\ortogonalx,\ortogonaly)(\direccionx,\direcciony)
}

\newcommand{\arcoblancosinpunto}[6]{%
\norma{#1}{#2}
\DIVIDE{#2}{\norm}{\tangentey}
\DIVIDE{#1}{\norm}{\tangentex}
\COPY{\tangentey}{\ortogonalx}
\COPY{\tangentey}{\centrox}
\SUBTRACT{0}{\tangentex}{\ortogonaly}
\MULTIPLY{\centrox}{#6}{\centrox}
\MULTIPLY{\ortogonaly}{#6}{\centroy}
\ADD{#3}{\centrox}{\centrox}
\ADD{#4}{\centroy}{\centroy}
\changereferencesystem(\centrox,\centroy)(\ortogonalx,\ortogonaly)(\tangentex,\tangentey)
\radiansangles
\SUBTRACT{\numberPI}{#5}{\angulo}
\COS{\angulo}{\puntox}
\SIN{\angulo}{\puntoy}
\MULTIPLY{#6}{\puntox}{\puntox}
\MULTIPLY{#6}{\puntoy}{\puntoy}
\SUBTRACT{0}{\puntox}{\direcciony}
\COPY{\puntoy}{\direccionx}
\norma{\direccionx}{\direcciony}
\DIVIDE{\direccionx}{\norm}{\direccionx}
\DIVIDE{\direcciony}{\norm}{\direcciony}
\COPY{\direcciony}{\ortogonalx}
\COPY{\direccionx}{\ortogonaly}
\SUBTRACT{0}{\ortogonaly}{\ortogonaly}
\changereferencesystem(\puntox,\puntoy)(\ortogonalx,\ortogonaly)(\direccionx,\direcciony)
}

\newcommand{\segmento}[4]{%
\MULTIPLY{#3}{#4}{\long}
\MULTIPLY{#1}{\long}{\dirx}
\MULTIPLY{#2}{\long}{\diry}
\xLINE(0,0)(\dirx,\diry)
\Put(\dirx,\diry){\circle*{0.08}}
\changereferencesystem(\dirx,\diry)(#2,-#1)(#1,#2)
}

\newcommand{\segmentosinpunto}[4]{%
\MULTIPLY{#3}{#4}{\long}
\MULTIPLY{#1}{\long}{\dirx}
\MULTIPLY{#2}{\long}{\diry}
\xLINE(0,0)(\dirx,\diry)
\changereferencesystem(\dirx,\diry)(#2,-#1)(#1,#2)
}

\newcommand{\segmentoblanco}[4]{%
\MULTIPLY{#3}{#4}{\long}
\MULTIPLY{#1}{\long}{\dirx}
\MULTIPLY{#2}{\long}{\diry}
\Put(\dirx,\diry){\circle*{0.04}}
\changereferencesystem(\dirx,\diry)(#2,-#1)(#1,#2)
}

\newcommand{\segmentoblancosinpunto}[4]{%
\MULTIPLY{#3}{#4}{\long}
\MULTIPLY{#1}{\long}{\dirx}
\MULTIPLY{#2}{\long}{\diry}
\changereferencesystem(\dirx,\diry)(#2,-#1)(#1,#2)
}

\setlength{\unitlength}{1.5cm}%

\begin{Picture}(-1.2,6)(11,8)
\thicklines

\referencesystem(0,0)(1,0)(0,1)

\COPY{1}{\rad}
\COPY{0.4}{\ang}
\DIVIDE{\ang}{2}{\angbis}
\DIVIDE{\ang}{3}{\angbisbis}
\COPY{0.7}{\angtris}

\COPY{0}{\taux}
\COPY{1}{\tauy}
\COPY{0}{\pux}
\COPY{7}{\puy}

\Put(-0.05,7.2){$\vdots$}
\Put(0,7){\circle*{0.08}}
\Put(-0.3,7){\tiny{$p_s$}}
\changereferencesystem(0,7)(1,0)(0,1)
\segmento{0}{-1}{\ang}{\rad}
\segmento{-1}{0}{\ang}{\rad}
\segmentosinpunto{0}{1}{\angbisbis}{\rad}
\Put(0.01,0.1){$\ldots$}
\changereferencesystem(0,0.5)(1,0)(0,1)
\segmentosinpunto{0}{1}{\angbisbis}{\rad}
\Put(0,0){\circle*{0.08}}
\Put(-0.2,-0.2){\tiny{$p_{i_0-1}$}}
\segmento{0}{1}{\ang}{\rad}
\Put(-0.2,0){\tiny{$p_{i_0}$}}
\segmento{0}{1}{\ang}{\rad}
\Put(-0.2,0){\tiny{$p_{i_0+1}$}}
\segmento{0}{-1}{\ang}{\rad}
\segmento{-1}{0}{\ang}{\rad}
\Put(-0.1,0){\tiny{$\mathfrak{q}$}}


\referencesystem(0,0)(1,0)(0,1)

\Put(4.95,7.2){$\vdots$}
\Put(5,7){\circle*{0.08}}
\Put(4.7,7){\tiny{$p_s$}}
\changereferencesystem(5,7)(1,0)(0,1)
\segmento{0}{-1}{\ang}{\rad}
\segmento{-1}{0}{\ang}{\rad}
\segmentosinpunto{0}{1}{\angbisbis}{\rad}
\Put(0.01,0.1){$\ldots$}
\changereferencesystem(0,0.5)(1,0)(0,1)
\segmentosinpunto{0}{1}{\angbisbis}{\rad}
\Put(0,0){\circle*{0.08}}
\Put(-0.2,-0.2){\tiny{$p_{i_0-1}$}}
\segmento{0}{1}{\ang}{\rad}
\Put(-0.2,0){\tiny{$p_{i_0}$}}
\segmento{0}{1}{\ang}{\rad}
\Put(-0.2,0){\tiny{$\mathfrak{q}$}}
\segmento{0}{-1}{\ang}{\rad}
\segmento{-1}{0}{\ang}{\rad}
\Put(-0.1,0){\tiny{$p_{i_0+1}$}}

\end{Picture}
\end{center}

\caption{Graph $\tilde{\mathcal E}_{\nu_{i_0+1}}$ ($p_{i_0}, p_{i_0+1}$ and $\mathfrak{q}$ satellite)}
  \label{figura10a}
\end{figure}
\end{proof}

\begin{theorem}
\label{triangulo}
Let $\nu$ be a non-minimal exceptional curve valuation and keep the above notation. The Newton-Okounkov body $\Delta_{\nu}$ is a triangle if and only if the graph $\Gamma_{\nu}^*$ is not connected.
\end{theorem}
\begin{proof}
 Let $g^*$ be the number of Puiseux pairs of $\nu$. Assume that $C$ is defined by $f=0$ and write $\nu(f)/\deg (f) = \left(\hat{\mu} (\nu_r), c \right)$. Notice that, by Theorem \ref{53}, the set of vertices of $\Delta_{\nu}$ is contained in $$\{(0,0), Q_1, Q_2, Q_3=\left(\hat{\mu} (\nu_r), c \right)\}.$$

First suppose that $p_{r+1}$ is satellite. By Lemma \ref{cont}, the point $Q_1$  (respectively, $Q_2$) belongs to the line passing through the origin and with slope $\betabarra_0(\nu_{\eta})/\betabarra_0(\nu_{r})$ (respectively, $\betabarra_{g^*}(\nu_{\eta})/\betabarra_{g^*}(\nu_{r})$), or viceversa; hence $\Delta_{\nu}$ is a triangle if and only if $$c/\hat{\mu} (\nu_r)\in \{\betabarra_0(\nu_{\eta})/\betabarra_0(\nu_{r}), \betabarra_{g^*}(\nu_{\eta})/\betabarra_{g^*}(\nu_{r})\},$$ i.e., if and only if, there is an index $j\in \{0,g^*\}$ such that, for each analytic branch of $C$ at $p$ (say, with equation $h=0$), $\upsilon(h)/\nu_r(h)=\betabarra_j(\nu_{\eta})/\betabarra_j(\nu_{r})$, where $\nu(h)=(\nu_r(h),\upsilon(h))$. Let us prove that this last condition holds if and only if all the edges of $\Gamma_{\nu}^*$ that are incident with $C$ meet $\Gamma_{\nu}$ at the same connected component (that is, $\Gamma_{\nu}^*$ is not connected).

Let us denote by $\Gamma_{\nu}(r)$ the connected component of $\Gamma_{\nu}$ that contains the vertex $r$ (which is finite), and by $\Gamma_{\nu}'(r)$ the other one. Fix a sufficiently large natural number $k>r$ and consider the dual graph $\Gamma_{\nu_k}$ of the divisorial valuation $\nu_k$. Let $h=0$ be the equation of an analytic branch of $C$ at $p$ and assume that its strict transform in $X_k$ meets the strict transform of the exceptional divisor $E_i$. Notice that, by Lemma \ref{prrrt}, we can assume without lost of generality that $h=\varphi_i$.

Assume that $r\preccurlyeq \eta$ in $\Gamma_{\nu_k}$ (the reasoning is similar otherwise). This means that the part of the graph $\Gamma_\nu$ corresponding to the two last Puiseux pairs  has the shape that is shown in Figure \ref{figtheorem}.

\begin{figure}[h]
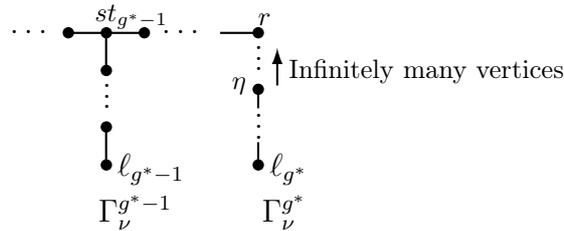


\begin{center}
\setlength{\unitlength}{0.5cm}%
\begin{Picture}(9,0)(24,8)
\thicklines


\put(9,6){$\;\;\ldots\;\;$}
\xLINE(11,6)(12,6)
\Put(11,6){\circle*{0.3}}
\Put(12,6){\circle*{0.3}}
\xLINE(12,6)(12,5)
\Put(12,5){\circle*{0.3}}
\Put(11.9,4){$\vdots$}
\xLINE(12,3.5)(12,2.5)
\Put(12,3.5){\circle*{0.3}}
\Put(12,2.5){\circle*{0.3}}
\Put(11.8,1){$\Gamma^{g^*-1}_\nu$}

\put(12.3,2.2){$\ell_{g^*-1}$}

\Put(11.7,6.3){\footnotesize $st_{g^*-1}$}



\xLINE(12,6)(13,6)
\Put(13,6){\circle*{0.3}}
\Put(13,6){$\;\;\ldots$}
\xLINE(15,6)(16,6)
\Put(16,6){\circle*{0.3}}
\Put(16,6.2){\footnotesize $r$}

\Put(15.9,5){$\vdots$}
\xLINE(16,3.1)(16,2.5)
\Put(16,2.5){\circle*{0.3}}



\Put(16,4.5){\circle*{0.3}}
\xLINE(16,4.5)(16,4)
\xVECTOR(16.5,4.6)(16.5,5.6)

\Put(15.3,4.5){\footnotesize $\eta$}

\Put(15.9,3.2){$\vdots$}

\Put(16.8,4.8){\footnotesize Infinitely many vertices }

\put(16.3,2.2){$\ell_{g^*}$}
\Put(16.1,1){$\Gamma^{g^*}_\nu$}

\end{Picture}
\end{center}
  \caption{Dual graph of $\Gamma_{\nu}$ if $r\preccurlyeq \eta$ in $\Gamma_{\nu_k}$ (proof of Theorem \ref{triangulo})}
  \label{figtheorem}
\end{figure}

We will distinguish two cases:\\

\noindent \emph{Case 1}: $r\not\preccurlyeq i$ in $\Gamma_{\nu_k}$. Here it must be $i\leq r$.  The Enriques diagram ${\mathcal E}_{\nu}$ (labeled with the sequence of values of $\nu$) and the Noether formula (\ref{Noether}) prove that $\nu(\varphi_i)=(\nu_r(\varphi_i),\nu_{\eta}(\varphi_i) )$. Moreover $\rho(i)<g^*$. Then, we consider two subcases:\\

\noindent \emph{Case 1.1}: $p_i$ is free and $i<\ell_{\rho(i)+1}$ which implies that $i\leq \eta$. Applying Part (a) of Proposition \ref{felix}, it holds that
$$\frac{\nu_{\eta}(\varphi_i)}{\nu_{r}(\varphi_i)}=\frac{e_{\rho(i)-1}(\nu_i) \bar{\beta}_{\rho(i)}(\nu_\eta) + d e_{\rho(i)} (\nu_i) e_{\rho(i)}(\nu_\eta)}{e_{\rho(i)-1}(\nu_i) \bar{\beta}_{\rho(i)}(\nu_r) + d e_{\rho(i)} (\nu_i) e_{\rho(i)}(\nu_r)}=\frac{\betabarra_0(\nu_{\eta})}{\betabarra_0(\nu_r)},$$
 where $d$ is the length of the path $[st_{\rho(i)}+1,i]$ in $\Gamma_{\nu_k}$ and the last equality holds because, as $\rho(i)<g^*$, one gets the following chain of equalities
 $$\frac{\betabarra_{\rho(i)}(\nu_{\eta})}{\betabarra_{\rho(i)}(\nu_{r})}=\frac{e_{\rho(i)}(\nu_{\eta})}{e_{\rho(i)}(\nu_{r})}=\frac{\betabarra_{0}(\nu_{\eta})}{\betabarra_{0}(\nu_{r})}.$$

 \noindent \emph{Case 1.2}: Either $p_i$ is satellite or $i=\ell_{\rho(i)+1}$ (notice that the last condition cannot happen if $\rho(i)=g^*-1$ because $r\preccurlyeq \ell_{g^*}$ in $\Gamma_{\nu_k}$). By Part (b) of Proposition \ref{felix} we have that
 $$\frac{\nu_{\eta}(\varphi_i)}{\nu_{r}(\varphi_i)}=\frac{e_{\rho(i)}(\nu_\eta) \betabarra_{\rho(i)+1}(\nu_i)}{e_{\rho(i)}(\nu_r) \betabarra_{\rho(i)+1}(\nu_i)}=\frac{\betabarra_{0}(\nu_\eta)}{\betabarra_{0}(\nu_r)}.$$

As a consequence, we have deduced, for this Case 1, that the slope of the line of $\mathbb{R}^2$ joining the origin and  $\nu(\varphi_i)$ is $\betabarra_{0}(\nu_\eta)/\betabarra_{0}(\nu_r)$.  So, it only remains to study the following case.\\

\noindent \emph{Case 2}: $r \preccurlyeq i$ in $\Gamma_{\nu_k}$. Notice that this implies that $\rho(i)=g^*-1$. Again we distinguish two subcases:\\

\noindent \emph{Case 2.1}: $\eta \preccurlyeq i$ in $\Gamma_{\nu_k}$. As before, one can deduce that $\nu(\varphi_i)=(\nu_r(\varphi_i), \nu_{\eta}(\varphi_i))$. Moreover we can apply Part (b) of Proposition \ref{felix} obtaining that
$$\frac{\nu_{\eta}(\varphi_i)}{\nu_{r}(\varphi_i)}=\frac{e_{g^*-1}(\nu_i) \betabarra_{g^*}(\nu_\eta)}{e_{g^*-1}(\nu_i) \betabarra_{g^*}(\nu_r)}=\frac{ \betabarra_{g^*}(\nu_\eta)}{ \betabarra_{g^*}(\nu_r)}.$$

\noindent \emph{Case 2.2}: $i\preccurlyeq \eta$ in $\Gamma_{\nu_k}$, that is, $p_i$ corresponds with one of the infinitely many vertices between $\eta$ and $r$ in the dual graph $\Gamma_{\nu}$ (see Figure \ref{figtheorem}). Consider the continued fraction expansion of the Puiseux exponent $\beta'_{g^*}(\nu_r)=[a_0,a_1,\ldots,a_n]$, where $n=\sigma(\nu_r)$. Using Enriques diagrams we deduce that $\beta'_{g^*}(\nu_\eta)=[a_0,a_1,\ldots,a_{n-1}]$ and $\beta'_{g^*}(\nu_{i})=[a_0,a_1,\ldots,a_n,\delta]$ for some positive integer $\delta$. By (\ref{Delta}) and \cite[Theorem 7.5]{niven}, we have that
$$\frac{\betabarra_{g^*}(\nu_{i})-n_{g^*-1}(\nu_i)\betabarra_{g^*-1}(\nu_i)}{e_{g^*-1}(\nu_i)}-\frac{\betabarra_{g^*}(\nu_{\eta})-n_{g^*-1}(\nu_\eta)\betabarra_{g^*-1}(\nu_\eta)}{e_{g^*-1}(\nu_\eta)}=\frac{(-1)^{n+1}\delta}{e_{g^*-1}(\nu_i)e_{g^*-1}(\nu_\eta) }.$$
From this equality, and taking into account that $n_{g^*-1}(\nu_i)=n_{g^*-1}(\nu_\eta)$ and that $
\betabarra_{g^*-1}(\nu_\eta)/\betabarra_{g^*-1}(\nu_i)=e_{g^*-2}(\nu_\eta)/e_{g^*-2}(\nu_i)$, it follows that
$$\delta=(-1)^{n+1}\left[ e_{g^*-1} (\nu_\eta)\betabarra_{g^*}(\nu_i) -e_{g^*-1} (\nu_i)\betabarra_{g^*}(\nu_\eta)\right].$$
Since $r\preccurlyeq i\preccurlyeq \eta$, it holds that $\beta'_{g^*}(\nu_{r})\leq \beta'_{g^*}(\nu_{i})\leq \beta'_{g^*}(\nu_{\eta})$. Then, the properties of the continued fractions allow us to prove that $n=\sigma_{g^*}(\nu_r)$ is even and, therefore,
$$\delta= e_{g^*-1} (\nu_i)\betabarra_{g^*}(\nu_\eta)-e_{g^*-1} (\nu_\eta)\betabarra_{g^*}(\nu_i).$$

With the help of the Enriques diagram of $\nu$ (labeled with the sequence of values of $\nu$) it is easily seen that $\nu(\varphi_i)=(\nu_r(\varphi_i),\nu_\eta(\varphi_i)+\delta )$. Then, by Part (b) of Proposition \ref{felix},
$$\frac{\nu_\eta(\varphi_i)+\delta}{\nu_r(\varphi_i)}=\frac{e_{g^*-1}(\nu_i)\betabarra_{g^*}(\nu_\eta)}{e_{g^*-1}(\nu_i)\betabarra_{g^*}(\nu_r)}=\frac{\betabarra_{g^*}(\nu_\eta)}{\betabarra_{g^*}(\nu_r)}.$$\\

Summarizing, we have proved that if $r\not \preccurlyeq i$ (i.e., if the divisor $E_i$ corresponds to a vertex in $\Gamma_{\nu}(r)$), then the slope of the line that joins $\nu(\varphi_i)$ and the origin is  $\betabarra_{0}(\nu_\eta)/\betabarra_{0}(\nu_r)$, and it is $\betabarra_{g^*}(\nu_\eta)/\betabarra_{g^*}(\nu_r)$ otherwise. This proves that, when $p_{r+1}$ is satellite, $c/\hat{\mu} (\nu_r)\in \{\betabarra_0(\nu_{\eta})/\betabarra_0(\nu_{r}), \betabarra_{g^*}(\nu_{\eta})/\betabarra_{g^*}(\nu_{r})\}$ if and only if the strict transforms of all the analytic branches of $C$ at $p$ meet the dual graph $\Gamma_{\nu}$ at the same connected component, that is, if and only if $\Gamma_{\nu}^*$ is not connected.\\

To conclude the proof, assume that $p_{r+1}$ is free. Then, the result can be proved using a similar reasoning as above after taking into account the part (b) in Theorem \ref{53} and the fact that $\nu(\varphi_i)=({\rm mult}_{p_r}(\varphi_i) \;\betabarra_{g^*} (\nu_r),m)$, where $m=\max\{{\rm mult}_{p_r}(\varphi_i),0\}$.
\end{proof}

An immediate consequence of the above result is the following.

\begin{corollary}
Under the same assumptions of Theorem \ref{triangulo}, it holds that if the supraminimal curve $C$ has only one analytic branch at $p$, then the Newton-Okounkov body $\Delta_{\nu}$ is a triangle.
\end{corollary}

\section{Newton-Okounkov bodies of non-positive at infinite valuations}
\label{NPI}

Keeping the notations of the preceding sections, set $(X:Y:Z)$ projective coordinates in $\mathbb{P}^2$, assume that the coordinates of the point $p$ are $(1:0:0)$ and take affine coordinates $u=Y/X$ and $v=Z/X$ around $p$. Also, consider coordinates $x=X/Z$ and $y=Y/Z$ in the affine chart defined by $Z\neq 0$. Then a divisorial valuation as above (say $\nu_r$) is called to be {\it non-positive at infinity} whenever $r\geq 2$, the line $L$ with equation $Z=0$ passes through $p_1=p$ and $p_2$, and $\nu_r (h) \leq 0$ for all $h \in \mathbb{C}[x,y] \setminus \{0\}$.

\begin{definition}
{\rm A exceptional curve valuation $\nu$ is named  {\it non-positive at infinity} when its attached divisorial valuation $\nu_r$ is  non-positive at infinity.}
\end{definition}

When $\nu_r$ is non-positive at infinity, it holds that  $\hat{\mu}(\nu_r) = \nu_r (v)$ and the line $Z=0$ is a supraminimal curve \cite[Theorem 1]{GM}. This result in \cite{GM} also provides an easy way to decide whether $\nu_r$ is non-positive at infinity. In fact, $\nu_r$ is non-positive at infinity if and only if $\nu_r (v)^2 \geq \betabarra_{g+1} (\nu_r)$. In view of Theorem \ref{triangulo} and applying Theorem \ref{53}, we deduce the following result, which gives {\it explicitly}  the vertices of the Newton-Okounkov body of this large family of exceptional curve valuations.

\begin{corollary}
\label{nonpos}
Let $E_{\bullet}:= \left\{ X=X_r \supset E_r \supset \{q\} \right\}$ be a flag whose divisor $E_r$ defines a non-positive at infinity valuation. Then the Newton-Okounkov body of the exceptional curve valuation $\nu = \nu_{E_\bullet}$ is a triangle whose vertices are $\left(0,0\right)$, $Q_2$ and $Q_3$,  where the latter points have the following coordinates:

\begin{itemize}
\item[(a)] If $\nu$ is not minimal:
\[
Q_2= \left(\frac{1}{\mathrm{vol(\nu_r)} \nu_r(v)}, \frac{1}{\nu_r(v)}\right) \;\;{\mbox and} \;\;Q_3=
\left(\nu_r(v),0\right),
\]
whenever $q=p_{r+1}$ is a free point in $E_r$.

\[
 Q_2= \left(\frac{1}{\mathrm{vol(\nu_r)} \nu_r(v)}, \frac{\nu_r(\varphi_\eta)}{\nu_r(v)}\right) \;\; {\mbox and} \;\;Q_3=
\left(\nu_r(v), \nu_\eta(v)\right),
\]
whenever $q$ is a satellite point in $E_\eta \cap E_r$ and $\eta \preccurlyeq r$.

\[
Q_2= \left(\frac{1}{\mathrm{vol(\nu_r)}\nu_r(v)}, \frac{\nu_r(\varphi_\eta)+1}{\nu_r(v)}\right)
\;\; {\mbox and} \;\; Q_3=  \left(\nu_r(v), \nu_\eta(v)\right),
\]
 otherwise.

\item[(b)] If $\nu$ is minimal

\[
  Q_2 = \left( \nu_r(v), 0 \right) \;\; \mbox{and} \;\; Q_3 = \left( \nu_r(v), \frac{1}{\nu_r(v)} \right),
\]
whenever $q$ is a free point in $E_r$.

\[
Q_2 = \left(\nu_r(v), \frac{1-\mathrm{vol} \left( \nu_r \right)}{\mathrm{vol} \left( \nu_r \right)\nu_r(v)}\right)  \;\; {\mbox and} \;\; Q_3 = \left(\nu_r(v),  \frac{1}{\mathrm{vol} \left( \nu_r \right)\nu_r(v)}\right),
\]
otherwise.

\end{itemize}
\end{corollary}

\begin{remark}
\label{la53}
{\rm Corollary \ref{nonpos} was announced in \cite{crm}. We have deduced it from Theorem \ref{53}. However, in \cite{crm} we deduced the same result computing directly the Zariski decompositions of the $\mathbb{R}$-divisors $D_t:=H- t E_r$ for all $t\in [0,\nu_r(v)]$ and applying then Theorem 6.4 of \cite{LM}.  For the sake of completeness we show now the decompositions $D_t=P_t+N_t$, where $P_t$ is the positive part and $N_t$ the negative part.

Let us begin with non-minimal exceptional curve valuations which are non-positive at infinity. This means that the valuation $\nu_r$ satisfies the inequality
 $\nu_r(v)^2\geq \betabarra_{g+1} (\nu_r)$ (see \cite{GM}).

For $i=1, 2, \ldots,r$, and keeping notation as above, consider the divisor on $X$
$$\mathbb{D}_i:=\nu_i(v)H -\sum_{i=1}^r \nu_i (\mathfrak{m}_i) E_i^*,$$
where $\nu_i$ is the divisorial valuation defined by the exceptional divisor $E_i$  and $E_i^*$ the total transform on $X$ of $E_i$.

Now, if $t$ is a real number in the interval $[0,\nu_r(v)]$ we have:
\begin{itemize}
\item[(a)] If $t\leq \betabarra_{g+1} (\nu_r) /\nu_r (v)$, then
$$P_t=b_0(t) H +b_r(t)\mathbb{D}_r\;\;\mbox{ and }\;\; N_t=\sum_{i=1}^{r-1} a_i(t)E_i,$$
where
$$b_0(t)=1-\frac{\nu_r(v)}{\betabarra_{g+1}(\nu_r)}t,\;\;\;\; b_r(t)=\frac{1}{\betabarra_{g+1}(\nu_r)}t,\;\;\;\; a_i(t)=\frac{\nu_r(\varphi_i)}{\betabarra_{g+1}(\nu_r)}t, \;\;\;1\leq i\leq r-1.$$

\item[(b)] If $t\geq\betabarra_{g+1} (\nu_r)/\nu_r(v)$, then
$$P_t=b_r(t)\mathbb{D}_r+c(t) (H -L^*)\;\;\mbox{ and } N_t= a_0(t)L+\sum_{i=1}^{r-1} a_i(t)E_i,$$
where, as above, $L^*$ means total transform on $X$ and $L$ stands for the strict transform on $X$ of the line $L$ given by $Z=0$, and
$$b_r(t)=\frac{\nu_r(v)-t}{\nu_r(v)^2-\betabarra_{g+1}(\nu_r)},\;\;\;\;c(t)=
\frac{\nu_r(v)t-\betabarra_{g+1}(\nu_r)}{\nu_r(v)^2-\betabarra_{g+1}(\nu_r)},\;\;\;a_0(t)=\frac{\nu_r(v)t-\betabarra_{g+1}(\nu_r)}{\nu_r(v)^2-\betabarra_{g+1}(\nu_r)},$$ $$a_i(t)=\frac{\nu_i(v)(\nu_r(v)t-\betabarra_{g+1}(\nu_r))+\nu_r(\varphi_i)(\nu_r(v)-t)}{\nu_r(v)^2-\betabarra_{g+1}(\nu_r)}, \;\;\;1\leq i\leq r-1.$$
\end{itemize}

Finally assume that $\nu$ is a minimal exceptional curve valuation which is non-positive at infinity, i.e., $\nu_r(v)^2=\betabarra_{g+1}(\nu_r)$. Then, for real numbers $t$ in $[0,\nu_r(v)]$, it holds
$$P_t=b_0(t) H+b_r(t)\mathbb{D}_r\;\;\mbox{ and }\;\; N_t=\sum_{i=1}^{r-1} a_i(t) E_i,$$
where
$$b_0(t)=\frac{\nu_r(v)-t}{\nu_r(v)},\;\;\;\; b_r(t)=\frac{t}{\nu_r(v)^2},\;\;\;\; a_i(t)=\frac{\nu_r(\varphi_i)}{\nu_r(v)^2}t, \;\;\;1\leq i\leq r-1.$$
}
\end{remark}

\section{Examples}
We finish this paper using our results to compute several examples of Newton-Okounkov bodies of exceptional curve valuations.\\

\noindent {\bf Example 1.} Let $\nu_r$ be a divisorial valuation centered at the local ring of a point $p$ of $\mathbb{P}^2$ whose sequence of maximal contact values is $\{\betabarra_j (\nu_r)\}_{j=0}^3 = \{9,30,101,303\}$. Let ${\mathcal C}_{\nu_r} = \{p_i\}_{i=1}^{12}$ (with $p=p_1$) be its cluster of centers and let $L$ be the projective line passing through $p$ and whose strict transform passes through $p_2$. By means of an affine change of coordinates, if necessary, we can assume that $L$ is defined, in projective coordinates $(X:Y:Z)$, by the equation $Z=0$. Moreover, assume that $L$ does not pass through $p_3$.

Then, in the notation of Section \ref{NPI} we have that $\nu_r(v)=18$ and $\nu_r(v)^2= 324 > 303 = \betabarra_3 (\nu_r)$. Thus $\nu_r$ is non-positive at infinity by \cite[Theorem 1]{GM}. Let  $\nu = \nu_{E_\bullet}$ be the valuation defined by the flag $E_\bullet = \left\{ X=X_{12} \supset E_{12}\supset \{q=p_{13}\}\right\}$, where $q \in E_{10} \cap E_{12}$.  The Enriques diagram ${\mathcal E}_{\nu}$ of $\nu$ is displayed in Figure \ref{fig11}.

\begin{figure}
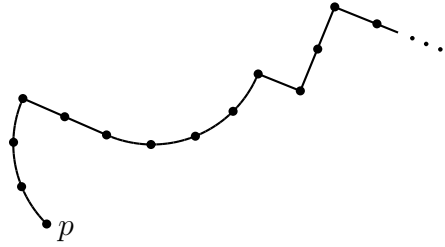

\begin{center}

\newcommand{\norma}[2]{%
\SQUARE{#1}{\m}
\SQUARE{#2}{\n}
\ADD{\m}{\n}{\k}
\SQUAREROOT{\k}{\norm}
}


\newcommand{\arco}[6]{%
\norma{#1}{#2}
\DIVIDE{#2}{\norm}{\tangentey}
\DIVIDE{#1}{\norm}{\tangentex}
\COPY{\tangentey}{\ortogonalx}
\COPY{\tangentey}{\centrox}
\SUBTRACT{0}{\tangentex}{\ortogonaly}
\MULTIPLY{\centrox}{#6}{\centrox}
\MULTIPLY{\ortogonaly}{#6}{\centroy}
\ADD{#3}{\centrox}{\centrox}
\ADD{#4}{\centroy}{\centroy}
\changereferencesystem(\centrox,\centroy)(\ortogonalx,\ortogonaly)(\tangentex,\tangentey)
\radiansangles
\SUBTRACT{\numberPI}{#5}{\angulo}
\ellipticArc{#6}{#6}{\angulo}{\numberPI}
\COS{\angulo}{\puntox}
\SIN{\angulo}{\puntoy}
\MULTIPLY{#6}{\puntox}{\puntox}
\MULTIPLY{#6}{\puntoy}{\puntoy}
\Put(\puntox,\puntoy){\circle*{0.08}}
\SUBTRACT{0}{\puntox}{\direcciony}
\COPY{\puntoy}{\direccionx}
\norma{\direccionx}{\direcciony}
\DIVIDE{\direccionx}{\norm}{\direccionx}
\DIVIDE{\direcciony}{\norm}{\direcciony}
\COPY{\direcciony}{\ortogonalx}
\COPY{\direccionx}{\ortogonaly}
\SUBTRACT{0}{\ortogonaly}{\ortogonaly}
\changereferencesystem(\puntox,\puntoy)(\ortogonalx,\ortogonaly)(\direccionx,\direcciony)
}

\newcommand{\arcosinpunto}[6]{%
\norma{#1}{#2}
\DIVIDE{#2}{\norm}{\tangentey}
\DIVIDE{#1}{\norm}{\tangentex}
\COPY{\tangentey}{\ortogonalx}
\COPY{\tangentey}{\centrox}
\SUBTRACT{0}{\tangentex}{\ortogonaly}
\MULTIPLY{\centrox}{#6}{\centrox}
\MULTIPLY{\ortogonaly}{#6}{\centroy}
\ADD{#3}{\centrox}{\centrox}
\ADD{#4}{\centroy}{\centroy}
\changereferencesystem(\centrox,\centroy)(\ortogonalx,\ortogonaly)(\tangentex,\tangentey)
\radiansangles
\SUBTRACT{\numberPI}{#5}{\angulo}
\ellipticArc{#6}{#6}{\angulo}{\numberPI}
\COS{\angulo}{\puntox}
\SIN{\angulo}{\puntoy}
\MULTIPLY{#6}{\puntox}{\puntox}
\MULTIPLY{#6}{\puntoy}{\puntoy}
\SUBTRACT{0}{\puntox}{\direcciony}
\COPY{\puntoy}{\direccionx}
\norma{\direccionx}{\direcciony}
\DIVIDE{\direccionx}{\norm}{\direccionx}
\DIVIDE{\direcciony}{\norm}{\direcciony}
\COPY{\direcciony}{\ortogonalx}
\COPY{\direccionx}{\ortogonaly}
\SUBTRACT{0}{\ortogonaly}{\ortogonaly}
\changereferencesystem(\puntox,\puntoy)(\ortogonalx,\ortogonaly)(\direccionx,\direcciony)
}

\newcommand{\arcoblanco}[6]{%
\norma{#1}{#2}
\DIVIDE{#2}{\norm}{\tangentey}
\DIVIDE{#1}{\norm}{\tangentex}
\COPY{\tangentey}{\ortogonalx}
\COPY{\tangentey}{\centrox}
\SUBTRACT{0}{\tangentex}{\ortogonaly}
\MULTIPLY{\centrox}{#6}{\centrox}
\MULTIPLY{\ortogonaly}{#6}{\centroy}
\ADD{#3}{\centrox}{\centrox}
\ADD{#4}{\centroy}{\centroy}
\changereferencesystem(\centrox,\centroy)(\ortogonalx,\ortogonaly)(\tangentex,\tangentey)
\radiansangles
\SUBTRACT{\numberPI}{#5}{\angulo}
\COS{\angulo}{\puntox}
\SIN{\angulo}{\puntoy}
\MULTIPLY{#6}{\puntox}{\puntox}
\MULTIPLY{#6}{\puntoy}{\puntoy}
\Put(\puntox,\puntoy){\circle*{0.04}}
\SUBTRACT{0}{\puntox}{\direcciony}
\COPY{\puntoy}{\direccionx}
\norma{\direccionx}{\direcciony}
\DIVIDE{\direccionx}{\norm}{\direccionx}
\DIVIDE{\direcciony}{\norm}{\direcciony}
\COPY{\direcciony}{\ortogonalx}
\COPY{\direccionx}{\ortogonaly}
\SUBTRACT{0}{\ortogonaly}{\ortogonaly}
\changereferencesystem(\puntox,\puntoy)(\ortogonalx,\ortogonaly)(\direccionx,\direcciony)
}

\newcommand{\arcoblancosinpunto}[6]{%
\norma{#1}{#2}
\DIVIDE{#2}{\norm}{\tangentey}
\DIVIDE{#1}{\norm}{\tangentex}
\COPY{\tangentey}{\ortogonalx}
\COPY{\tangentey}{\centrox}
\SUBTRACT{0}{\tangentex}{\ortogonaly}
\MULTIPLY{\centrox}{#6}{\centrox}
\MULTIPLY{\ortogonaly}{#6}{\centroy}
\ADD{#3}{\centrox}{\centrox}
\ADD{#4}{\centroy}{\centroy}
\changereferencesystem(\centrox,\centroy)(\ortogonalx,\ortogonaly)(\tangentex,\tangentey)
\radiansangles
\SUBTRACT{\numberPI}{#5}{\angulo}
\COS{\angulo}{\puntox}
\SIN{\angulo}{\puntoy}
\MULTIPLY{#6}{\puntox}{\puntox}
\MULTIPLY{#6}{\puntoy}{\puntoy}
\SUBTRACT{0}{\puntox}{\direcciony}
\COPY{\puntoy}{\direccionx}
\norma{\direccionx}{\direcciony}
\DIVIDE{\direccionx}{\norm}{\direccionx}
\DIVIDE{\direcciony}{\norm}{\direcciony}
\COPY{\direcciony}{\ortogonalx}
\COPY{\direccionx}{\ortogonaly}
\SUBTRACT{0}{\ortogonaly}{\ortogonaly}
\changereferencesystem(\puntox,\puntoy)(\ortogonalx,\ortogonaly)(\direccionx,\direcciony)
}

\newcommand{\segmento}[4]{%
\MULTIPLY{#3}{#4}{\long}
\MULTIPLY{#1}{\long}{\dirx}
\MULTIPLY{#2}{\long}{\diry}
\xLINE(0,0)(\dirx,\diry)
\Put(\dirx,\diry){\circle*{0.08}}
\changereferencesystem(\dirx,\diry)(#2,-#1)(#1,#2)
}

\newcommand{\segmentosinpunto}[4]{%
\MULTIPLY{#3}{#4}{\long}
\MULTIPLY{#1}{\long}{\dirx}
\MULTIPLY{#2}{\long}{\diry}
\xLINE(0,0)(\dirx,\diry)
\changereferencesystem(\dirx,\diry)(#2,-#1)(#1,#2)
}

\newcommand{\segmentoblanco}[4]{%
\MULTIPLY{#3}{#4}{\long}
\MULTIPLY{#1}{\long}{\dirx}
\MULTIPLY{#2}{\long}{\diry}
\Put(\dirx,\diry){\circle*{0.04}}
\changereferencesystem(\dirx,\diry)(#2,-#1)(#1,#2)
}

\newcommand{\segmentoblancosinpunto}[4]{%
\MULTIPLY{#3}{#4}{\long}
\MULTIPLY{#1}{\long}{\dirx}
\MULTIPLY{#2}{\long}{\diry}
\changereferencesystem(\dirx,\diry)(#2,-#1)(#1,#2)
}

\setlength{\unitlength}{1.5cm}%

\begin{Picture}(-1,-1)(4.5,3)
\thicklines

\referencesystem(0,0)(1,0)(0,1)

\COPY{1}{\rad}
\COPY{0.4}{\ang}
\DIVIDE{\ang}{2}{\angbis}
\DIVIDE{\ang}{3}{\angbisbis}

\COPY{-1}{\taux}
\COPY{1}{\tauy}
\COPY{0}{\pux}
\COPY{0}{\puy}
\Put(\pux,\puy){\circle*{0.08}} \Put(0.1,-0.1){$p$}

\arco{\taux}{\tauy}{\pux}{\puy}{\ang}{\rad}
\arco{0}{1}{0}{0}{\ang}{\rad}
\arco{0}{1}{0}{0}{\ang}{\rad}
\segmento{1}{0}{\ang}{\rad}
\segmento{0}{1}{\ang}{\rad}

\changereferencesystem(0,0)(-1,0)(0,1)

\arco{0}{1}{0}{0}{\ang}{\rad}
\arco{0}{1}{0}{0}{\ang}{\rad}
\arco{0}{1}{0}{0}{\ang}{\rad}
\arco{0}{1}{0}{0}{\ang}{\rad}

\segmento{-1}{0}{\ang}{\rad}
\segmento{1}{0}{\ang}{\rad}
\segmento{0}{1}{\ang}{\rad}
\segmento{-1}{0}{\ang}{\rad}

\segmentosinpunto{0}{1}{\angbis}{\rad}
\segmentoblanco{0}{1}{\angbisbis}{\rad}
\segmentoblanco{0}{1}{\angbisbis}{\rad}
\segmentoblanco{0}{1}{\angbisbis}{\rad}

\end{Picture}
\end{center}

\caption{Enriques diagram of the non-positive at infinity valuation $\nu$ in Example 1}
  \label{fig11}
\end{figure}

By Corollary \ref{nonpos}, and since $\nu_r$ is not minimal, we get that $\Delta_\nu$ is a triangle with vertices $(0,0)$,
\[
Q_2= \left(\frac{1}{\mathrm{vol(\nu_r)}\nu_r(v)}, \frac{\nu_r(\varphi_\eta)+1}{\nu_r(v)}\right) = \left(\frac{303}{18}, \frac{102}{8} \right) \]
and
\[
Q_3=  \left(\nu_r(v), \nu_\eta(v)\right) = \left( 18, 6 \right).
\]

Notice that one could compute, in a similar way, the vertices of the Newton-Okounkov body in the case in which $L$ goes through more free points $p_i$.\\

\noindent {\bf Example 2.} With the notation of this paper, consider the homogeneous polynomial $f_1= Y^2 X + Z^3$ and the curve $C$ given by $f=0$, where
\[
f = \left( \left( f_1 Z + a Y^4\right)^3 - f_1^4 \right) / X^2 ;  \;\; \mathbb{C} \ni a \neq 0.
\]
It has degree 10 and, by \cite{melle}, it defines at $p=(1:0:0)$ an analytically irreducible germ $\varphi$ whose sequence of multiplicities is $6, 3, 3,3,3,3,3,3,1,1,1$. Now, consider the divisorial valuation $\nu_r$ centered at the local ring at $p$ with sequence of maximal contact values $\{\betabarra_j (\nu_r)\}_{j=0}^3 = \{8,12,45,180\}$ and cluster of centers   ${\mathcal C}_{\nu_r} = \{p_i\}_{i=1}^{12}$ which shares with $f=0$ the first eleven points and $p_{12} \in E_{11} \cap E_8$. Consider the flag valuation $\nu = \nu_{E_\bullet}$ defined by $E_\bullet = \{ X=X_{12} \supset E_{12}\supset \{q=p_{13}\}\}$, where $q \in E_{8} \cap E_{12}$.

\begin{figure}
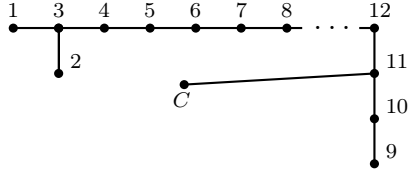

\begin{center}

\newcommand{\norma}[2]{%
\SQUARE{#1}{\m}
\SQUARE{#2}{\n}
\ADD{\m}{\n}{\k}
\SQUAREROOT{\k}{\norm}
}


\newcommand{\arco}[6]{%
\norma{#1}{#2}
\DIVIDE{#2}{\norm}{\tangentey}
\DIVIDE{#1}{\norm}{\tangentex}
\COPY{\tangentey}{\ortogonalx}
\COPY{\tangentey}{\centrox}
\SUBTRACT{0}{\tangentex}{\ortogonaly}
\MULTIPLY{\centrox}{#6}{\centrox}
\MULTIPLY{\ortogonaly}{#6}{\centroy}
\ADD{#3}{\centrox}{\centrox}
\ADD{#4}{\centroy}{\centroy}
\changereferencesystem(\centrox,\centroy)(\ortogonalx,\ortogonaly)(\tangentex,\tangentey)
\radiansangles
\SUBTRACT{\numberPI}{#5}{\angulo}
\ellipticArc{#6}{#6}{\angulo}{\numberPI}
\COS{\angulo}{\puntox}
\SIN{\angulo}{\puntoy}
\MULTIPLY{#6}{\puntox}{\puntox}
\MULTIPLY{#6}{\puntoy}{\puntoy}
\Put(\puntox,\puntoy){\circle*{0.08}}
\SUBTRACT{0}{\puntox}{\direcciony}
\COPY{\puntoy}{\direccionx}
\norma{\direccionx}{\direcciony}
\DIVIDE{\direccionx}{\norm}{\direccionx}
\DIVIDE{\direcciony}{\norm}{\direcciony}
\COPY{\direcciony}{\ortogonalx}
\COPY{\direccionx}{\ortogonaly}
\SUBTRACT{0}{\ortogonaly}{\ortogonaly}
\changereferencesystem(\puntox,\puntoy)(\ortogonalx,\ortogonaly)(\direccionx,\direcciony)
}

\newcommand{\arcosinpunto}[6]{%
\norma{#1}{#2}
\DIVIDE{#2}{\norm}{\tangentey}
\DIVIDE{#1}{\norm}{\tangentex}
\COPY{\tangentey}{\ortogonalx}
\COPY{\tangentey}{\centrox}
\SUBTRACT{0}{\tangentex}{\ortogonaly}
\MULTIPLY{\centrox}{#6}{\centrox}
\MULTIPLY{\ortogonaly}{#6}{\centroy}
\ADD{#3}{\centrox}{\centrox}
\ADD{#4}{\centroy}{\centroy}
\changereferencesystem(\centrox,\centroy)(\ortogonalx,\ortogonaly)(\tangentex,\tangentey)
\radiansangles
\SUBTRACT{\numberPI}{#5}{\angulo}
\ellipticArc{#6}{#6}{\angulo}{\numberPI}
\COS{\angulo}{\puntox}
\SIN{\angulo}{\puntoy}
\MULTIPLY{#6}{\puntox}{\puntox}
\MULTIPLY{#6}{\puntoy}{\puntoy}
\SUBTRACT{0}{\puntox}{\direcciony}
\COPY{\puntoy}{\direccionx}
\norma{\direccionx}{\direcciony}
\DIVIDE{\direccionx}{\norm}{\direccionx}
\DIVIDE{\direcciony}{\norm}{\direcciony}
\COPY{\direcciony}{\ortogonalx}
\COPY{\direccionx}{\ortogonaly}
\SUBTRACT{0}{\ortogonaly}{\ortogonaly}
\changereferencesystem(\puntox,\puntoy)(\ortogonalx,\ortogonaly)(\direccionx,\direcciony)
}

\newcommand{\arcoblanco}[6]{%
\norma{#1}{#2}
\DIVIDE{#2}{\norm}{\tangentey}
\DIVIDE{#1}{\norm}{\tangentex}
\COPY{\tangentey}{\ortogonalx}
\COPY{\tangentey}{\centrox}
\SUBTRACT{0}{\tangentex}{\ortogonaly}
\MULTIPLY{\centrox}{#6}{\centrox}
\MULTIPLY{\ortogonaly}{#6}{\centroy}
\ADD{#3}{\centrox}{\centrox}
\ADD{#4}{\centroy}{\centroy}
\changereferencesystem(\centrox,\centroy)(\ortogonalx,\ortogonaly)(\tangentex,\tangentey)
\radiansangles
\SUBTRACT{\numberPI}{#5}{\angulo}
\COS{\angulo}{\puntox}
\SIN{\angulo}{\puntoy}
\MULTIPLY{#6}{\puntox}{\puntox}
\MULTIPLY{#6}{\puntoy}{\puntoy}
\Put(\puntox,\puntoy){\circle*{0.04}}
\SUBTRACT{0}{\puntox}{\direcciony}
\COPY{\puntoy}{\direccionx}
\norma{\direccionx}{\direcciony}
\DIVIDE{\direccionx}{\norm}{\direccionx}
\DIVIDE{\direcciony}{\norm}{\direcciony}
\COPY{\direcciony}{\ortogonalx}
\COPY{\direccionx}{\ortogonaly}
\SUBTRACT{0}{\ortogonaly}{\ortogonaly}
\changereferencesystem(\puntox,\puntoy)(\ortogonalx,\ortogonaly)(\direccionx,\direcciony)
}

\newcommand{\arcoblancosinpunto}[6]{%
\norma{#1}{#2}
\DIVIDE{#2}{\norm}{\tangentey}
\DIVIDE{#1}{\norm}{\tangentex}
\COPY{\tangentey}{\ortogonalx}
\COPY{\tangentey}{\centrox}
\SUBTRACT{0}{\tangentex}{\ortogonaly}
\MULTIPLY{\centrox}{#6}{\centrox}
\MULTIPLY{\ortogonaly}{#6}{\centroy}
\ADD{#3}{\centrox}{\centrox}
\ADD{#4}{\centroy}{\centroy}
\changereferencesystem(\centrox,\centroy)(\ortogonalx,\ortogonaly)(\tangentex,\tangentey)
\radiansangles
\SUBTRACT{\numberPI}{#5}{\angulo}
\COS{\angulo}{\puntox}
\SIN{\angulo}{\puntoy}
\MULTIPLY{#6}{\puntox}{\puntox}
\MULTIPLY{#6}{\puntoy}{\puntoy}
\SUBTRACT{0}{\puntox}{\direcciony}
\COPY{\puntoy}{\direccionx}
\norma{\direccionx}{\direcciony}
\DIVIDE{\direccionx}{\norm}{\direccionx}
\DIVIDE{\direcciony}{\norm}{\direcciony}
\COPY{\direcciony}{\ortogonalx}
\COPY{\direccionx}{\ortogonaly}
\SUBTRACT{0}{\ortogonaly}{\ortogonaly}
\changereferencesystem(\puntox,\puntoy)(\ortogonalx,\ortogonaly)(\direccionx,\direcciony)
}

\newcommand{\segmento}[4]{%
\MULTIPLY{#3}{#4}{\long}
\MULTIPLY{#1}{\long}{\dirx}
\MULTIPLY{#2}{\long}{\diry}
\xLINE(0,0)(\dirx,\diry)
\Put(\dirx,\diry){\circle*{0.08}}
\changereferencesystem(\dirx,\diry)(#2,-#1)(#1,#2)
}

\newcommand{\segmentosinpunto}[4]{%
\MULTIPLY{#3}{#4}{\long}
\MULTIPLY{#1}{\long}{\dirx}
\MULTIPLY{#2}{\long}{\diry}
\xLINE(0,0)(\dirx,\diry)
\changereferencesystem(\dirx,\diry)(#2,-#1)(#1,#2)
}

\newcommand{\segmentoblanco}[4]{%
\MULTIPLY{#3}{#4}{\long}
\MULTIPLY{#1}{\long}{\dirx}
\MULTIPLY{#2}{\long}{\diry}
\Put(\dirx,\diry){\circle*{0.04}}
\changereferencesystem(\dirx,\diry)(#2,-#1)(#1,#2)
}

\newcommand{\segmentoblancosinpunto}[4]{%
\MULTIPLY{#3}{#4}{\long}
\MULTIPLY{#1}{\long}{\dirx}
\MULTIPLY{#2}{\long}{\diry}
\changereferencesystem(\dirx,\diry)(#2,-#1)(#1,#2)
}

\setlength{\unitlength}{1.5cm}%

\begin{Picture}(-1.2,6)(4.5,8)
\thicklines

\referencesystem(0,0)(1,0)(0,1)

\COPY{1}{\rad}
\COPY{0.4}{\ang}
\DIVIDE{\ang}{2}{\angbis}
\DIVIDE{\ang}{3}{\angbisbis}
\COPY{0.7}{\angtris}

\COPY{0}{\taux}
\COPY{1}{\tauy}
\COPY{0}{\pux}
\COPY{7}{\puy}

\Put(1.5,6.5){\circle*{0.08}}
\Put(1.4,6.3){\tiny{$C$}}
\xLINE(1.5,6.5)(3.2,6.6)

\Put(0,7){\circle*{0.08}}
\Put(-0.05,7.1){\tiny{1}}
\changereferencesystem(0,7)(1,0)(0,1)
\segmento{1}{0}{\ang}{\rad}
\Put(-0.1,-0.05){\tiny{3}}
\segmento{1}{0}{\ang}{\rad}
\Put(-0.1,-0.05){\tiny{2}}
\segmento{0}{-1}{\ang}{\rad}
\segmento{1}{0}{\ang}{\rad}
\Put(-0.1,-0.05){\tiny{4}}
\segmento{0}{1}{\ang}{\rad}
\Put(-0.1,-0.05){\tiny{5}}
\segmento{0}{1}{\ang}{\rad}
\Put(-0.1,-0.05){\tiny{6}}
\segmento{0}{1}{\ang}{\rad}
\Put(-0.1,-0.05){\tiny{7}}
\segmento{0}{1}{\ang}{\rad}
\Put(-0.1,-0.05){\tiny{8}}
\segmentosinpunto{0}{1}{\angbisbis}{\rad}
\Put(0.01,0.1){$\ldots$}
\changereferencesystem(0,0.5)(1,0)(0,1)
\segmentosinpunto{0}{1}{\angbisbis}{\rad}
\Put(0,0){\circle*{0.08}}
\Put(-0.1,-0.05){\tiny{12}}
\segmento{1}{0}{\ang}{\rad}
\Put(-0.1,-0.05){\tiny{11}}
\segmento{0}{1}{\ang}{\rad}
\Put(-0.1,-0.05){\tiny{10}}
\segmento{0}{1}{\ang}{\rad}
\Put(-0.1,-0.05){\tiny{9}}

\end{Picture}
\end{center}

\caption{Graph $\Gamma_{\nu}^*$ in Example 2}
  \label{figura13}
\end{figure}

We have $\nu_r (f) =135$ and $135 > 10 \sqrt{180}$,  hence $C$ is supraminimal. In Figure \ref{figura13}, we depict the graph $\Gamma_{\nu}^*$ of that valuation. In addition
\[
\frac{\nu(f)}{\deg(f)} = \frac{(135,33)}{10},
\]
then by Theorems \ref{53} and \ref{triangulo}, the Newton-Okounkov body $\Delta_\nu$ is a triangle with vertices $(0,0)$,
\[
\frac{10}{135} \left( 180, 45 \right)  \; \; \mbox{and} \; \; \left(\frac{135}{10}, \frac{33}{10} \right).
\]
{\bf Example 3.} Let $\nu_r$ be a divisorial valuation with sequence of maximal contact values $ \{\betabarra_j (\nu_r) \}_{j=0}^2 = \{ 48, 329, 15792 \} $. The points $p_i$, $1 \leq i \leq 7$, of the configuration $C_{\nu_r} = \{p_i\}_{i=1}^{19}$ are free and there exists a nodal cubic $C$ (defined by certain equation $f=0$) passing through these points with multiplicities $2, 1, 1, 1, 1, 1, 1$. Thus $\nu_r (f) = 48 \cdot 7 + 41 = 377$. Then
\[
\frac{\nu_r(f)}{3}  > \sqrt{\betabarra_2 (\nu_r)},
\]
and $C$ is supraminimal.

Let now  $\nu = \nu_{E_\bullet}$ be the flag valuation defined by $E_\bullet = \left\{ X=X_{19} \supset E_{19}\supset \{q=p_{20}\}\right\}$, where $q \in E_{13} \cap E_{19}$. It holds that
\[
\frac{\nu(f)}{\deg(f)} = \frac{(377,55)}{3} = Q_3,
\]
and applying Theorem \ref{53}, we get that the Newton-Okounkov body $\Delta_\nu$ is a quadrilateral with vertices $(0,0)$,
\[
Q_1= \left(\frac{15792 \cdot 3}{377} , \frac{2303 \cdot 3}{377} \right),  \; \; Q_2= \left(\frac{15792 \cdot 3}{377} , \frac{2304 \cdot 3}{377} \right) \; \; \mbox{and} \; \; Q_3.
\]

We conclude this example (and the paper) showing in Figure \ref{figura12} the dual graph $\Gamma^*_\nu$, which is connected as we state in Theorem \ref{triangulo}.

\begin{figure}[h]
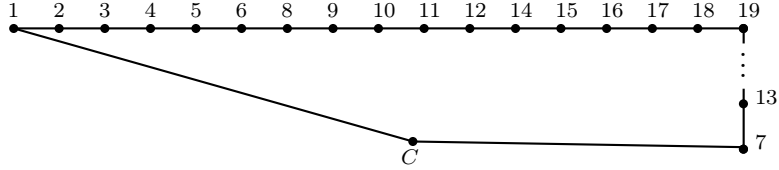

\begin{center}

\newcommand{\norma}[2]{%
\SQUARE{#1}{\m}
\SQUARE{#2}{\n}
\ADD{\m}{\n}{\k}
\SQUAREROOT{\k}{\norm}
}


\newcommand{\arco}[6]{%
\norma{#1}{#2}
\DIVIDE{#2}{\norm}{\tangentey}
\DIVIDE{#1}{\norm}{\tangentex}
\COPY{\tangentey}{\ortogonalx}
\COPY{\tangentey}{\centrox}
\SUBTRACT{0}{\tangentex}{\ortogonaly}
\MULTIPLY{\centrox}{#6}{\centrox}
\MULTIPLY{\ortogonaly}{#6}{\centroy}
\ADD{#3}{\centrox}{\centrox}
\ADD{#4}{\centroy}{\centroy}
\changereferencesystem(\centrox,\centroy)(\ortogonalx,\ortogonaly)(\tangentex,\tangentey)
\radiansangles
\SUBTRACT{\numberPI}{#5}{\angulo}
\ellipticArc{#6}{#6}{\angulo}{\numberPI}
\COS{\angulo}{\puntox}
\SIN{\angulo}{\puntoy}
\MULTIPLY{#6}{\puntox}{\puntox}
\MULTIPLY{#6}{\puntoy}{\puntoy}
\Put(\puntox,\puntoy){\circle*{0.08}}
\SUBTRACT{0}{\puntox}{\direcciony}
\COPY{\puntoy}{\direccionx}
\norma{\direccionx}{\direcciony}
\DIVIDE{\direccionx}{\norm}{\direccionx}
\DIVIDE{\direcciony}{\norm}{\direcciony}
\COPY{\direcciony}{\ortogonalx}
\COPY{\direccionx}{\ortogonaly}
\SUBTRACT{0}{\ortogonaly}{\ortogonaly}
\changereferencesystem(\puntox,\puntoy)(\ortogonalx,\ortogonaly)(\direccionx,\direcciony)
}

\newcommand{\arcosinpunto}[6]{%
\norma{#1}{#2}
\DIVIDE{#2}{\norm}{\tangentey}
\DIVIDE{#1}{\norm}{\tangentex}
\COPY{\tangentey}{\ortogonalx}
\COPY{\tangentey}{\centrox}
\SUBTRACT{0}{\tangentex}{\ortogonaly}
\MULTIPLY{\centrox}{#6}{\centrox}
\MULTIPLY{\ortogonaly}{#6}{\centroy}
\ADD{#3}{\centrox}{\centrox}
\ADD{#4}{\centroy}{\centroy}
\changereferencesystem(\centrox,\centroy)(\ortogonalx,\ortogonaly)(\tangentex,\tangentey)
\radiansangles
\SUBTRACT{\numberPI}{#5}{\angulo}
\ellipticArc{#6}{#6}{\angulo}{\numberPI}
\COS{\angulo}{\puntox}
\SIN{\angulo}{\puntoy}
\MULTIPLY{#6}{\puntox}{\puntox}
\MULTIPLY{#6}{\puntoy}{\puntoy}
\SUBTRACT{0}{\puntox}{\direcciony}
\COPY{\puntoy}{\direccionx}
\norma{\direccionx}{\direcciony}
\DIVIDE{\direccionx}{\norm}{\direccionx}
\DIVIDE{\direcciony}{\norm}{\direcciony}
\COPY{\direcciony}{\ortogonalx}
\COPY{\direccionx}{\ortogonaly}
\SUBTRACT{0}{\ortogonaly}{\ortogonaly}
\changereferencesystem(\puntox,\puntoy)(\ortogonalx,\ortogonaly)(\direccionx,\direcciony)
}

\newcommand{\arcoblanco}[6]{%
\norma{#1}{#2}
\DIVIDE{#2}{\norm}{\tangentey}
\DIVIDE{#1}{\norm}{\tangentex}
\COPY{\tangentey}{\ortogonalx}
\COPY{\tangentey}{\centrox}
\SUBTRACT{0}{\tangentex}{\ortogonaly}
\MULTIPLY{\centrox}{#6}{\centrox}
\MULTIPLY{\ortogonaly}{#6}{\centroy}
\ADD{#3}{\centrox}{\centrox}
\ADD{#4}{\centroy}{\centroy}
\changereferencesystem(\centrox,\centroy)(\ortogonalx,\ortogonaly)(\tangentex,\tangentey)
\radiansangles
\SUBTRACT{\numberPI}{#5}{\angulo}
\COS{\angulo}{\puntox}
\SIN{\angulo}{\puntoy}
\MULTIPLY{#6}{\puntox}{\puntox}
\MULTIPLY{#6}{\puntoy}{\puntoy}
\Put(\puntox,\puntoy){\circle*{0.04}}
\SUBTRACT{0}{\puntox}{\direcciony}
\COPY{\puntoy}{\direccionx}
\norma{\direccionx}{\direcciony}
\DIVIDE{\direccionx}{\norm}{\direccionx}
\DIVIDE{\direcciony}{\norm}{\direcciony}
\COPY{\direcciony}{\ortogonalx}
\COPY{\direccionx}{\ortogonaly}
\SUBTRACT{0}{\ortogonaly}{\ortogonaly}
\changereferencesystem(\puntox,\puntoy)(\ortogonalx,\ortogonaly)(\direccionx,\direcciony)
}

\newcommand{\arcoblancosinpunto}[6]{%
\norma{#1}{#2}
\DIVIDE{#2}{\norm}{\tangentey}
\DIVIDE{#1}{\norm}{\tangentex}
\COPY{\tangentey}{\ortogonalx}
\COPY{\tangentey}{\centrox}
\SUBTRACT{0}{\tangentex}{\ortogonaly}
\MULTIPLY{\centrox}{#6}{\centrox}
\MULTIPLY{\ortogonaly}{#6}{\centroy}
\ADD{#3}{\centrox}{\centrox}
\ADD{#4}{\centroy}{\centroy}
\changereferencesystem(\centrox,\centroy)(\ortogonalx,\ortogonaly)(\tangentex,\tangentey)
\radiansangles
\SUBTRACT{\numberPI}{#5}{\angulo}
\COS{\angulo}{\puntox}
\SIN{\angulo}{\puntoy}
\MULTIPLY{#6}{\puntox}{\puntox}
\MULTIPLY{#6}{\puntoy}{\puntoy}
\SUBTRACT{0}{\puntox}{\direcciony}
\COPY{\puntoy}{\direccionx}
\norma{\direccionx}{\direcciony}
\DIVIDE{\direccionx}{\norm}{\direccionx}
\DIVIDE{\direcciony}{\norm}{\direcciony}
\COPY{\direcciony}{\ortogonalx}
\COPY{\direccionx}{\ortogonaly}
\SUBTRACT{0}{\ortogonaly}{\ortogonaly}
\changereferencesystem(\puntox,\puntoy)(\ortogonalx,\ortogonaly)(\direccionx,\direcciony)
}

\newcommand{\segmento}[4]{%
\MULTIPLY{#3}{#4}{\long}
\MULTIPLY{#1}{\long}{\dirx}
\MULTIPLY{#2}{\long}{\diry}
\xLINE(0,0)(\dirx,\diry)
\Put(\dirx,\diry){\circle*{0.08}}
\changereferencesystem(\dirx,\diry)(#2,-#1)(#1,#2)
}

\newcommand{\segmentosinpunto}[4]{%
\MULTIPLY{#3}{#4}{\long}
\MULTIPLY{#1}{\long}{\dirx}
\MULTIPLY{#2}{\long}{\diry}
\xLINE(0,0)(\dirx,\diry)
\changereferencesystem(\dirx,\diry)(#2,-#1)(#1,#2)
}

\newcommand{\segmentoblanco}[4]{%
\MULTIPLY{#3}{#4}{\long}
\MULTIPLY{#1}{\long}{\dirx}
\MULTIPLY{#2}{\long}{\diry}
\Put(\dirx,\diry){\circle*{0.04}}
\changereferencesystem(\dirx,\diry)(#2,-#1)(#1,#2)
}

\newcommand{\segmentoblancosinpunto}[4]{%
\MULTIPLY{#3}{#4}{\long}
\MULTIPLY{#1}{\long}{\dirx}
\MULTIPLY{#2}{\long}{\diry}
\changereferencesystem(\dirx,\diry)(#2,-#1)(#1,#2)
}

\setlength{\unitlength}{1.5cm}%

\begin{Picture}(-1.2,6)(8,8)
\thicklines

\referencesystem(0,0)(1,0)(0,1)

\COPY{1}{\rad}
\COPY{0.4}{\ang}
\DIVIDE{\ang}{2}{\angbis}
\DIVIDE{\ang}{3}{\angbisbis}
\COPY{0.7}{\angtris}

\COPY{0}{\taux}
\COPY{1}{\tauy}
\COPY{0}{\pux}
\COPY{7}{\puy}

\Put(3.5,6){\circle*{0.08}}
\Put(3.4,5.8){\tiny{$C$}}
\xLINE(3.5,6)(0,7)
\xLINE(3.5,6)(6.4,5.95)

\Put(0,7){\circle*{0.08}}
\Put(-0.05,7.1){\tiny{1}}
\changereferencesystem(0,7)(1,0)(0,1)
\segmento{1}{0}{\ang}{\rad}
\Put(-0.1,-0.05){\tiny{2}}
\segmento{0}{1}{\ang}{\rad}
\Put(-0.1,-0.05){\tiny{3}}
\segmento{0}{1}{\ang}{\rad}
\Put(-0.1,-0.05){\tiny{4}}
\segmento{0}{1}{\ang}{\rad}
\Put(-0.1,-0.05){\tiny{5}}
\segmento{0}{1}{\ang}{\rad}
\Put(-0.1,-0.05){\tiny{6}}
\segmento{0}{1}{\ang}{\rad}
\Put(-0.1,-0.05){\tiny{8}}
\segmento{0}{1}{\ang}{\rad}
\Put(-0.1,-0.05){\tiny{9}}
\segmento{0}{1}{\ang}{\rad}
\Put(-0.1,-0.05){\tiny{10}}
\segmento{0}{1}{\ang}{\rad}
\Put(-0.1,-0.05){\tiny{11}}
\segmento{0}{1}{\ang}{\rad}
\Put(-0.1,-0.05){\tiny{12}}
\segmento{0}{1}{\ang}{\rad}
\Put(-0.1,-0.05){\tiny{14}}
\segmento{0}{1}{\ang}{\rad}
\Put(-0.1,-0.05){\tiny{15}}
\segmento{0}{1}{\ang}{\rad}
\Put(-0.1,-0.05){\tiny{16}}
\segmento{0}{1}{\ang}{\rad}
\Put(-0.1,-0.05){\tiny{17}}
\segmento{0}{1}{\ang}{\rad}
\Put(-0.1,-0.05){\tiny{18}}
\segmento{0}{1}{\ang}{\rad}
\Put(0,0){\circle*{0.08}}
\Put(-0.1,-0.05){\tiny{19}}
\segmentosinpunto{1}{0}{\angbisbis}{\rad}
\Put(0.04,0.3){$\vdots$}

\changereferencesystem(0,0.4)(1,0)(0,1)
\segmentosinpunto{0}{1}{\angbisbis}{\rad}
\Put(0,0){\circle*{0.08}}
\Put(-0.1,0){\tiny{13}}
\segmento{0}{1}{\ang}{\rad}
\Put(0,0){\circle*{0.08}}
\Put(-0.1,0){\tiny{7}}


\end{Picture}
\end{center}

\caption{Graph $\Gamma_{\nu}^*$ in Example 3}
  \label{figura12}
\end{figure}

\section*{Acknowledgements}

The authors wish to thank A. K\"uronya  and J. Ro\'e for stimulating their interest in Newton-Okounkov bodies and their helpful comments. Moreover, they would like to thank the CRM (Centre de Recerca Matem\`atica at Barcelona) for hosting the workshop 'Positivity and valuations' where this cooperation started.

\end{document}